\documentclass[reqno,a4paper,12pt]{amsart} 

\usepackage{amsmath,amscd,amsfonts,amssymb}
\usepackage{mathrsfs,dsfont}
\usepackage{color}
\usepackage{mathtools}
\usepackage{hyperref}
\usepackage{enumitem}
\usepackage{todonotes}
\usepackage[normalem]{ulem}
\usepackage{subcaption}
\usepackage{csquotes}
\usepackage{geometry}
\usepackage{enumitem}
\usepackage[breakable]{tcolorbox}
\usepackage{thmtools}
\usepackage{thm-restate}

\usepackage{tikz}
\usetikzlibrary{decorations.pathreplacing}

\numberwithin{equation}{section}
\numberwithin{figure}{section}

\def\R{\mathbb{R}}

\def\Q{\mathbb{Q}}
\def\Z{\mathbb{Z}}
\def\N{\mathbb{N}}

\def\T{\mathbb{T}}

\def\1{\mathds{1}}

\newcommand{\Tile}{\operatorname{Tile}}

\newcommand{\interior}{\operatorname{int}}

\renewcommand\le{\leqslant}
\renewcommand\ge{\geqslant}
\renewcommand\leq{\leqslant}
\renewcommand\geq{\geqslant}

\renewcommand\hat{\widehat}

\newcommand{\diam}{\operatorname{diam}}

\newcommand{\quake}[2]{\operatorname{Plates}_{\operatorname{nc}}(#2,#1)}
\newcommand{\quakeall}[2]{\operatorname{Plates}(#2,#1)}
\newcommand{\quakeone}[2]{\operatorname{Plates}_1(#2,#1)}
\newcommand{\orbit}[3]{\mathcal{O}_{#1}(#3;#2)}

\theoremstyle{plain}
\newtheorem{thm}{Theorem}[section]
\newtheorem{theorem}{Theorem}[section]

\newtheorem{lemma}[thm]{Lemma}
\newtheorem{corollary}[thm]{Corollary}

\newtheorem{proposition}[thm]{Proposition}

\newtheorem{claim}[thm]{Claim}
\newtheorem*{claim*}{Claim}

\theoremstyle{definition}
\newtheorem{definition}[thm]{Definition}
\newtheorem*{definition*}{Definition}
\newtheorem*{remark*}{Remark}
\newtheorem{remark}[thm]{Remark}

\newtheorem*{corollary*}{Corollary}

\begin{document}

\title{Translational tilings by axis-parallel polygonal sets}

\author[J. de Dios Pont]{Jaume de Dios Pont}
\address{NYU Center for Data Science, 60 5th Ave, New York, NY 10011}
\email{jdedios@nyu.edu}
     
\author[J. Greb\'ik]{Jan Greb\'ik}
\address{Computer Science Institute, Charles University, Malostranské nám. 25, 118 00 Praha~1, Czechia.}
\email{grebikj@gmail.com}
     
\author[R. Greenfeld]{Rachel Greenfeld}
\address{Department of Mathematics, Northwestern University, Evanston, IL 60208.}
\email{rgreenfeld@northwestern.edu}

\author[J. Madrid]{Jos\'e Madrid}
\address{Department of  Mathematics, Virginia Polytechnic Institute and State University,  225 Stanger Street, Blacksburg, VA 24061-1026, USA.}
\email{josemadrid@vt.edu}

    	\subjclass[03B25, 52C22, 52C23]{03B25, 52C22, 52C23}
    	\date{}
    	
    	\keywords{tiling, periodicity, decidability}

\begin{abstract}
    The periodic tiling conjecture asserts that if a bounded set $\Omega\subset \mathbb R^d$ of finite measure tiles $\mathbb R^d$ by translations, then it admits at least one \emph{fully periodic} tiling.
    This conjecture was proved to hold in $\mathbb R$ by Lagarias and Wang, and was recently disproved by Greenfeld and Tao in sufficiently high dimensions.
    Its discrete analogue in $\mathbb Z^2$ was proved by Bhattacharya, whereas the conjecture in $\mathbb R^2$ remains open.
    In this paper, we analyze the structure of tilings by \emph{axis-parallel polygonal sets}: bounded open sets in $\mathbb R^2$ whose boundary is a finite union of line segments parallel to the axes.
    We prove that every such tiling that is topologically minimal is weakly periodic. 
    Moreover, we show that if an axis-parallel polygonal set $\Omega$ tiles $\R^2$ then it admits a quasi-periodic tiling (i.e., a finite union of translated lattices).
    The proofs combine geometric decompositions of tilings with rational approximation and discrete tiling theory.
    As a consequence, we obtain that the translational tiling problem for axis-parallel polygonal sets with real algebraic vertex coordinates is decidable.
\end{abstract}

\maketitle

\section{Introduction}
We say that a bounded measurable set $\Omega\subset \R^d$ {\bf \emph{tiles $\R^d$ by translations}} if there exists a
countable set $T\subset \R^d$ such that the family $$\Omega+t,\quad t\in T$$
covers almost every point in $\R^d$ exactly once.
We then say that {\bf \emph{$T$ is a tiling of $\R^d$ by $\Omega$}} and write $\Omega\oplus T=\R^d$.
We denote by $\Tile(\Omega;X)$ the set of all tilings of $X$ by $\Omega$, and we say that $\Omega$ is a {\bf \emph{(translational) tile}} if $\Tile(\Omega;\R^d)\not=\emptyset$.
A tiling $T$ is {\bf \emph{fully periodic}}, or simply {\bf \emph{periodic}}, if it is invariant under a full-rank lattice.
It is {\bf \emph{weakly periodic}} if it is a finite disjoint union of sets with nonzero periods, and {\bf \emph{quasi-periodic}} if it is a finite disjoint union of periodic sets.

The {\bf \emph{periodic tiling conjecture}} ({\bf \emph{PTC}}) asserts that every translational tile admits a periodic tiling \cite{S,GS,LW}.
It holds in $\R$~\cite{LW,KolountzakisLagarias} but fails in sufficiently high dimensions~\cite{GreenfeldTaoCounterexample}\footnote{A recent preprint of OpenAI announces a counterexample to PTC in $\Z^3$; see \cite{OpenAITiling3D}.}, even for connected open tiles \cite{GK}. Subsequently, the undecidability of translational tilings was obtained  \cite{GTundecidability}. See also \cite{G} for an exposition of the ideas motivating these constructions, as well as the references therein.
In $\R^2$, PTC remains open, although it is known for topological disks~\cite{GiraultBeauquierNivat,Kenyon,err} and rational polygonal sets~\cite{rational}.
The discrete analogue of PTC in $\Z^2$ was proved by Bhattacharya~\cite{Bhattacharya}.
Subsequently, Greenfeld and Tao~\cite{GreenfeldTaoStructureZd} showed that every tiling of $\Z^2$ by a finite set is
weakly periodic.
The structure theory developed in \cite{GreenfeldTaoStructureZd} and its refinements are key ingredients in our work.

In this paper, we study the structure of tilings of $\R^2$
by {\bf \emph{axis-parallel polygonal sets}}, i.e.,
  a bounded regular open set whose boundary consists of finitely many line segments, each parallel to one of the coordinate axes.
See Figure~\ref{fig:tiling}.

\begin{figure}[ht]
    \centering
    \includegraphics[width=0.8\linewidth]{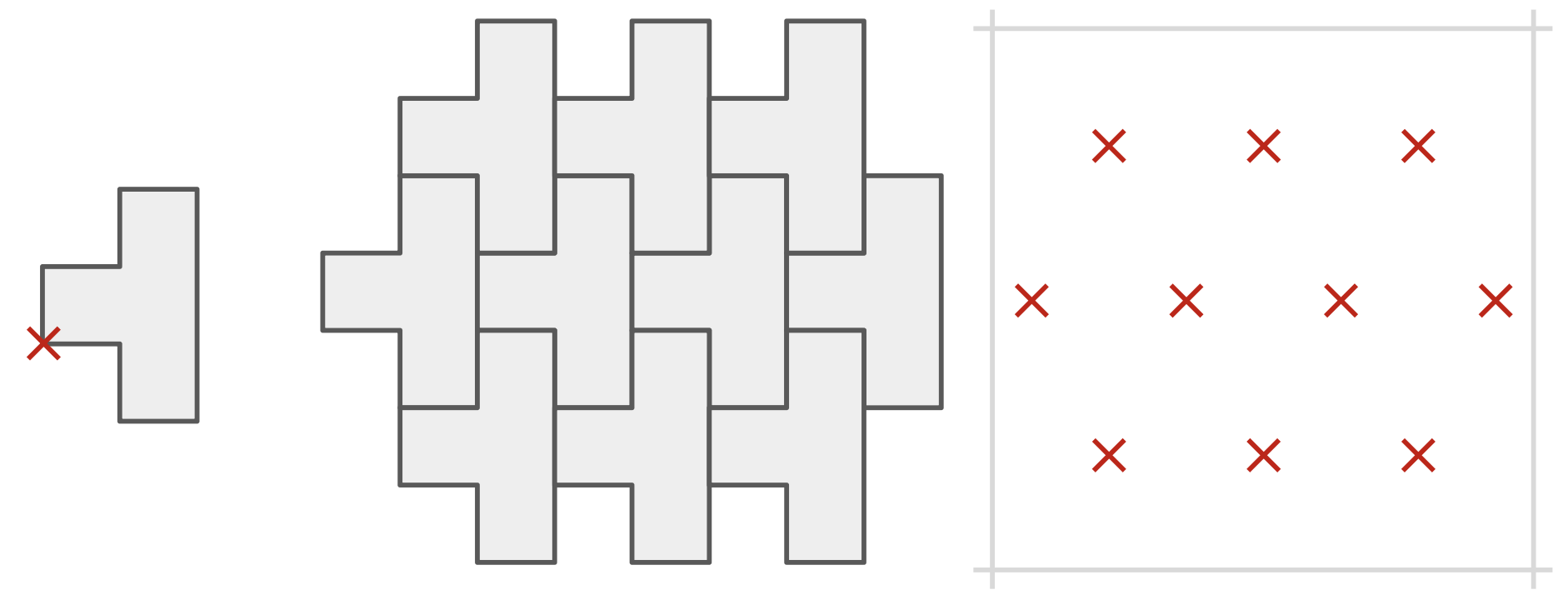}
    \caption{On the left, an axis-parallel tile $\Omega\subseteq \R^2$, on the right, a translation set $T\subseteq \R^2$ and in the middle, the obtained set $\Omega\oplus T$.}
    \label{fig:tiling}
\end{figure}

\subsection{Main results}
A key difference between discrete tilings of $\Z^2$ and their continuous counterparts is that every tiling of $\Z^2$ is locally simple.
This phenomenon is captured in the following definition.

\begin{definition}
    A discrete set $A\subseteq  \mathbb{R}^d$ has {\bf \emph{finite local complexity}} if for every $R>0$ the collection
    $$\{((\mathcal{B}_R+a)\cap A)-a:a\in A\},$$
    where $\mathcal{B}_R$ is the open ball of radius $R$ around the origin, is finite.    
\end{definition}

Observe that every periodic tiling has finite local complexity.
On the other hand, it is straightforward to produce tilings of $\R^2$ with the unit square $(0,1)^2$ that do not have finite local complexity.
A more complicated example of a quasi-periodic tiling without finite local complexity was found by LLMs~\cite[Section~3.2]{Bolzano}.
Our first main result, that is a consequence of more general Theorem~\ref{thm:MainLocallyFin}, and its proof extend the Greenfeld--Tao structure theory \cite{GreenfeldTaoStructureZd} for tilings of $\Z^2$ to tilings of $\R^2$ that have finite local complexity.

\begin{theorem}[Periodicity of tilings of finite local complexity]\label{thm:IntroMainLocallyFin}
Let $\Omega$ be an axis-parallel polygonal set. If  $T\in \Tile(\Omega,\R^2)$ has finite local complexity 
then $T$ is weakly periodic, and there is $S\in \Tile(\Omega;\R^2)$ that is periodic.
\end{theorem}

To prove Theorem~\ref{thm:IntroMainLocallyFin} and the results below, we develop a method for approximating tilings by axis-parallel polygonal sets with discrete tilings of $\Z^2$, see Section~\ref{sec:RationalApprox}.
This is done by encoding the relevant data in a system of linear equations over $\Q$.
The difficulty lies in the fact that the quantitative bounds in~\cite{GreenfeldTaoStructureZd} depend on the size of the
discrete tile, which grows as the approximations become finer.
We overcome this difficulty by showing that sufficiently faithful local approximations preserve the relevant global
structure of the tiling.
This step relies crucially on finite local complexity and is developed in Sections~\ref{sec:GoodApprox}, \ref{sec:FinerStructure} and~\ref{sec:PeriodicityFLC}.

The main novelty of our work is the treatment of tilings without finite local complexity.
Our starting point is the following observation, originating in the work of Kenyon~\cite{Kenyon}.
If a tiling $T$ of $\R^2$ by an axis-parallel polygonal set $\Omega$ does not have finite local complexity, then some weak limit $S$ of translates of $T$ admits a decomposition $S=S_0\sqcup S_1$ into nonempty sets such that the regions $\Omega\oplus S_0$ and $\Omega\oplus S_1$ are both either vertically or horizontally invariant, see Section~\ref{sec:EarthquakeDecomposition}.
Hence, their common boundary consists of parallel straight lines, and the pieces $S_0$ and $S_1$ can be slid independently in the direction of these lines without changing the regions they cover.

Building on this viewpoint, we introduce the {\bf \emph{earthquake decomposition of a tiling $T$ with respect to $\Omega$}} in Section~\ref{sec:earthquakesperiod}.
For a fixed coordinate direction, this decomposition consists of minimal collections of translations whose covered regions are invariant in that direction.
For clarity, we focus on the vertical direction only in this introduction.
In order to state are results precisely, we need to distinguish special pieces in the decomposition.

\begin{definition}
    We say that $S\subseteq\R^2$ is a {\bf \emph{vertical column for $\Omega$}} if there is
    $a\in\R$ such that
    $$\Omega\oplus S=(p_1(\Omega)+a)\times\R,$$
    where $p_1$ denotes projection onto the first coordinate.
    For $T\in\Tile(\Omega;\R^2)$, we denote by $T_c$ the union of all subsets of $T$ that are vertical columns.
\end{definition}

The pieces of the earthquake decomposition contained in $T\setminus T_c$ are called {\bf \emph{tectonic plates}}, and their collection is denoted by $\quake{T}{\Omega}$.
Much of the paper is devoted to showing that under a mild dynamical assumption on $T$ we have that $\Omega\oplus P$ is periodic for each $P\in \quake{T}{\Omega}$.
This in particular implies that $\Omega\oplus T_c$ is periodic as well.

To make the dynamical assumption precise, recall that one can endow the set of all tilings with the weak topology and dynamics coming from the translations by $\R^2$.
This turns $\Tile(\Omega,\R^2)$ into a compact dynamical system.
A tiling is {\bf \emph{topologically minimal}} if its orbit closure contains no proper nonempty closed invariant subset.
Birkhoff's theorem guarantees that every nonempty orbit closure contains such a tiling.
Our main structural result connects this dynamical minimality with earthquake decompositions and finite local complexity, see Section~\ref{sec:MinimalTiling}.

\begin{restatable}{theorem}{mainEarthquaketheorem}
\label{thm:IntroMainTectonicPlatesFinLocComp}
    Let $\Omega$ be an axis-parallel polygonal set.
    If $T\in \Tile(\Omega,\R^2)$ is topologically minimal, then, after possibly swapping the vertical and horizontal axes, there is a finite decomposition
    $$T=T_c\sqcup T_1\sqcup\dots\sqcup T_m,$$ 
    with $m\ge 0$, such that for every $1\le j\le m$, the set $T_j$ has finite local complexity and $\Omega\oplus T_j$ is periodic and $\R(0,1)$-invariant.
\end{restatable}

To prove Theorem~\ref{thm:IntroMainTectonicPlatesFinLocComp}, we first establish new periodicity results for tectonic plates of discrete tilings of $\Z^2$ in Section~\ref{sec:discrete}, building on the structure theory of Greenfeld and Tao~\cite{GreenfeldTaoStructureZd}.
We then combine these results with our approximation theory for tilings by axis-parallel polygonal sets.
A substantial part of the technical work concerns establishing the periodicity of $\Omega\oplus T_c$; this is addressed in Sections~\ref{sec:earthquakesperiod} and \ref{sec:finhorcomp}.

Combining Theorems~\ref{thm:IntroMainLocallyFin} and~\ref{thm:IntroMainTectonicPlatesFinLocComp} yields our main periodicity result, proved in Section~\ref{sec:app}.

\begin{restatable}{theorem}{mainQuasitheorem}
\label{thm:IntroQuasiPTC}
Let $\Omega$ be an axis-parallel polygonal set.
If $T\in \Tile(\Omega,\R^2)$ is topologically minimal, then $T$ is weakly periodic, and there is a quasi-periodic tiling $S\in\Tile(\Omega;\R^2)$ that is $\langle\vec v\rangle$-periodic for some $\vec v\in\R^2\setminus\{(0,0)\}$.
\end{restatable}

Theorem~\ref{thm:IntroQuasiPTC} is our strongest conclusion towards PTC.
By Theorem~\ref{thm:IntroMainLocallyFin}, resolving PTC for axis-parallel polygonal sets is equivalent to showing that every tile in this class admits a tiling with finite local complexity.
This seems to be beyond reach of our techniques as they control the periodicity of the pieces separately.
In fact, the tectonic plates might be periodic but the intersection of the corresponding lattices might not be a lattice as the example in \cite[Section~3.2]{Bolzano} demonstrates.

Nevertheless, our results imply decidability of the translational tiling problem for tiles with real algebraic
vertex coordinates.
The precise formulation is given in Section~\ref{subsec:Decidability}.

\begin{restatable}{theorem}{mainDecidabilitytheorem}
\label{thm:decidability}
	The problem whether an axis-parallel polygonal set $\Omega$ with real algebraic vertex coordinates tiles $\R^2$ is decidable.
\end{restatable}

\begin{remark}[Generalizations of our argument]
A natural next step is to consider polygonal sets with arbitrary edge directions.
Our main results, except for the existence of a nonzero period vector $\vec{v}$ in Theorem~\ref{thm:IntroQuasiPTC}, extend to polygonal sets whose edges have {\bf \emph{rational slope}}, meaning that each edge is parallel to a nonzero vector in $\Z^2$.
These extensions require additional technical arguments but no new ideas, so we restrict the exposition to axis-parallel polygonal sets for simplicity.\\
Since our rational approximation does not necessarily preserve the combinatorial structure of the tiling when the slopes are not rational, extending our results to arbitrary polygonal sets and, more generally, to bounded measurable subsets of $\R^2$, remains open. 
\end{remark}

\begin{remark}[Lean formalization]
The results of this paper have been formalized in Lean 4.
The formalization is publicly available \url{https://github.com/jaumededios/Tiling}.
\end{remark}

\section*{Acknowledgments:}
JG was supported by an ERC CZ grant LL2608 funded by the Czech Ministry of Education, Youth and Sports from the ERC CZ Program.
RG was supported by NSF CAREER Award DMS-2441769 and by the Alfred P. Sloan Research Fellowship.

\section*{Declaration of AI usage}
The ideas and proofs are entirely the work of the authors.
An LLM model was used for proofreading of the finished manuscript. A couple of remarks regarding alternative arguments suggested by the model in the proofreading process were added in Sections \ref{subsec:Bands} and \ref{subsec:Decidability}, and the presentation in Section \ref{subsec:Decidability} was improved accordingly.

\section{Definitions and notation}\label{sec:Definitions}
We collect definitions and notation that is used throughout the text.

We have $\N=\{0,1,2,\dots\}$.
We equip $\mathbb{R}^d$, mainly for $d\in \{1,2\}$, with the Euclidean distance that we denote as $\|x-y \|$ for $x,y\in \mathbb{R}^d$.
For distinct points $a,b\in \R^d$, we denote as $\omega_{a,b}$ or $P_{a,b}$ the line segment that connects them.
We denote the length of a line segment $\omega$ as $|\omega|$, that is, if $\omega=\omega_{a,b}$, then we have $|\omega|=\|b-a\|$.
For $r\in \mathbb{R}$, write $\mathcal{B}_r\subseteq \R^d$ for the open ball of radius $r$ around the origin $(0,\dots,0)$.
For a measurable $S\subseteq \R^d$ of dimension $n\leq d$, we denote by $m_n(S)$ the $n$-dimensional Lebesgue measure of $S$.
This is mainly used when $d=2$ , $n=1$ and $S$ is a finite union of line segments.
When it is clear from the context, we use the standard set inclusion (or equality) for inclusion (or equality) up to null sets.
For a set $X\subseteq  \mathbb{R}^d$, we denote as $\overline{X}$ the closure of $X$, as $\interior(X)$ the interior of $X$ and as $\partial (X)$ the boundary of $X$.
That is, $\partial(X)=\overline{X}\setminus \interior(X)$.
An open set $X\subseteq\mathbb{R}^d$ is called regular open if $\interior(\overline{X})=X$.

{\bf Polygonal sets.}
A \emph{\bf \emph{polygonal set}} is a bounded regular open set $\Omega\subseteq  \mathbb{R}^2$, whose boundary is a finite union of segments.
Note that a polygonal set can certainly be disconnected and its connected components are not necessarily simply connected.
We denote the vertices of $\Omega$ by $V(\Omega)$, and the edges of $\Omega$ by $E(\Omega)$.

{\bf axis-parallel polygonal sets.}
We say that a polygonal set $\Omega$ is {\bf \emph{axis-parallel}} if every line segment $\omega$ in $E(\Omega)$ lies in a coset of either the $x$-axis, or the $y$-axis; i.e., there exists $s=s(\omega)\in \R^2$ such that $\omega \subseteq (\R\times \{0\}) \cup (\{0\}\times \R) +s$.

{\bf Periodic sets.}
For $S\subseteq \mathbb{R}^d$, or $S\subseteq \mathbb{Z}^d$, we denote as $\langle S\rangle$ the subgroup generated by $S$.
A subgroup $\Gamma\subseteq \mathbb{R}^d$, or $\Gamma\subseteq \mathbb{Z}^d$, is a {\bf \emph{lattice}} if it is discrete and admits a finite volume fundamental domain.

For a closed set $Z\subseteq  \mathbb{R}^d$, we set $\Gamma(Z)$ to be the group of translational symmetries of $Z\subseteq  \R^d$, that is,
$$\Gamma(Z)=\{g\in \mathbb{R}^d:Z+g=Z\}.$$
Observe that $\Gamma(Z)$ is closed subgroup.
We say that $Z\subseteq  \mathbb{R}^d$ is {\bf \emph{periodic}} if $\Gamma(Z)$ contains a lattice.

{\bf Tilings.}
We say that {\bf \emph{$T$ is a tiling of $X\subseteq  \mathbb{R}^d$ by $\Omega$}} if $X=\Omega\oplus T$ up to a null set.
In this paper, we will work mainly in the case when $\Omega$ is an axis-parallel polygonal set and $X$ is equal to $\mathbb{R}^2$ or $E\times \mathbb{R}$, where $E\subseteq  \mathbb{R}$ is a closed and periodic subset of $\mathbb{R}$.
We denote as $\Tile(\Omega;X)$ the set of all tilings of $X$ by $\Omega$.

{\bf Packing.}
Let $\Omega\subseteq \R^d$ be measurable, and $S\subseteq \R^d$. 
We say that $\Omega+S$ is a {\bf \emph{packing}} if the measure $m_d((\Omega+s)\cap (\Omega+s'))$ is zero for every $s\neq s'\in S$.
We abbreviate the fact that $\Omega+S$ is a packing by using the direct sum notation $\Omega\oplus S$.
In particular, $\Omega\oplus S$ does not necessarily mean that $\Omega\oplus S=\R^d$.

{\bf Periodicity of tilings.}
Let $T$ be a tiling of $\mathbb{R}^d$ by $\Omega$.
We say that $T$ is {\bf \emph{periodic}} if $\Gamma(T)$ is a lattice, {\bf \emph{quasi-periodic}} if there is a decomposition $T=T_1\sqcup \dots \sqcup T_m$ such that $\Gamma(T_i)$ is a lattice for every $1\le i\le m$, and {\bf \emph{weakly periodic}} if there is a decomposition $T=T_1\sqcup \dots \sqcup T_k$ such that $\Gamma(T_i)$ is non-trivial for every $1\le i\le k$.

\begin{remark}
    In the discrete case, that is tilings of $\mathbb{Z}^d$, see Section~\ref{sec:discrete}, the notions of periodic and quasi-periodic tilings coincide.
    It is easy to see that this is not the case even for tilings of $\mathbb{R}^2$ by a unit square.
\end{remark}

Next, we only consider the case of $\mathbb{R}^2$.
For $i\in \{1,2\}$, let $p_i\colon \R^2\to \R$ be the {\bf \emph{orthogonal projection}}, defined by $p_i(u)=u_i$ for every $u=(u_1,u_2)\in \R^2$.
For a set $\omega \subseteq \R^2$, which is contained in a vertical line, i.e.,  $\omega \subseteq \{x\}\times \R$ for some $x\in \R$, we identify the singleton $p_1(\omega)=\{x\}$ with its value $x$. Similarly, when $\omega\subseteq \R^2$ is contained in a horizontal line, i.e., $\omega \subseteq \R\times \{y\}$ for some $y\in \R$, we identify the singleton $p_2(\omega)=\{y\}$ with its value $y$.
For $a\in \mathbb{R}^2$ and $A\subseteq  \mathbb{R}^2$ we write $p_1(a)\le p_1(A)$ as a shorthand for $p_1(a)\le p_1(x)$ for every $x\in A$.
We define $p_1(A)\le p_1(a)$, $p_1(a)<p_1(A)$ and $p_1(A)<p_1(a)$ analogously.

{\bf Edges and proper intersection}
Let $\Omega$ be an axis-parallel polygonal set and $\omega,\omega'\in E(\Omega)$ be horizontal and $t,t'\in \R^2$. We say that $\omega+t$ and $\omega'+t'$ {\bf \emph{properly intersect}} if $p_2(\omega+t)=p_2(\omega'+t')$
and $$m_1\left(p_1(\omega+t)\cap p_1(\omega'+t')\right)$$ is positive (see Figure \ref{fig:propaerint}). Note that if $a,a'$ is the left vertex of $\omega,\omega'$ (resp.) and $b,b'$ is the right vertex of $\omega,\omega'$ (resp.), then $\omega+t$ and $\omega'+t'$ properly intersect if and only if $p_2(a+t)=p_2(a'+t')$, $p_1(a+t)<p_1(b'+t')$ and $p_1(a'+t')<p_1(b+t)$.
Proper intersection for \emph{vertical} edges is defined similarly, by only swapping the roles of $p_1$ and $p_2$ in the above definition. 

\begin{figure}[ht]
    \centering
     \rotatebox{90}{\includegraphics[width=0.4\linewidth]{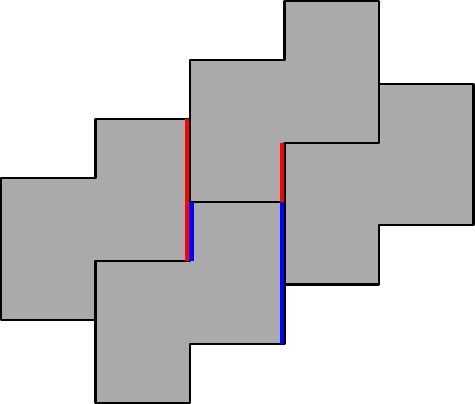}}
    \caption{The blue and red edges on the left properly intersect. The blue and red edges on the right intersect only at a point, and therefore do not properly intersect.}
    \label{fig:propaerint}
\end{figure}

{\bf Proper vertices}
We will use the notion of proper vertex from \cite[Section 2]{rational}: A point $x\in \partial \overline{D}$ is a {\bf \emph{proper vertex of $D\subseteq  \R^2$}} if for every $\epsilon>0$, the intersection  $(\mathcal{B}_\epsilon + x) \cap \partial \overline{D}$ is \emph{not} a line segment. We denote by $W(D)$ the set of proper vertices of $D$.

{\bf Columns}
We say that a bounded measurable set $\Omega$ of positive measure {\bf \emph{can tile a (vertical) column}} if there is $S\subseteq \mathbb{R}(0,1)$ such that $\Omega\oplus S=p_1(\Omega)\times \mathbb{R}$.
Additionally, we say that $S\subseteq \mathbb{R}^2$ is a {\bf \emph{$\mathbb{R}(0,1)$-column for $\Omega$}} if there is $c\in \mathbb{R}$ such tat $\Omega\oplus S=(p_1(\Omega)+c)\times \mathbb{R}$.
Note that $S$ is a $\mathbb{R}(0,1)$-column for $\Omega$ implies that $S$ is a subset of a coset of $\mathbb{R}(0,1)$.
The definition of horizontal column is analogous.

\section{Discrete earthquakes and periodicity}\label{sec:discrete}

\newcommand{\discsim}{\sim^{\text{disc}}_{F,\Uparrow}}
\newcommand{\lito}{\operatorname{Slabs}}

We derive a new periodicity result for discrete tilings of $\Z^2$ that will be used in Section~\ref{sec:earthquakesperiod} as the combinatorial core for establishing a result on the periodicity of planar tilings by axis-parallel polygonal tiles (see Theorem~\ref{thm:PeriodicityMain} bellow).
Our approach\footnote{This approach  was already used in \cite{rational} to show that  tilings of $\R^2$ by polygonal sets with rational vertices are  weakly periodic.} builds on a further development of  the structure theory introduced in \cite{GreenfeldTaoStructureZd}.

\subsection{Tilings of $\Z^2$}
Let $F\subseteq\mathbb{Z}^2$ be finite and $X\subseteq \mathbb{Z}^2$.
We say that $A$ is a tiling of $X$ by $F$, denoted as $F\oplus A=X$, if $F+a$, $a\in A$, partition $X$.
We denote by $\Tile(F;X)$ the set of all tilings of $X$ by $F$.
We are mainly interested in the case when $X=\mathbb{Z}^2$, and $X=E\times \mathbb{Z}$ for some periodic, or finite $E\subseteq \mathbb{Z}$.
See example in Figure \ref{fig:discrete-tiling}.

\begin{definition}[Discrete earthquake decomposition and discrete tectonic plates]\label{def:discplates}
    For $A\in \Tile(F;\Z^2)$, we define the {\bf \emph{(vertical) discrete earthquake decomposition}}\footnote{This notion is similar to  \cite[Definition~4.1]{rational} with $N=1$ and $\vec v=(0,1)$.} of $A$ (with respect to $F$) to be the decomposition
$$A=\bigsqcup_{P\in \quakeall{A}{F}} P$$
of $A$ into {\bf \emph{discrete tectonic plates}}, i.e., sets $P\subseteq A$ such that
\begin{itemize}
    \item [(i)] $F\oplus P$ is $\Z(0,1)$-periodic for every $P\in \quakeall{A}{F}$,
    \item [(ii)] if $P\in \quakeall{A}{F}$ and $P'\subseteq  P$ satisfies (i), then $P'\in \{\emptyset,P\}$.
\end{itemize}
\end{definition}

\begin{remark}
    The discrete earthquake decomposition,
    Definition~\ref{def:discplates}, can be defined in the same
    way for tilings of $E\times\mathbb{Z}$ with arbitrary
    $E\subseteq\mathbb{Z}$.
    When $E$ is periodic, \cite[Lemma~4.2]{rational},
    restated here as Proposition~\ref{pr:EarthquakeStructure},
    Theorem~\ref{thm:DiscretePeriodicity}, and the intermediate
    lemmas extend to this setting as well.
    This follows by applying the structure results of
    \cite{GreenfeldTaoStructureZd} for periodic sets and
    making the corresponding adjustments to the proofs in
    \cite{rational} and in the present section.
\end{remark}

\begin{figure}[ht]
    \centering
    \includegraphics[width=0.8\linewidth]{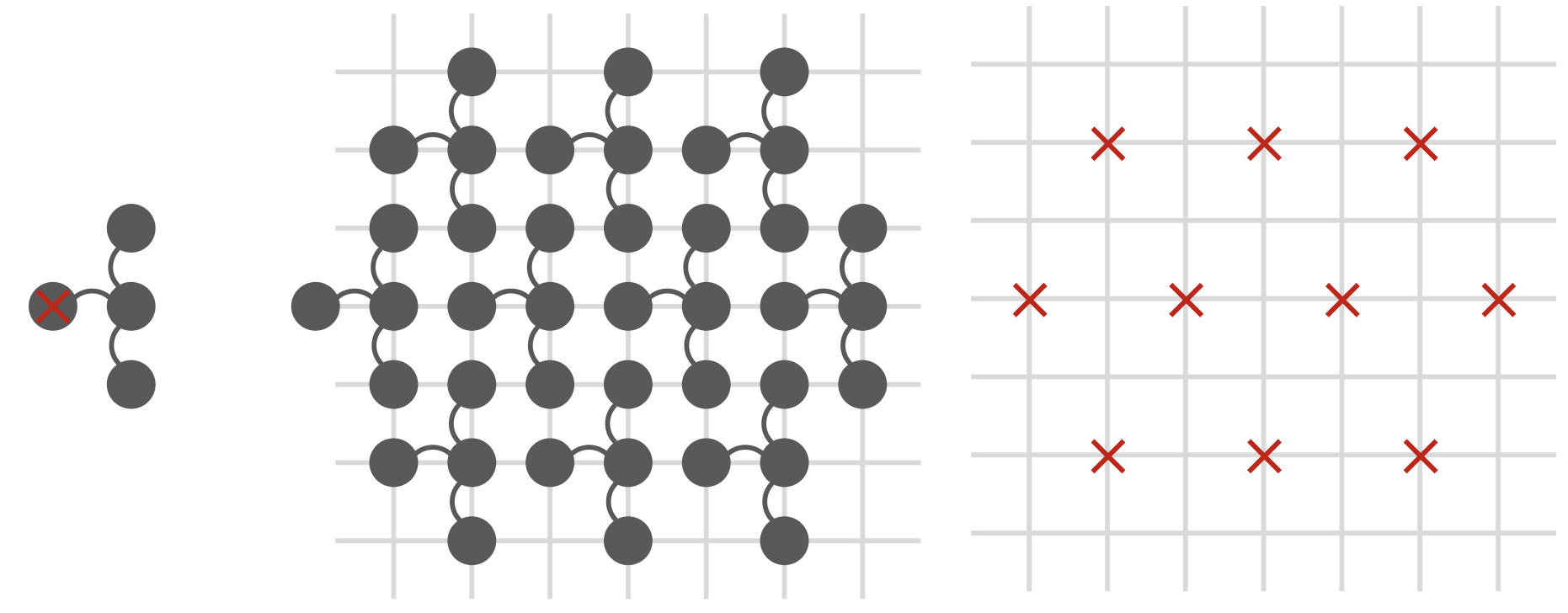}
    \caption{On the left is a discrete tile $F\subseteq \Z^2$, on the right (in red) is part of a tiling $A\in\Tile(F;\Z^2)$, and on the middle is $F\oplus A$.}
    \label{fig:discrete-tiling}
\end{figure}

\begin{figure}[ht]
    \centering
    \includegraphics[width=0.8\linewidth]{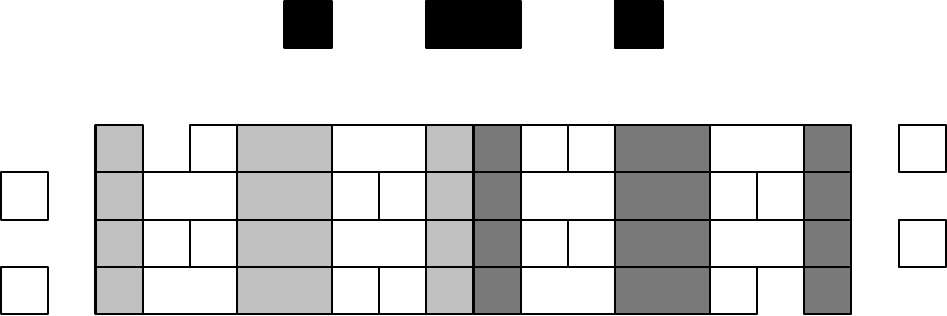}
    \caption{Above, in black, a discrete tile $F\subseteq \Z^2$, consisting of four points (one square represents one point in the lattice $\Z^2$). Below, a tiling $A$ by this tile. Each of the colors (dark gray, light gray, and white) represents a plate in the earthquake decomposition of $A$.}
    \label{fig:discrete-quake}
\end{figure}

In \cite{GreenfeldTaoStructureZd} it is proven that every tiling $A$ of $\Z^2$ by $F$ is weakly periodic.
The following result, which is proven in \cite{rational}, shows that discrete earthquake decomposition is compatible with the weak periodicity structure of $A$.

\begin{proposition}[Lemma~4.2 in \cite{rational}]\label{pr:EarthquakeStructure}
    Let $A$ be a tiling of $\mathbb{Z}^2$ by a finite set $F$.
    Then there are pairwise incommensurable vectors $h_1,\dots,h_m\in \mathbb{Z}^2\setminus \{(0,0)\}$ such that for every $P\in \quakeall{A}{F}$, we can write $P=\bigsqcup_{i=1}^m P^i$, where $P^i$ is $\langle h_i\rangle$-periodic for every $1\le i\le m$.
\end{proposition}

For our analysis, we need to distinguish between two types of tectonic plates.
We say that $P\in \quakeall{A}{F}$ is a {\bf \emph{discrete vertical column of $(F,A)$}} if there exists $k\in \Z$  such that $P \subseteq \{k\}\times \Z$ and we have $F+ P= X\times \Z$, where $X=p_1(F)+k$.
We denote by $A^{\text{disc}}_c$ the union of vertical columns of $A$.
We define $\quake{A}{F}$ to be the collection of those $P\in \quakeall{A}{F}$ that are \emph{not} discrete vertical columns. For example, in Figure \ref{fig:discrete-quake}, both gray components of $\quakeall{A}{F}$ correspond to discrete vertical columns (hence, contained in $A^{\text{disc}}_c$), while the white component is the single plate in $\quake{A}{F}$.

\subsection{Periodicity of discrete tectonic plates}
The goal of this section is showing that $F\oplus P$ is periodic for every $P\in \quake{A}{F}$.

\begin{theorem}[Periodicity of discrete tectonic plates]\label{thm:DiscretePeriodicity}
    Let $A$ be a tiling of $\mathbb{Z}^2$ by a finite set $F$.
    Then there is $\ell\in \mathbb{N}$ such that $\ell>0$ and the sets
    $$F\oplus A^{\operatorname{disc}}_{c} \text{ and } F\oplus P$$
    are $(\ell,0)\mathbb{Z}$-periodic for every $P\in \quake{A}{F}$. 
\end{theorem}

Using the notation of Proposition~\ref{pr:EarthquakeStructure}, define $$A^i=\bigsqcup_{P\in \quakeall{A}{F}} P^i.$$
Note that $A^i$ is $\langle h_i\rangle$-periodic for every $1\le i\le m$ and $A=\bigsqcup_{i=1}^m A^i$.

\begin{lemma}\label{pr:SetPeriodic}
    There is $N>0$ such that $F\oplus A^i$ is $N\mathbb{Z}^2$-periodic for every $1\le i\le m$.
\end{lemma}
\begin{proof}
    We may assume that $m>1$ as otherwise the claim trivially holds with $N=1$.
    Since $h_1,\dots,h_m$ are incommensurable, 
    by\footnote{We apply here only the last part of \cite[Theorem~5.4]{GreenfeldTaoStructureZd}.} \cite[Theorem~5.4]{GreenfeldTaoStructureZd}, for every $1\leq i\leq m$ we obtain the existence of $N_i>0$ such that $F\oplus A^i$ is $N_i\mathbb{Z}^2$-periodic.
    Then, by setting $N$ to be the least common multiple of $N_i$, $1\le i\le m$, we obtain the claim.
\end{proof}

We make the following convention:
Allowing $A^1=\emptyset$, we assume that $h_1=(0,k)$ for $k\in \mathbb{N}$ such that $k>0$.

\begin{lemma}\label{lem:PeriodicNonParallel}
    There is $M>0$ such that  $F\oplus P^i$ is $M\mathbb{Z}^2$-periodic for every $2\le i\le m$ and $P\in \quakeall{A}{F}$.
    Moreover, $F\oplus A^1$ is $M\mathbb{Z}^2$-periodic.
\end{lemma}
\begin{proof}
    If $m=1$, the claim holds with $M=1$.
    Hence we may assume that $m\ge2$.
    First, suppose that $2\le i\le m$ and let $P_0\in \quakeall{A}{F}$ be such that $P_0^i$ is nonempty. 
    By Lemma~\ref{pr:SetPeriodic}, we know that $F\oplus A^i$ is $N\mathbb{Z}^2$-periodic for some $N>0$.
    Thus, if $x\in F\oplus P^i_0$, then
     $y=N(0,1)+x\in F\oplus A^i$.
    As $F\oplus P_0$ is $\mathbb{Z}(0,1)$-periodic, it follows that $y\in F\oplus P_0$.
    So $$y\in (F\oplus A^i)\cap (F\oplus P_0),$$ which implies that  $y\in F\oplus P_0^i$, since  $F\oplus A^i=\bigsqcup_{P\in \quakeall{A}{F}}(F\oplus P^i)$ and
    $$\left( \bigsqcup_{P\in \quakeall{A}{F}\setminus \{P_0\}} F\oplus P^i\right)\cap (F\oplus P_0) =\emptyset.$$
    We obtain that $F\oplus P_0^i$ is $N(0,1)$-periodic. Moreover, by the $\langle h_i\rangle$-periodicity of $P_0^i$,  the set $F\oplus P_0^i$ is $\langle h_i\rangle$-periodic.  Since $h_i\not\in \mathbb{Z}(0,1)$, we conclude that  $F\oplus P_0^i$ is $M_i\mathbb{Z}^2$-periodic for $M_i\Z^2 \subseteq \langle N(0,1), h_i \rangle$.
    Then, setting $M$ to be the least common multiple of $M_i$, $2\le i\le m$, we obtain that $F\oplus P^i$ is $M\Z^2$-periodic for every $2\leq i\leq m$ and $P\in \quakeall{A}{F}$.
    
    Finally, since $$F\oplus A^1=\mathbb{Z}^2\setminus \left(F\oplus\bigsqcup_{\substack{P\in \quakeall{A}{F} \\ 2\le i\le m}} P^i\right),$$  using that the set $$F\oplus\bigsqcup_{\substack{P\in \quakeall{A}{F} \\ 2\le i\le m}} P^i$$  is $M\mathbb{Z}^2$-periodic, we obtain that $F\oplus A^1$ is $M\Z^2$-periodic as well.
\end{proof}

We introduce the notion of {\bf \emph{earthquake sliding}}, which will be useful in analyzing the structure of the tectonic plates of a tiling. 

\begin{definition}[Earthquake sliding]
    For $\psi:\quakeall{A}{F}\to \mathbb{Z}$, we define $P_{\psi}\coloneqq P+(0,\psi(P))$ for every $P\in \quakeall{A}{F}$ and $A_\psi\coloneqq \bigsqcup_{P\in \quakeall{A}{F}} P_\psi$. (See Figure \ref{fig:sliding}.)   
\end{definition}

\begin{figure}[ht]
    \centering
    \includegraphics[width=0.8\linewidth]{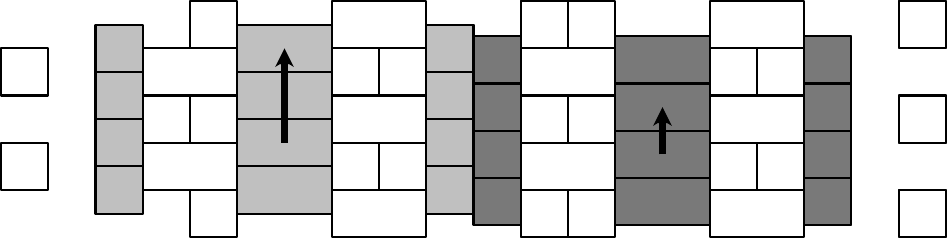}
    \caption{Sliding of the tiling in Figure \ref{fig:discrete-quake}. The function $\psi(P)$ quantifies the amount of \emph{sliding} in each plate.}
    \label{fig:sliding}
\end{figure}

By property (i) in Definition~\ref{def:discplates} of discrete tectonic plates, we have that $F\oplus P=F\oplus P_\psi$ for every $P\in \quakeall{A}{F}$, 
and thus $A_\psi\in \Tile(F;\Z^2)$ whenever $A\in \Tile(F;\Z^2)$, and one can show that 
$$\quakeall{A_\psi}{F}=\{P_\psi:P\in \quakeall{A}{F}\}.$$ In other words, the operation $\psi$ on the tiling $A\in \Tile(F;\Z^2)$ preserves the tiling nature of $A$ as well as its earthquake structure. (See also \cite[Theorem 10.1]{GT2} for a similar sliding feature.)

\begin{lemma}\label{lem:Shifts}
    Let $M>0$  be as in Lemma~\ref{lem:PeriodicNonParallel}, $\psi:\quakeall{A}{F}\to \mathbb{Z}$ and define $P^i_{\psi}=P^i+(0,\psi(P))$ for $1\le i\le m$ and $P\in \quakeall{A}{F}$. 
    Then 
    \begin{itemize}
        \item[(i)] $P_{\psi}=\bigsqcup_{1\le i\le m} P^i_{\psi}$;
        \item[(ii)] $P^i_{\psi}$ is $\langle h_i\rangle$-periodic for every $1\le i\le m$ and $P\in \quakeall{A}{F}$;
        \item[(iii)] $F\oplus P^i_{\psi}$ is $M\mathbb{Z}^2$-periodic for every $2\le i \le m$ and $P\in \quakeall{A}{F}$;
        \item[(iv)] $F\oplus A^1_\psi$ is $M\mathbb{Z}^2$-periodic, where $A^1_\psi=\bigsqcup_{P\in \quakeall{A}{F}} P^1_{\psi}$.
    \end{itemize} 
\end{lemma}

\begin{proof}
    Part (i) follows directly from the definitions as 
    $$P_\psi=P+(0,\psi(P))=\bigsqcup_{i=1}^m P^i+(0,\psi(P))=\bigsqcup_{i=1}^m P^i_\psi.$$
    For (ii), let $1\le i\le m$, $P\in \quakeall{A}{F}$ and $x\in P^i_\psi$.
    Then $$y=x-(0,\psi(P))\in P^i$$ by the definition of $P^i_{\psi}$.
    It follows that $$y+\ell h_i = x +\ell h_i -(0,\psi(P)) \in P^i$$ for every $\ell\in \Z$ as $P^i$ is $\langle h_i\rangle$-periodic.
    Therefore, by the definition of $P^{i}_{\psi}$
    we have  $x +\ell h_i \in P^{i}_{\psi}$.
   This shows that $P^i_{\psi}$ is $\langle h_i\rangle$-periodic.

    Using a similar argument, one can deduce the $M\Z^2$-periodicity of $F\oplus P_\psi^i$ from the $M\Z^2$-periodicity of $F\oplus P^i$, proven in Lemma~\ref{lem:PeriodicNonParallel} for every $2\leq i\leq m$ and $P\in \quakeall{A}{F}$, giving (iii).
    
    Part (iv) follows from (iii) and the fact that
    $$F\oplus A^1_{\psi}=\mathbb{Z}^2\setminus \left(F\oplus\bigsqcup_{\substack{P\in \quakeall{A}{F} \\ 2\le i\le m}} P^i_{\psi}\right)$$
    as $A_\psi\in \Tile(F;\Z^2)$.
\end{proof}

We use the slicing lemma \cite[Lemma~5.1]{GreenfeldTaoStructureZd} to derive the following version of \cite[Corollary~5.3]{GreenfeldTaoStructureZd}.

\begin{lemma}\label{lem:slices}
    There is $K>0$ such that for every $\psi:\quakeall{A}{F}\to \mathbb{Z}$ and $x,d\in \Z$ we have
    $$F\oplus \left(A^{1}_\psi\cap (\{x\}\times \Z)\right)+(dK,0)=F\oplus \left(A^{1}_\psi\cap (\{x+dK\}\times \mathbb{Z})\right).$$
\end{lemma}
\begin{proof}
We apply \cite[Lemma~5.1]{GreenfeldTaoStructureZd} with $h:=(0,1)$, $\ell:=M$, $k$ such that $(0,k)=h_1$, $A:=A^1_\psi$ and $F$.
By Lemma~\ref{lem:Shifts}, the assumptions of \cite[Lemma~5.1]{GreenfeldTaoStructureZd} are satisfied.
Then there are positive $q,s\in \mathbb{N}$, which do not depend on $\psi$, such that 
$$(F\cap (\{n\}\times \Z))\oplus A^1_\psi$$
is $qks\mathbb{Z}^2$-periodic for every $n\in \mathbb{Z}$.

Set $K:=qks$ and let $x,d\in \mathbb{Z}$.
We first show that 
\begin{equation}\label{eq:oneside}
    F\oplus \left(A^{1}_\psi\cap (\{x\}\times \mathbb{Z})\right)+(dK,0)\subseteq  F\oplus \left(A^{1}_\psi\cap (\{x+dK\}\times \mathbb{Z})\right).
\end{equation}
Let
\begin{equation}\label{eq:uin}
    u\in F\oplus \left(A^{1}_\psi\cap (\{x\}\times \mathbb{Z})\right).
\end{equation}
We need to prove that
\begin{equation}\label{eq:udkin}
    u+(dK,0)\in F\oplus \left(A^{1}_\psi\cap (\{x+dK\}\times \mathbb{Z})\right).
\end{equation}
By \eqref{eq:uin}, there are $a\in A^1_{\psi}\cap \{x\}\times \Z$ and $f=(f_1,f_2)\in F$ such that $a+f=u$.
By the $K\mathbb{Z}^2$-periodicity of $$(F\cap (\{f_1\}\times \Z))\oplus A^1_\psi$$ we have the existence of $a'\in A^1_\psi$ and $f'=(f_1,f_2')\in F\cap (\{f_1\}\times \Z)$ such that
$$u+(dK,0)=a'+f'.$$
Thus, 
\begin{equation*}
    \begin{split}
        p_1(a')= & \ p_1(u+(dK,0)-f')\\
        = & \ p_1(a+f+(dK,0)-f')\\
        = & \ p_1(a)+ f_1+dK-f_1\\
        = & \ x+dK.
    \end{split}
\end{equation*}
This gives that $a'\in A^{1}_\psi\cap (\{x+dK\}\times \mathbb{Z})$, and \eqref{eq:udkin} in turn. Hence, \eqref{eq:oneside} holds.

The inclusion 
$$F\oplus \left(A^{1}_\psi\cap (\{x+dK\}\times \mathbb{Z})\right) \subseteq  F\oplus \left(A^{1}_\psi\cap (\{x\}\times \mathbb{Z})\right)+(dK,0)$$
follows by repeating the same argument for $x+dK$ in place of $x$ and $-d$ in place of $d$.
\end{proof}

\begin{lemma}\label{lem:PeriodicNotParallel}
    There is $\ell>0$ such that for every $P\in \quakeall{A}{F}$ if there exists  $2\le i\le m$ with $P^i\not=\emptyset$ then  
    $F\oplus P$ is $\ell \mathbb{Z}^2$-periodic.
\end{lemma}

\begin{proof}
Set $\ell>0$ to be the least common multiple of $M$ from Lemma~\ref{lem:PeriodicNonParallel}, $K$ from Lemma~\ref{lem:slices} and $|k|$, where $(0,k)=h_1$ from Proposition~\ref{pr:EarthquakeStructure}. 

    Let $P\in \quakeall{A}{F}$ be such that$P^i\not=\emptyset$ for some $2\le i\le m$.

    Let $\{r_n = f_n+a_n\}_{n=0}^\infty$ be an enumeration of all the elements in $F\oplus P$, where $f_n\in F$, $a_n\in P$, such that:
\begin{itemize}
    \item $a_0\in \bigsqcup_{i=2}^m P^i$;
    \item for every $n\geq 1$ there exists $n'<n$ such that $p_1(r_n)\in p_1(F+a_{n'})$.
\end{itemize}
Observe that such an enumeration exists by the $\langle (0,1)\rangle$-periodicity and minimality of $F\oplus P$.

    We construct a sequence $\{C_n\}_{n=0}^\infty$ such that for every $n=0,1,\dots$
    \begin{itemize}
        \item[(i)] $C_n\subseteq C_{n+1}\subseteq P$;
        \item[(ii)] $B_n\coloneqq F\oplus C_n$ is $\ell \Z^2$-periodic;
        \item[(iii)] $r_n\in B_n$;
    \end{itemize}
as follows. We first set $C_0\coloneqq \bigsqcup_{i=2}^m P^i \subseteq P$, so $C_0$ is nonempty and satisfies (i). 
By Lemma~\ref{lem:PeriodicNonParallel}, the  set
    $$B_0=\bigsqcup_{i=2}^m F\oplus P^i$$
    is $M\mathbb{Z}^2$-periodic, giving (ii), by the choice of $\ell$. And, (iii) follows from the fact that  $a_0\in C_0$, $f_0\in F$ and $B_0=F\oplus C_0$.

 Now, let $n\geq 1$ and suppose we have constructed $C_0,\dots, C_{n-1}$ with properties (i),(ii) and (iii). 
We define $C_n$ as follows:

 If  $r_{n}\in B_{n-1}$, then we set $C_n\coloneqq C_{n-1}$, and so properties (i), (ii) and (iii) follow immediately from the hypothesis.

 Otherwise, $r_{n}\not\in B_{n-1}$, and we set $$C_{n}\coloneqq C_{n-1}\cup  \left(A^1\cap \left(\{p_1(a_n)+\ell \Z\}\times \mathbb{Z}\right)\right).$$
 In particular,  $r_{n}\in B_{n}=F\oplus C_n$, giving (iii). 
 
 Note that to show (ii), by hypothesis, it suffices to prove that  $F\oplus(C_{n}\setminus C_{n-1})$ is $\ell \Z^2$-periodic.
As $A^1$ is $\langle h_1\rangle$-periodic, the set
    $$A^1 \cap \left((p_1(a_n)+\ell\Z)\times \mathbb{Z}\right)$$
    is $\langle h_1\rangle$-periodic and thus, by the choice of $\ell$, it is $\mathbb{Z}(0,\ell)$-periodic. Therefore $F\oplus (C_n\setminus C_{n-1})$ is $\mathbb{Z}(0,\ell)$-periodic.
Moreover, as by Lemma~\ref{lem:slices} 
 \begin{align*}
    F\oplus \left(A^{1}\cap (\{p_1(a_n)\}\times \mathbb{Z})\right)+(d\ell,0) &=F\oplus \left(A^{1}\cap (\{p_1(a_n)+d\ell \}\times \mathbb{Z})\right)
 \end{align*}
    for every $d\in \mathbb{Z}$, we conclude that 
$F\oplus(C_{n}\setminus C_{n-1})$ is $\mathbb{Z}(\ell,0)$-periodic as well, giving (ii).

We turn to prove (i). By hypothesis, it suffices to show that $C_{n}\setminus C_{n-1}\subseteq P$.
Note that by construction of $\{r_n\}_{n=0}^\infty$  we have that there exists $n'\leq n-1$ such that $p_1(r_{n})\in p_1(F+a_{n'}) \subseteq p_1(B_{n'})$.
So, as by hypothesis $B_{n'}\subseteq B_{n-1}$, we have 
\begin{equation}\label{eq:p1in}
    p_1(r_{n})\in p_1(B_{n-1}).
\end{equation}
Now, let $a'\in C_{n}\setminus C_{n-1}$. 
By construction of $C_n$, there exists  $d\in \Z$ such that $p_1(a_n)+\ell d=p_1(a')$. 
Since, by hypothesis, $B_{n-1}$ is $\ell\Z^2$-periodic and as
 \begin{equation*}
        \begin{split}
            p_1(f_n+a')= 
             & \ p_1(f_n+ a_n)+\ell d\\
            = & \ p_1(r_{n})+\ell d,
        \end{split}
    \end{equation*}
     from \eqref{eq:p1in},  we conclude that 
\begin{equation}\label{eq:nonemptyslice}
    p_1(f_n+ a')\in p_1(F\oplus a')\cap p_1(B_{n-1}).
\end{equation}
So, since by hypothesis $C_{n-1}\subseteq  P$ and $B_{n-1}=F\oplus C_{n-1}$, from \eqref{eq:nonemptyslice} and  the $(0,1)\Z$-periodicity of $F\oplus P$ we conclude that $f_n +a'\in F\oplus P$, and so in turn we get $a'\in P$, giving (i).

Iterating this process, we obtain a sequence $\{C_n\}_{n=0}^\infty$ of sets such that for every $n=0,1,\dots$, $C_n$ satisfies (i), (ii) and (iii) above.

Using the sequence $\{C_n\}_{n=0}^\infty$ we conclude the proof:
By (i) we have that $$\bigcup_{n=0}^\infty B_n = \bigcup_{n=0}^\infty F\oplus C_n\subseteq F\oplus P.$$ 
On the other hand, by (iii) and since $F\oplus P=\{r_n\}_{n=0}^\infty$, we have that $$F\oplus P\subseteq \bigcup_{n=0}^\infty B_n.$$ Thus, $F\oplus P = \bigcup_{n=0}^\infty B_n$.
And so, since by (ii) the set $\bigcup_{n=0}^\infty B_n$ is $\ell \Z^2$-periodic,  we obtain that $F\oplus P$ is $\ell \Z^2$-periodic, as claimed.
\end{proof}

We are finally ready to prove Theorem~\ref{thm:DiscretePeriodicity}:

\begin{proof}[Proof of Theorem~\ref{thm:DiscretePeriodicity}]
Let $\ell$ be as in Lemma~\ref{lem:PeriodicNotParallel}.
We denote by $\quakeone{A}{F}$ the set of all the $P\in \quakeall{A}{F}$ such that $P=P^1$. Observe that, in particular, if $P\in \quakeall{A}{F}$ is a vertical column then $P\in \quakeone{A}{F}$.

If $P\in \quakeall{A}{F}\setminus \quakeone{A}{F}$, then Lemma~\ref{lem:PeriodicNotParallel} implies that $F\oplus P$ is $\ell \mathbb{Z}^2$-periodic.
Thus, we also have that 
$$\bigsqcup_{P\in \quakeone{A}{F}} F\oplus P = \Z^2\setminus \bigsqcup_{P\in \quakeall{A}{F} \setminus \quakeone{A}{F}} F\oplus P$$
 is $\mathbb{Z}(0,1)$-invariant and $\ell \mathbb{Z}^2$-periodic. Moreover,  as we chose $\ell$ such that $\Z(0,\ell)\subseteq \Z h_1$, every $P\in \quakeone{A}{F}$ is $\mathbb{Z}(0,\ell)$-periodic.

 Let $P_0\in \quakeone{A}{F}\cap \quake{A}{F}$, that is, $P_0$ is a discrete tectonic plate of $A$ which is not a discrete vertical column.
    This means that for every $a\in P_0$ there are $x\in p_1(F+a)$ and $y,y'\in \Z$ such that    
    \begin{equation}\label{eq:xy}
        (x,y)\in F\oplus \left(P_0\cap (\{p_1(a)\}\times \mathbb{Z})\right)
    \end{equation}
    and 
    \begin{equation}\label{eq:xy'}
        (x,y')\not\in F\oplus \left(P_0\cap (\{p_1(a)\}\times \mathbb{Z})\right).
    \end{equation}
    Moreover, by the $\mathbb{Z}(0,1)$-periodicity  of $F\oplus P_0$, we have
    \begin{equation}\label{eq:P0A1}
        P_0\cap (\{p_1(a)\}\times \mathbb{Z})=A^1\cap (\{p_1(a)\}\times \mathbb{Z}).
    \end{equation}

   Let $\psi:\quakeall{A}{F}\to\mathbb{Z}$ be such that $\psi(P_0)=y-y'$ and $\psi(P)=0$ for every $P \in\quakeall{A}{F}\setminus \{P_0\}$.
    Then, by Lemma~\ref{lem:slices} and \eqref{eq:P0A1}, we have that 
    \begin{equation}\label{eq:first}
        F\oplus \left(P_0\cap (\{p_1(a)\}\times \mathbb{Z})\right)+(d\ell,0)=F\oplus \left(A^1\cap (\{p_1(a)+d\ell\}\times \mathbb{Z})\right)
    \end{equation}
    and
    \begin{equation}\label{eq:second}
        F\oplus \left((P_0)_{\psi}\cap (\{p_1(a)\}\times \mathbb{Z})\right)+(d\ell,0)=F\oplus \left(A^{1}_\psi\cap (\{p_1(a)+d\ell\}\times \mathbb{Z})\right)
    \end{equation}
    for every $d\in \Z$.
   By \eqref{eq:xy}, \eqref{eq:first} implies that
    $$(x,y)+(d\ell,0)\in F\oplus \left(A^1\cap (\{p_1(a)+d\ell\}\times \mathbb{Z})\right)$$
    and as, by \eqref{eq:xy'},
    \begin{equation*}
        \begin{split}
            (x,y)=(x,y')+(0,y-y')\not \in & \  F\oplus \left(P_0\cap (\{p_1(a)\}\times \mathbb{Z})\right)+(0,y-y')\\
            = & \  F\oplus \left((P_0+(0,y-y'))\cap (\{p_1(a)\}\times \mathbb{Z})\right)\\
            = & \ F\oplus \left((P_0)_\psi\cap (\{p_1(a)\}\times \mathbb{Z})\right),
        \end{split}
    \end{equation*}
    \eqref{eq:second} implies that
    $$(x,y)+(d\ell,0)\not \in F\oplus \left(A^{1}_\psi\cap (\{p_1(a)+d\ell \}\times \mathbb{Z})\right).$$
    
    As $\psi(P)=0$ for every $P\in \quakeall{A}{F}\setminus \{P_0\}$, we conclude that $$A^1\cap (\{p_1(a)+d\ell\}\times \mathbb{Z})\subseteq  P_0$$ for every $d\in \Z$.
    In particular, as $a\in P_0$ was arbitrary, we get that $p_1(P_0)$ is $\ell\Z$-periodic.
    Combined with the fact that $F\oplus P_0$ is $(0,1)\Z$-periodic by Definition~\ref{def:discplates}~(i), we infer that $F\oplus P_0$ is $\ell\Z^2$-periodic.

    Thus, we have proven that $F\oplus P$ is $(\ell,0)\Z$-periodic for every $P\in \quake{A}{F}$, since $\quake{A}{F}$ can be partitioned as $$\left(\quakeall{A}{F}\setminus \quakeone{A}{F}\right)\sqcup \left( \quakeone{A}{F}\cap \quake{A}{F}\right).$$
    Consequently, 
    $$F\oplus A^\text{disc}_c = \Z^2\setminus \left(\bigcup_{P\in \quake{A}{F}} F\oplus P\right)$$ 
    is $(\ell,0)\Z$-periodic as well. This concludes the proof.
\end{proof}

\subsection{Application to tilings of discrete bands}\label{subsec:Bands}

As an application of the technique utilized in this section, we derive the following result about discrete tilings.

\begin{corollary}\label{cor:BtimesZ}
    Let $B\subseteq  \Z$ be finite and $A\in\Tile(F; B\times \Z)$, where $F\subseteq \Z^2$ is finite.
    Then $A=A^{\operatorname{disc}}_c$.
    In particular, $p_1(F)\oplus p_1(A)=B$.
\end{corollary}

\begin{remark}
    Observe that when the set $B\subseteq \Z$ and $F\subseteq  \mathbb{Z}^2$ are such that $B+[0,1)$ and $F+[0,1)^2$ are connected, the statement of Corollary~\ref{cor:BtimesZ} follows from \cite[Lemma~5.6]{GGRT}. 
\end{remark}

Since the argument is mostly included in the proof of Theorem~\ref{thm:DiscretePeriodicity} (more precisely, in the proof of  Lemma~\ref{lem:slices}), the proof below is not detailed.

\begin{proof}[Sketch of the proof of Corollary~\ref{cor:BtimesZ}]
   Take $\ell>0$ much larger than the diameter of $B$ and define $\tilde{A}=\bigsqcup_{d\in \Z} A+(d\ell,0)$.
It follows from the choice of $\ell$ that $\tilde{A}$ is a tiling of an $\ell\Z^2$-periodic set $X\times \Z$ by $F$, where $X=\bigsqcup_{d\in \Z} B+d\ell$.
A direct application of a version of Theorem~\ref{thm:DiscretePeriodicity} to the periodic set $X$, $F$ and $\tilde{A}$, that can be derived in the same way as Theorem~\ref{thm:DiscretePeriodicity}, shows that
$$\tilde{A}=(\widetilde{A})^{\operatorname{disc}}_c\sqcup \left(\bigsqcup_{P\in \quake{\widetilde{A}}{F}} P\right),$$
where every $F\oplus P$ is $(m,0)\mathbb{Z}$-periodic for some $m>0$ for every $P\in \quake{X}{F}$.
It follows from the definition of $X$ that $\quake{X}{F}=\emptyset$ and the claim follows. 
\end{proof}

\begin{remark}\label{rem:GPTproof}
    The following  alternative proof of Corollary~\ref{cor:BtimesZ} was suggested by an LLM during proofreading.
    We sketch the argument.

    Assume $B\ne\emptyset$, and translate $F$ and $A$ so that $\min p_1(F)=\min p_1(A)=\min B=0$.
    Using \cite[Corollary 3.5]{GreenfeldTaoStructureZd} one can show that $A$ is $(0,q)$-periodic for some positive integer $q$.
    Define the integer polynomials
    $$f(z)=\sum_{(x,y)\in F}z^x,\qquad c(z)=\sum_{\substack{(x,y)\in A\\0\le y<q}}z^x, \qquad b(z)=\sum_{x\in B}z^x$$
    and observe that the tiling equation implies that
    $$f(z)c(z)=q\cdot b(z).$$
    Let $d$ and $e$ be the positive greatest common divisors of the coefficients of $f$ and $c$, respectively.
    Since $b$ has coefficients in $\{0,1\}$, Gauss's lemma gives $de=q$, hence
    $$\frac{f(z)}d\,\frac{c(z)}e=b(z).$$
    Both factors have nonnegative integer coefficients.
    Since their product has coefficients in $\{0,1\}$, the exponent of each monomial on the right hand side, has a unique representation
    as a sum of exponents from the two factors on the left hand side, $p_1(F)$ and $p_1(A)$, so
    $$p_1(F)\oplus p_1(A)=B.$$
    Consequently, $A=A_c^{\operatorname{disc}}$ as desired.
\end{remark}

\section{Rational Approximations}\label{sec:RationalApprox}

Let $X\not=\emptyset$ be a regular open periodic subset of $\R$ such that $\partial X$ has finite local complexity, and set $E=X\times \R$.
In case $X\not=\mathbb{R}$, write $(x_i)_{i\in \Z}$ for some increasing enumeration of $\partial X$ with the convention that $(x_{2k},x_{2k+1})\subseteq X$ for every $k\in \mathbb{Z}$ and denote as $\alpha$ a period of $X$.
Note that $\{x_i\}_{i\in \mathbb{Z}}=\{x_i+\alpha\}_{i\in \Z}$ and the set of gaps
$$D_X=\{x_{2k}-x_{2k-1}:k\in \Z\}$$
is finite.
We set $D_X=\emptyset$ when $X=\R$.

    For an axis-parallel tile $\Omega$ and $T\in \Tile(\Omega; E)$, we denote by
     $K_{(\Omega,T)}$  the vector space over $\mathbb{Q}$ generated by the set
    $$\left\{p_i(v'+t')-p_i(v+t): i\in \{1,2\}, v,v'\in V(\Omega),t,t'\in T\right\},$$
    and let  $(\beta_j)_{j\in J}$ be a fixed basis of $K_{(\Omega,T)}$.

    We fix $v_0\in V(\Omega)$ and $t_0\in T$.
By translation invariance, we may, without loss of generality, assume that $v_0=t_0=(0,0)$. This, in particular, implies that  $p_i(t),p_i(v)\in K_{(\Omega,T)}$ for every $i\in \{1,2\}$, $t\in T$ and $v\in V(\Omega)$; henceforth, this will be our assumption in what follows.
Observe that if $X\not=\R$, then $\{x_i\}_{i\in \Z},D_X\subseteq K_{(\Omega,T)}$ and $\alpha\in K_{(\Omega,T)}$.
Also the assumption that $\partial X$ has finite local complexity is always satisfied when $T\in \Tile(\Omega,E)$, hence this will be tacitly assumed through out.

\begin{definition}
Let $\Omega$ be an axis-parallel polygonal tile and $T\in \Tile(\Omega;E)$.
    Let $\epsilon>0$.
    A $\mathbb{Q}$-linear map $\varphi:K_{(\Omega,T)}\to \mathbb{Q}$ that satisfies $\varphi(r)=r$ for every $r\in \Q\cap K_{(\Omega,T)}$  
    is called an \emph{$\epsilon$-approximation for $(\Omega,T)$} if 
    $$|\beta_j-\varphi(\beta_j)|<\epsilon$$
    for every $j\in J$. For fixed $\Omega,T$, we may abbreviate and say that $\varphi$ is an \emph{$\epsilon$-approximation}.
 We say that $\varphi$ is an \emph{approximation} if it is an $\epsilon$-approximation for some $\epsilon>0$.
\end{definition}

Note that as $\Q$ is dense in $\R$ and $(\beta_j)_{j\in J}$ is a basis of $K_{(\Omega,T)}$ we have that $\epsilon$-approximations exist for any fixed $\epsilon>0$.

Let $\varphi:K_{(\Omega,T)}\to \mathbb{Q}$ be an approximation.
we define the collections $\{v^\varphi\}_{v\in V(\Omega)}$ and $T^\varphi=\{t^\varphi\}_{t\in \T}$ as $p_i(v^\varphi)=\varphi(p_i(v))$ and $p_i(t^\varphi)=\varphi(p_i(t))$ for $i=1,2$.
In particular, $(v_0)^\varphi=(t_0)^\varphi=(0,0)$.
Given $S\subseteq  T$, we set $S^\varphi=\{t^\varphi\in T^\varphi:t\in S\}$.
In case $X\not=\R$, we also define $\alpha^\varphi=\varphi(\alpha)$ and $x^\varphi_i=\varphi(x_i)$ for every $i\in \Z$.

\begin{definition}
Let $\Omega$ be an axis-parallel polygonal tile and $T\in \Tile(\Omega;E)$.
    We say that an approximation $\varphi$ for $(\Omega, T)$ is a \emph{good approximation} if $p_i(v^\varphi)-p_i((v')^\varphi)$ has the same sign ($+$, $-$, or $0$) as $p_i(v)-p_i(v')$ for every $v,v'\in V(\Omega)$ and $i=1,2$, and $\varphi(z)>0$ for every $z\in D_X$. (See Figure \ref{fig:goodapprox}.)  
\end{definition}

\begin{figure}[ht]
    \centering
    \includegraphics[width=0.2\linewidth]{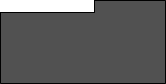}
    \hspace{2em}
    \includegraphics[width=0.2\linewidth]{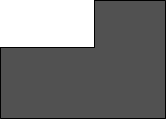}
    \hspace{2em}
    \includegraphics[width=0.2\linewidth]{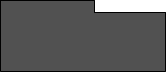}

    \caption{(Left) an irrational tile $\Omega$. (Middle) a good approximation of $\Omega$ by a rational tile. (Right) a bad (not good) approximation of $\Omega$ by a rational tile. Good approximations do not change the \emph{shape} of the tile in a qualitative manner.}
    \label{fig:goodapprox}
\end{figure}

Assuming that $\varphi$ is a good approximation, we define an axis-parallel polygonal set with rational vertices $\Omega^\varphi$.
The set of vertices is defined as $V(\Omega^\varphi)=\{v^\varphi\}_{v\in V(\Omega)}$, and the set of edges $E(\Omega^\varphi)$ consists of line segments $\omega_{v^\varphi,(v')^\varphi}$ whenever $\omega_{v,v'}\in E(\Omega)$.
We also denote as $\omega^\varphi$ the edge in $E(\Omega^\varphi)$ that corresponds to $\omega\in E(\Omega)$.
Consider a maximal bounded connected component $\Omega_0$ of $\Omega$.
The assumption that $\varphi$ is a good approximation guarantees that there is a unique maximal connected open region $\Omega^\varphi_0$ with boundary consisting of those $\omega^\varphi\in E(\Omega^\varphi)$ such that $\omega\in E(\Omega_0)$.
This allows to define $\Omega^\varphi$ as the disjoint union of $\Omega^\varphi_0$, where $\Omega_0$ runs over all connected component of $\Omega$.
Similarly, in case $X\not=\R$, we may define the regular open set $X^\varphi=\bigsqcup_{k\in\Z} (x^\varphi_{2k},x^\varphi_{2k+1})$, as well as $E^\varphi=X^\varphi\times \mathbb{R}$.
Note that $\partial X^\varphi$ has finite local complexity.
We also set $X^\varphi=\R$ and $E^\varphi=\R^2$, in case $X=\R$.

Our main goal in this section is showing that a good approximation is tiling preserving: 

\begin{theorem}\label{thm:GoodApproximation}
    Let $\Omega$ be an axis-parallel tile, $T\in \Tile(\Omega; E)$ and $\varphi$ be a good approximation for $(\Omega, T)$.
    Then $T^\varphi\in \Tile(\Omega^\varphi;E^\varphi)$, where in case that $X\not=\R$ we have that $X^\varphi$ a regular open periodic set with period $\alpha^\varphi$.
\end{theorem}

In what follows, we will use the following notation, which allows us to view the approximation as a two-step process: We define $K^1_{(\Omega,T)}$ to be the vector space over $\mathbb{Q}$ generated by the sets
    $$\left\{p_1(v'+t')-p_1(v+t):, v,v'\in V(\Omega),t,t'\in T\right\}$$
    and similarly define $K^2_{(\Omega,T)}$.
    We then refer to the collections $\{v^1\}_{v\in V(\Omega)}$ and $T^1=\{t^1\}_{t\in \T}$ as $p_1(v^1)=\varphi(p_1(v))$, $p_2(v^1)=p_2(v)$ and $p_1(t^1)=\varphi(p_1(t))$, $p_2(t^1)=p_2(t)$.
    Analogously as above, define $\Omega^1$, $\omega^1$ for $\omega\in E(\Omega)$, and $X^1$.
    Observe that $K^2_{(\Omega^1,T^1)}=K^2_{(\Omega,T)}$, $(\Omega^1)^2=\Omega^\varphi$, $(T^1)^2=T^\varphi$, and $X^1=(X^1)^2=X^\varphi$.

As a preliminary step in proving Theorem~\ref{thm:GoodApproximation}, we prove the following lemma, which states that the relative position of edges in the tiling is preserved under good approximations.

\begin{lemma}\label{lem:step0}
Let $\Omega$ be an axis-parallel tile, $T\in \Tile(\Omega; E)$ and $\varphi$ be a good approximation.
Suppose that $\omega,\omega'\in E(\Omega)$,  $t,t' \in T$ are  such that $$m_1\left( p_2(\omega+t)\cap p_2(\omega'+t')\right)$$ is positive. Then $p_1(\omega^1+t^1)-p_1((\omega')^1+(t')^1)$  has the same sign\footnote{Note that since $\Omega$ is axis-parallel, we have that $\omega, \omega'$ as well as $\omega^1,(\omega')^1$ are vertical. And so, we refer to their $p_1$ projection as a number rather than a set.} ($+$, $-$, or $0$) as $p_1(\omega+t)-p_1(\omega'+t')$.
\end{lemma}

\begin{proof}
    Note that the claim follows trivially when $p_1(\omega+t)=p_1(\omega'+t')$.
    Then, by symmetry, it suffices to show that if $p_1(\omega+t)>p_1(\omega'+t')$ then $p_1(\omega^1+t^1)>p_1((\omega')^1+(t')^1)$. 
      By the assumptions that $m_1\left( p_2(\omega+t)\cap p_2(\omega'+t')\right)$ is positive and $T\in \Tile(\Omega;E)$, we can find finitely many  vertical segments $\omega_0,\dots,\omega_\ell\in E(\Omega)+T$ such that $p_1(\omega_0)=p_1(\omega'+t')$, $p_1(\omega_\ell)=p_1(\omega+t)$, $$m_1\left(  p_2(\omega_0)\cap \dots \cap  p_2(\omega_\ell) \right)>0$$ and for every $0\le j\le \ell-1$ there are  $v'_j,v_j\in V(\Omega)$, or $k_j\in Z$ such that
    \begin{equation}\label{eq:positive0}
        p_1(\omega_{j+1})-p_1(\omega_j)=p_1(v'_{j})-p_1(v_j)>0,
    \end{equation}
    or
    \begin{equation}\label{eq:positive00}
        p_1(\omega_{j+1})-p_1(\omega_j)=x_{2k_j}-x_{2k_j-1}>0.
    \end{equation}

    As $\varphi$ is good, by \eqref{eq:positive0}, or \eqref{eq:positive00} we have that $$p_1((\omega_{j+1})^1)-p_1((\omega_j)^1)$$ is positive for every $0\leq j\leq \ell -1$.
    Thus,
    \begin{align*}
        p_1(\omega^1+t^1)-p_1((\omega')^1+(t')^1)&=p_1((\omega_\ell)^1)-p_1((\omega_0)^1)\\ &=\sum_{j=0}^{\ell-1} p_1((\omega_{j+1})^1)-p_1((\omega_j)^1)
    \end{align*} 
    is positive, as desired.
\end{proof}

 We can now finally prove Theorem~\ref{thm:GoodApproximation}.

\begin{proof}[Proof of Theorem~\ref{thm:GoodApproximation}]
    Write $E=X\times \mathbb{R}$, and set $E^1=X^1\times \R$.
    It is enough to show that $T^1\in \Tile(\Omega^1;E^1)$ as we might apply the same reasoning for $K^2_{(\Omega^1,T^1)}$ and $\varphi$ to get that $(T^1)^2\in \Tile((\Omega^1)^2;E^\varphi)$ as $E^\varphi=(E^1)^2$.
    This implies the claim as $K^2_{(\Omega^1,T^1)}=K^2_{(\Omega,T)}$, $(\Omega^1)^2=\Omega^\varphi$ and $(T^1)^2=T^\varphi$.

    In case $X\not=\R$, it follows from the $\Q$-linearity of $\varphi$ that $X^1=X^\varphi$ is $\alpha^\varphi$ periodic and that $\partial X^1=\{x^\varphi_i\}_{i\in \Z}$.
    Moreover, Lemma~\ref{lem:step0} implies that $\Omega^1+t^1\subseteq E^1$ for every $t\in T$.

   The proof of $T^1\in \Tile(\Omega^1;E^1)$ consists of two steps, as follows.

    \emph{Step 1}.
    We show that  $m_2\left((\Omega^1+t^1)\cap (\Omega^1+(t')^1)\right)$ is positive for $t,t'\in T$ if and only if $t=t'$. In particular, $\Omega^1\oplus T^1$ is a packing.    
    
    As $\Omega$ is a polygonal set and $t^1=(t')^1$ when $t=t'$, we have that  $$m_2\left((\Omega^1+t^1)\cap (\Omega^1+(t')^1)\right)=m_2(\Omega^1)$$ is positive if $t=t'$. Conversely, suppose that
    $$m_2\left((\Omega^1+t^1)\cap (\Omega^1+(t')^1)\right)>0.$$
    As $\Omega^1$ is a polygonal set, there are open rectangles $R\subseteq  \Omega^1+t^1$ and $R'\subseteq  \Omega^1+(t')^1$ that satisfy $m_2(R\cap R')>0$ and $J=p_2(R)=p_2(R')$.
    Moreover, we may assume that there are vertical edges $(\omega_{00})^1+t^1,(\omega_{01})^1+t^1\in E(\Omega^1)+t^1$ that cover the vertical boundary of $R$ and $(\omega_{10})^1+(t')^1,(\omega_{11})^1+(t')^1\in E(\Omega^1)+(t')^1$ that cover the vertical boundary of $R'$.
Since $\varphi$ is good and the transition from $(\Omega, T)$  to $(\Omega^1, T^1)$ only changes the $x$-coordinates, the rectangles  $U$ and $U'$ that satisfy $J=p_2(U)=p_2(U')$ and whose vertical boundaries are covered by $\omega_{00}+t,\omega_{01}+t$ and $\omega_{10}+t',\omega_{11}+t'$, respectively, satisfy 
\begin{equation}\label{eq:Uin}
    U\subseteq  \Omega+t \; \hbox{ and }\; U'\subseteq  \Omega+t'.
\end{equation}
Moreover, as $m_2(R\cap R')>0$ and  $p_2(R)=p_2(R')$, we must have  $$(\omega_{00})^1+t^1 < (\omega_{11})^1+(t')^1 \; \hbox{ and }  \; (\omega_{10})^1+(t')^1 < (\omega_{01})^1+t^1.$$
By Lemma \ref{lem:step0}, we thus have 
$$\omega_{00}+t < \omega_{11}+t' \; \hbox{ and }  \; \omega_{10}+t' < \omega_{01}+t,$$
and so, as $p_2(U)=p_2(U')$ we conclude that 
\begin{equation}\label{eq:UU'pos}
    m_2(U\cap U')>0.
\end{equation}   
Hence, since $T\in \Tile(\Omega;E)$, by \eqref{eq:Uin} and \eqref{eq:UU'pos}, we must have $t=t'$.

    \emph{Step 2}. We show that $\Omega^1+T^1$ covers $E^1$ up to null sets.
    We have already observed that $\Omega^1+T^1\subseteq E^1=X^1\times \R$.
    Set
    $$C=E^1\setminus\overline{\Omega^1+T^1},$$
    and suppose towards a contradiction that $C$ is nonempty.
    By Step 1, $\Omega^1\oplus T^1$ is a packing, so $T^1$ is uniformly discrete and $\partial C$ consists of line segments.
    Observe that if $X\not=\R$, then $\{x^1_i\}\times \R\subseteq \partial (\Omega^1+T^1)$ for every $i\in \Z$.
    This is because $\varphi$ is a good approximation that is $\Q$-linear.
    Hence if we write $\{x_i\}\times \R=\bigcup_{\ell\in \Z} \omega_\ell+t_\ell$, then we must have that
    $$\{x^1_i\}\times \R=\bigcup_{\ell\in \Z} \omega^1_\ell+t^1_\ell.$$
    In particular, we have $\partial X^1\times \mathbb{R}\subseteq \partial(\Omega^1+T^1)$ in both cases $X=\R$ and $X\not=R$.

    We claim that $\partial C$ contains a vertical line segment disjoint from $\partial E^1$.
    Suppose for a contradiction that this is not the case, and let $P\subseteq \partial C$ be a maximal connected horizontal line segment.
    Then $P$ is a coset of $\mathbb{R}(1,0)$ in case $X=\R$, and must intersect $\partial E^1$ in case $X\not=\R$.
    The former case is impossible as the transition from $(\Omega, T)$ to $(\Omega^1, T^1)$  only changes the $x$-coordinates and $T\in \Tile(\Omega;\R^2)$, and the latter would contradict the fact that $\partial X^1\times \mathbb{R}\subseteq \partial(\Omega^1+T^1)$.
    Hence, the boundary of $C$ contains a vertical edge $\omega_{a,b}$ disjoint from $\partial E^1$, where $a=(x,y_a),b=(x,y_b)\in \R^2$ are distinct, and there is a vertical  $\omega\in E(\Omega)$ and $t\in T$ such that $\omega_{a,b}\subseteq  \omega^1+t^1$ (so, in particular, $x=p_1(\omega^1+t^1)$).
   Using that $T$ is in $\Tile(\Omega;E)$ and that $\omega\in E(\Omega)$ once again, we have that there are $s_1,\dots,s_\ell\in T$ that are distinct from $t$ such that the closure of $\Omega+\{s_1,\dots,s_\ell\}$ covers $\omega_{a',b'}$, where $a'=(x',y_a)$, $b'=(x',y_b)$ and $x'=p_1(\omega+t)$.
    Thus, since $x=\varphi(x')$, the closure of the set $\Omega^1+\{(s_1)^1,\dots, (s_\ell)^1\}$ covers $\omega_{a,b}$. However, this contradicts our assumption that $\omega_{a,b}$ is contained in $\partial(C)$, as from Step 1 we have that $t^1$ is distinct from $(s_1)^1,\dots, (s_\ell)^1$. This shows that the boundary of $C$ must be empty, hence, either $\Omega^1+T^1$ is empty, or $C$ is empty. Clearly, by construction, $\Omega^1+T^1$ cannot be empty since $\Omega +T$ is nonempty. We conclude that $C$ is empty, which means that $\Omega^1+T^1$ is a covering.
    
    Finally, combining Step 1 and Step 2, we see that $\Omega^1+T^1$ is both a packing and a covering of $E^1$; in other words, $T^1\in \Tile(\Omega^1;E^1)$, as claimed.
\end{proof}
    
Using Theorem~\ref{thm:GoodApproximation} we can now show that any good approximation preserves the sliding properties of the tiling.

\begin{lemma}\label{lem:ApproxKeepsCosetPieces}
    Let $E=X\times \R$ be as above, $\Omega$ be an axis-parallel polygonal tile, $T\in \Tile(\Omega;E)$, $S\subseteq T$ be such that 
    $\Omega\oplus S = Y\times \R$ for some regular open $Y\subseteq \R$ and $\varphi$ be a good approximation for $(\Omega, T)$.
    Then $\Omega^\varphi\oplus S^\varphi= Y^\varphi\times \R$, where $Y^\varphi \subseteq \R$ is such that $\partial Y^\varphi= \varphi(\partial Y)$ (in particular, $\varphi(\partial Y)$ is well-defined) and we have that $[a,b]\times \mathbb{R}\subseteq  \Omega\oplus S$ if and only if $[\varphi(a),\varphi(b)]\times \mathbb{R}\subseteq  \Omega^\varphi\oplus S^\varphi$ for every $a<b\in \partial Y$.
\end{lemma}
\begin{proof}
    The cases $Y=\emptyset$ and $Y=X$ are immediate.
    As $T\in \Tile(\Omega,E)$, $\partial Y$ is uniformly discrete, hence there is an increasing enumeration $(y_i)_{i\in I}$ of $\partial Y$, where $I=\{1,\dots,n\}$ for some $n\in \N$, $I=\pm\mathbb{N}$, or $I=\Z$.
    Each line $\{y_i\}\times\mathbb{R}$, $i\in I$, is covered by translated vertical edges of $\Omega$, hence $\partial Y\subseteq K_{(\Omega,T)}$.
    In addition, $\varphi(y_{i+1})>\varphi(y_i)$ for every $i\in I$ by the fact that $\varphi$ is a good approximation combined with Lemma~\ref{lem:step0}.
    This shows that $Y^\varphi$ can be defined to satisfy $\partial Y^\varphi= \varphi(\partial Y)$.
    As $T\in \Tile(\Omega,E)$, we have that $\Omega+s\subseteq Y\times \R$ if and only if $m_2((\Omega+s)\cap(Y\times \R))>0$.
    By Lemma~\ref{lem:step0}, we conclude that $\Omega^\varphi+s^\varphi\subseteq Y^\varphi\times\R$ if and only if $s\in S$.
    Now, the claim follows from Theorem~\ref{thm:GoodApproximation}.
\end{proof}

We finish this section by a technical result that implies that good approximations exist and that will be useful later on as well.

\begin{lemma}\label{lm:TechnicalForFiniteSets}
Let $\Omega$ be an axis-parallel polygonal tile, $T\in \Tile(\Omega;E)$, $\eta>0$ and $Z\subseteq  K_{(\Omega,T)}$ be finite.
Then there is $\epsilon>0$ such that every $\epsilon$-approximation $\varphi$ for $(\Omega, T)$ satisfies
$$|z-\varphi(z)|<\eta$$
for every $z\in Z$.
\end{lemma}
\begin{proof}
    We may assume that $Z$ contains a non-zero element as otherwise the claim follows immediately.
    As the set $Z$ is finite, its coordinates are spanned by a finite subset, $\{\beta_1,\dots,\beta_k\}$, of the basis $(\beta_j)_{j\in J}$.
    Thus, by the $\Q$-linearity of $\varphi$, the coordinates of the vectors in $\{\varphi(z)\}_{z\in Z}$ are spanned by $\{\varphi(\beta_1),\dots,\varphi(\beta_k)\}$.
    In other words, for every $z\in Z$ we can write $$z=\sum_{j=1}^k r_j^{z}\beta_j,$$ where $r_j^{z}\in \Q$ for every $j=1,\dots,k$. Thus, $\varphi(z)=\sum_{j=1}^k r_j^{z}\varphi(\beta_j)$.
    
    Let $r=\max\{|r_j^{z}|\colon z\in Z, j=1,\dots,k\}$.
    Taking $\epsilon= \frac{\eta}{10kr}$, we have that if $\varphi$ is an $\epsilon$-approximation for $(\Omega , T)$ then 
    for every $z\in Z$  
      \begin{align*}
          |z-\varphi(z)|& \leq 2r \sum_{j=1}^k |\beta_j-\varphi(\beta_j)| \\ 
          &< 2rk\epsilon < \eta,
      \end{align*}
      thus $\varphi$ has the required properties.
\end{proof}

\begin{corollary}\label{cor:GoodApproxExists}
Let $\Omega$ be an axis-parallel polygonal tile and $T\in \Tile(\Omega;E)$.
    There is $\epsilon>0$ such that every $\epsilon$-approximation $\varphi$ for $(\Omega, T)$ is good. 
\end{corollary}
\begin{proof}
    Apply Lemma~\ref{lm:TechnicalForFiniteSets} with the finite set
    $$Z=\{p_i(v)-p_i(v')\colon v,v'\in V(\Omega),  i=1,2\}\cup D_X$$
    with $\eta= \min\{|z|:0\not= z\in Z\}$.
    Then for every $(v_1,v_2),(v'_1,v'_2)\in V(\Omega)$ and $i=1,2$ such that $p_i(v)\not\not =p_i(v')$ we have
    \begin{align*}
          |\; | p_i(v^\varphi)-p_i((v')^\varphi)|- |p_i(v)-p_i(v')|\;|&= |\; | \varphi(p_i(v)-p_i(v')) |- |p_i(v)-p_i(v') | \;| \\
          &\leq   \left|\varphi(p_i(v)-p_i(v'))-(p_i(v)-p_i(v'))\right| \\ 
          & < \eta \leq |p_i(v)-p_i(v')|.
    \end{align*}
    Similar argument shows that $\varphi(z)>0$ for every $z\in D_X$, and the claim follows.
\end{proof}

\section{Good approximation of tiling with finite local complexity}\label{sec:GoodApprox}

In this section, we prove properties of good approximations under the additional assumption that the tiling $T$ has \emph{finite local complexity}.

\begin{definition}[Local combinatorial structure]\label{def:LocalCombStructure}
    Let $\Omega$ be a polygonal set. Suppose
     $\Omega\oplus S$ is a packing and let $v\in V(\Omega)+S$. We define
$$\mathcal{B}(v,\Omega\oplus S)=\frac{1}{r}\left(\left((\mathcal{B}_r + v)\cap (\Omega\oplus S)\right)-v\right)$$
where $r>0$ is such that for every $0<r'<r$ we have
$$\frac{1}{r'}\left((\left(\mathcal{B}_{r'}+ v)\cap (\Omega\oplus S)\right)-v\right)=\frac{1}{r}\left(\left((\mathcal{B}_r + v)\cap (\Omega\oplus S)\right)-v\right).$$
Observe that such $r>0$ must exist since $\Omega$ is a polygonal set.
\end{definition}

\begin{theorem}\label{thm:GoodFiniteLocalComplexity}
    Let $X$ be a regular open periodic subset of $\R$, $T\in \Tile(\Omega;E)$, where $\Omega$ is an axis-parallel polygonal set and $E=X\times \mathbb{R}$.
    Suppose that $T$ has finite local complexity and let $M>0$.
    Then there is $\epsilon>0$ such that every $\epsilon$-approximation $\varphi$ of $(\Omega,T)$ is good and
    \begin{enumerate}
        \item $E^\varphi=\Omega^\varphi\oplus T^\varphi$ is an open periodic set of the form $X^\varphi\times \mathbb{R}$;
        \item $T^\varphi, V(\Omega^\varphi)\subseteq  \frac{1}{d}\mathbb{Z}^2$ for some $0\not=d\in \mathbb{N}$;
        \item $\diam(\Omega^\varphi)\le 2\diam(\Omega)$;
        \item if $v+t\in \omega_{ab}+s$, where $a,b,v\in V(\Omega)$ and $t,s\in T$, then
        $$p_i(v+t)-p_i(x+s) \text{ and } p_i(v^\varphi+t^\varphi)-p_i(x^\varphi+s^\varphi)$$
        have the same sign ($+$, $-$, or $0$) for every $i\in \{1,2\}$ and $x\in \{a,b\}$;
        \item $(\omega+t)\cap (\omega'+s)=\emptyset$ if and only if $(\omega^\varphi+t^\varphi)\cap ((\omega')^\varphi+s^\varphi)=\emptyset$ for every $\omega,\omega'\in E(\Omega)$ and $t,s\in T$;
        \item if $(\overline{\Omega}+t)\cap (\overline{\Omega}+s)=\emptyset$, then $(\overline{\Omega^\varphi}+t^\varphi)\cap (\overline{\Omega^\varphi}+s^\varphi)=\emptyset$ for every $t,s\in T$;
        \item for every $S\subseteq  T$ and $v\in V(\Omega)+S$ we have that
        $$\mathcal{B}(v,\Omega\oplus S)=\mathcal{B}(v^\varphi,\Omega^\varphi\oplus S^\varphi);$$
        \item for every $S\subseteq  T$ we have that $S$ is a $\mathbb{R}(0,1)$-column (or $\mathbb{R}(1,0)$-column) if and only if $S^\varphi$ is a $\mathbb{R}(0,1)$-column (or $\mathbb{R}(1,0)$-column);
        \item for every $v,w\in V(\Omega)+ T$, if $\|w-v\| <M$, then $\|(w^\varphi-v^\varphi)-(w-v)\| <1/M$;
        \item for every $v,w\in V(\Omega)+ T$, if $\|w^\varphi-v^\varphi\| <M/2$, then $\|w-v\| <M$;
        \item we have $v^\varphi=w^\varphi$ if and only if $v=w$ for every $v,w\in V(\Omega)+T$, and $t^\varphi=s^\varphi$ if and only if $t=s$ for every $t,s\in T$;
        \item for every $\delta\in \{\mathbb{R}(1,0),\mathbb{R}(0,1)\}$ and $v,w\in V(\Omega)+T$ such that $\|v^\varphi-w^\varphi\| <M/2$, the vertices $v$ and $w$ lie in a coset of $\delta$ if and only if $v^\varphi$ and $w^\varphi$ lie in a coset of $\delta$;
        \item for every $v_0,w_0,v_1,w_1\in V(\Omega)+T$ and $i\in \{1,2\}$ we have
        $p_i(w_0-v_0)=p_i(w_1-v_1)$
        whenever $p_i(w^\varphi_0-v^\varphi_0)=p_i(w^\varphi_1-v^\varphi_1)$ and $\|w^\varphi_0-v^\varphi_0\| ,\|w^\varphi_1-v^\varphi_1\| <M/2$;
        \item for every $n,m<M$, $\omega_1,\dots,\omega_n,\omega'_1,\dots,\omega'_m \in E(\Omega)$ and $i\in \{1,2\}$, if
        $$\sum_{j=1}^n |p_i(\omega^\varphi_j)|=\sum_{j=1}^m |p_i((\omega_j')^\varphi)|$$
         then 
        $$\sum_{j=1}^n |p_i(\omega_j)|=\sum_{j=1}^m |p_i(\omega'_j)|.$$
    \end{enumerate} 
\end{theorem}
\begin{proof}
    By Corollary~\ref{cor:GoodApproxExists}, there is $\epsilon_0>0$ such that every $\epsilon_0$-approximation $\varphi$ is good.

    Recall that in the definition of $T^\varphi$ and $\Omega^\varphi$, we fix $v_0\in V(\Omega)$ and $t_0\in T$ and have that $(v_0)^\varphi=(t_0)^\varphi=(0,0)$.

    (1).
    This follows directly from Theorem~\ref{thm:GoodApproximation}.
    
    (2).
    As $X\times \mathbb{R}$ is periodic and $T$ has finite local complexity, it follows that there is a finite set $R\subseteq  T$ such that
    $$t-t_0=\sum_{r\in R} a_r(r-t_0)$$
    for every $t\in T$, where $a_r\in \mathbb{Z}$.
    By linearity of $\varphi$, it follows that
    $$t^\varphi=\sum_{r\in R} a_r r^\varphi,$$
    which implies that $T^\varphi\subseteq  \frac{1}{d'}\mathbb{Z}^2$ for some integer $d'>0$ such that $r^\varphi\in \frac{1}{d'}\mathbb{Z}^2$ for every $r\in R$.
    An analogous argument for $V(\Omega^\varphi)$ shows that there is an integer $d''>0$ such that $V(\Omega^\varphi)\subseteq  \frac{1}{d''}\mathbb{Z}^2$, and the claim follows by setting $d=d'\cdot d''$.

    (3).
    Follows by taking $\epsilon_0\geq \epsilon_3>0$ small enough as $V(\Omega)$ is finite.

    (4). 
    By the finite local complexity of $T$, the subset $Z$ of $K_{(\Omega,T)}$ of all values of the form $p_i(v+t)-p_i(x+s)$, where $i\in\{1,2\}$, $a,b,v\in V(\Omega)$, $x\in \{a,b\}$, and $t,s\in T$ are such that $v+t\in \omega_{ab}+s$, is finite.
    By Lemma~\ref{lm:TechnicalForFiniteSets}, for every $\eta>0$ there is $\epsilon_3\ge \epsilon_4>0$ such that every $\epsilon_4$-approximation $\varphi$ satisfies 
    $$|(p_i(v+t)-p_i(x+s))-(p_i(v^\varphi+t^\varphi)-p_i(x^\varphi+s^\varphi))|<\eta$$
    for every $i\in\{1,2\}$, $a,b,v\in V(\Omega)$ and $t,s\in T$ are such that $v+t\in \omega_{ab}+s$.
    In particular, taking $0<\eta<1/3\min\{|\alpha|:\alpha\in Z\setminus \{0\}\}$ gives the claim whenever $p_i(v+t)-p_i(x+s)\not= 0$. If $p_i(v+t)-p_i(x+s)=0$, then it follows from $\Q$-linearity of $\varphi$.

    (5).
    Note that part (4) implies that if $(\omega+t)\cap (\omega'+s)\not=\emptyset$, then $(\omega^\varphi+t^\varphi)\cap ((\omega')^\varphi+s^\varphi)\not=\emptyset$.
    Conversely, let $\omega+t\in E(\Omega)+T$.
    From a combination of part (4) and the fact that $\Omega^\varphi\oplus T^\varphi$ is a packing it follows  that if $(\omega+t)\cap \partial (X\times \mathbb{R})=\emptyset$, then $(\omega')^\varphi+s^\varphi$ intersects $\omega^\varphi+t^\varphi$ if and only if $\omega'+s$ intersects $\omega+t$.
    Now suppose that $(\omega+t)\cap \partial (E)\not=\emptyset$.
    Recall the properties of $X^\varphi$ from Theorem~\ref{thm:GoodApproximation}, from which it follows that if $(\omega+t)\cap \partial (E)\not=\emptyset$ then $(\omega')^\varphi+s^\varphi$ intersects $\omega^\varphi+t^\varphi$ if and only if $\omega'+s$ intersects $\omega+t$ for every $\epsilon_4$-approximation $\varphi$.
    Hence, taking $\epsilon_5>0$ so that $\epsilon_4\ge \epsilon_5$ would give the claim (since $\epsilon_0\geq \epsilon_4$).

    (6).
    Follows from (5) combined with the fact that $\Omega^\varphi\oplus T^\varphi$ is a packing.

    (7).
    Follows from (4).

    (8).
    Follows from (7).
    We note that the direction that if $S$ is a column, then $S^\varphi$ is a column does not use finite local complexity.

    (9).
    By the assumption that $T$ has finite local complexity we have that the set 
    $$Z=\{p_i(v-w):\|v-w\| <M, \ v,w\in V(\Omega)+T, \ i\in \{1,2\}\}\subseteq  K_{(\Omega,T)}$$
    is finite. Thus, by Lemma~\ref{lm:TechnicalForFiniteSets} there is $\epsilon_5\ge\epsilon_9>0$ such that the claim of part (9) is satisfied for every $\epsilon_{9}$-approximation $\varphi$.

    (10).
    As $E = X\times \mathbb{R}$ is periodic, there is $r\ge 1$ such that for every $x\in \mathbb{R}^2$, there is $v\in V(\Omega)+T$ such that $\|v-x\| \le r$.
    Pick 
    \begin{equation}\label{eq:R}
        R=\max\{100r^2,2M^2,2/(3M)\}.
    \end{equation}
    By part (9), we can find $\epsilon_9\ge\epsilon_{10}>0$ such that for every $\epsilon_{10}$-approximation $\varphi$ and every $v,w\in V(\Omega)+ T$ with $\|w-v\| <3R$ we have
    $$\|(w-v)-(w^\varphi-v^\varphi)\| <1/3R.$$
    
    Let $v,w\in V(\Omega)+T$ be such that $\|v-w\| \ge R$, and assume without loss of generality that $|p_1(v-w)|\ge |p_2(v-w)|$ and $p_1(v-w)\le 0$.
    Pick a sequence $(x_0,\dots,x_k)$ on the line segment connecting $v$ and $w$ such that
    \begin{itemize}
        \item [(i)] $x_0=v$ and $x_k=w$,
        \item [(ii)] $R\le \|x_i-x_{i+1}\|  \le 2R$ for every $0\le i<k$,
        \item [(iii)] $p_1(x_{i}-x_{i+1})\le 0$.
    \end{itemize}
    Then, in particular, $\|w-v\| \le 2Rk$.
    By the definition of $r$, we find a sequence $(a_0,\dots a_k)$ in $V(\Omega)+T$ such that $a_0=x_0=v$, $a_k=x_k=w$ and $\|a_i-x_i\| \le r$ for every $1\le i \le k-1$. By the assumption $|p_1(v-w)|\ge |p_2(v-w)|$, (ii) and (iii), we have
    $$p_1(a_{i+1}-a_i)= p_1(x_{i+1}-x_i)+p_1(a_{i+1}-x_{i+1})+p_1(x_{i}-a_{i})\ge \sqrt{R/2}-2r.$$
    Since
    $$\|a_{i+1}-a_i\|\le 2R+2r<3R,$$
    the choice of $\epsilon_{10}$ gives
    $$\|((a_{i+1})^\varphi-(a_i)^\varphi)
    -(a_{i+1}-a_i)\|<\frac{1}{3R}.$$
    Consequently,
    $$p_1((a_{i+1})^\varphi-(a_i)^\varphi)>p_1(a_{i+1}-a_i)-\frac{1}{3R}\ge\sqrt{R/2}-2r-\frac{1}{3R}.$$
    By linearity of $\varphi$, this shows that
    $$p_1(w^\varphi-v^\varphi)=\sum_{i=0}^{k-1} p_1((a_{i+1})^{\varphi}-(a_i)^\varphi)\ge k\left(\sqrt{R/2}-2r-2/(3R)\right).$$
    This gives
    \begin{equation}\label{eq:phidist}
        \begin{split}
            \|w^\varphi-v^\varphi\| \ge & \ k\left(\sqrt{R/2}-2r-2/(3R)\right)\\
            \ge & \ k\left(\sqrt{R/2}-\sqrt{R/25}-2/(100r^2)\right)\\
            \ge & \ \sqrt{2Rk/8}\ge \sqrt{\|w-v\| /8},
        \end{split}
    \end{equation}
    by \eqref{eq:R}.
    To complete the proof, let $v,w\in V(\Omega)+T$ be such that $\|w^\varphi-v^\varphi\| \le M/2$.
    Observe that $\|w-v\| \le R$, as otherwise, by \eqref{eq:R}, $\|w-v\|>R\ge 2M^2$, which by \eqref{eq:phidist} contradicts the assumption  $\|w^\varphi-v^\varphi\| \le M/2$. 
    By the choice of $\epsilon_{10}>0$, we can further estimate that 
    $$M/2\ge \|w^\varphi-v^\varphi\| > \|w-v\| -1/(3R),$$
    and the claim follows as $M\ge M/2+1/3R$,  by \eqref{eq:R}.

    (11).
    The reverse implication holds for any good approximation and follows from linearity of $\varphi$.
    Suppose that $v^\varphi=w^\varphi$.
    Then, by part (5), there must be $t\in T$ and $a,b\in V(\Omega)$ such that $v=a+t$ and $w\in \omega_{ab}+t$.
    By part (4), this can only happen if $w=a+t$.
    This shows that $v=w$.
    The claim for $t$ and $s$ follows as a special case.

    (12).
    Without loss of generality assume that $\delta=\mathbb{R}(0,1)$.
    As $T$ has finite local complexity, we have that the set 
    \begin{equation}\label{eq:Zdef}
    Z=\{p_1(v-w):v,w\in V(\Omega)+T \text{ with } \|v-w\| <M\}\subseteq  K_{(\Omega,T)}\end{equation}
    is finite.
    By Lemma~\ref{lm:TechnicalForFiniteSets}, for every $\eta>0$ there is $\epsilon_{10}\ge\epsilon_{12} >0$ such that every $\epsilon_{12}$-approximation $\varphi$ satisfies that
    $$|p_1(v-w)-p_1(v^\varphi-w^\varphi)|<\eta$$
    for every $v,w\in V(\Omega)+T$ such that $\|v-w\| <M$.
    In particular, this holds for $0<\eta<1/3\min\{|\alpha|:\alpha\in Z\setminus \{0\}\}$.

    Let $v,w\in V(\Omega)+T$ be such that $\|v^\varphi-w^\varphi\| <M/2$, and suppose that $v^\varphi$ and $w^\varphi$ lie in the same coset of $\delta$.
    Suppose that $v$ and $w$ are not in the same coset of $\delta$, i.e., $p_1(v-w)\neq 0$. Then, by the choice of $\eta$, we have that $|p_1(v-w)|\geq 3\eta$. Moreover, by part (10), we have that $\|v-w\| <M$, which, by the definition of $Z$, implies that 
    $$|p_1(v-w)-p_1(v^\varphi-w^\varphi)|<\eta.$$
    Therefore,  $|p_1(v^\varphi-w^\varphi)|> 2\eta>0$ which contradicts the assumption that $v^\varphi$ and $w^\varphi$ lie on the same coset of $\delta$.

    (13).
    Without loss of generality, we may assume that $i=1$.
    Define the finite set $Z$ as in \eqref{eq:Zdef} 
    and let $0<\eta<1/6\min\{|\alpha-\beta|:\alpha\not=\beta\in Z\}$.
    By Lemma~\ref{lm:TechnicalForFiniteSets}, there is $\epsilon_{12}\geq \epsilon_{13}>0$ such that every $\epsilon_{13}$-approximation $\varphi$ satisfies 
    $$|p_1(w-v)-p_1(w^\varphi-v^\varphi)|<\eta$$
    for every $v,w\in V(\Omega)+T$ such that $\|v-w\|<M$.

    Let $v_0,v_1,w_0,w_1\in V(\Omega)+T$ be as in the claim.
    By part (10), we have that $\|w_0-v_0\| ,\|w_1-v_1\|<M$.
    Suppose for a contradiction that $p_1(w_0-v_0)\not=p_1(w_1-v_1)$.
    Then we have, by the choice of $\eta$, that $|p_1(w_0-v_0)-p_1(w_1-v_1)|>6\eta$.
    This gives
    \begin{equation*}
        \begin{split}
            & \ |p_i(w^\varphi_0-v^\varphi_0)-p_i(w^\varphi_1-v^\varphi_1)| \\
            \ge & \ |p_1(w_0-v_0)-p_1(w_1-v_1)|-|p_1(w_0-v_0)-p_1(w^\varphi_0-v^\varphi_0)|\\
            & \ -|p_1(w_1-v_1)-p_1(w^\varphi_1-v^\varphi_1)|\\
            \ge & \ 4\eta>0,
        \end{split}
    \end{equation*}
    which contradicts the assumption that $p_i(w^\varphi_0-v^\varphi_0)=p_i(w^\varphi_1-v^\varphi_1)$.

    (14).
    As $M>0$ is fixed, $E(\Omega)$ is finite and so is the set
    $$Z=\left\{\sum_{j=1}^k |p_i(\omega_j)|:1\le k<M \ \omega_i\in E(\Omega), \ i\in \{1,2\} \right\}\subseteq  K_{(\Omega,T)}.$$
    A direct application of Lemma~\ref{lm:TechnicalForFiniteSets} as above shows that there is $\epsilon_{13}\ge\epsilon_{14}>0$ such that every $\epsilon_{14}$-approximation $\varphi$ satisfies the statement of (14).

    Finally, since clearly, if $0<\epsilon'\leq \epsilon$ then every $\epsilon'$-approximation is in particular a $\epsilon$-approximation, we complete the proof by letting $\epsilon\coloneqq \epsilon_{14}$.
\end{proof}

\section{Earthquake decomposition and periodicity}\label{sec:earthquakesperiod}

In this section we introduce the notion of {\bf \emph{earthquake decomposition}} of a tiling in $\R^2$ and, using Theorems~\ref{thm:DiscretePeriodicity} and  \ref{thm:GoodApproximation}, prove the periodicity of the earthquake components.

We begin by introducing the {\bf \emph{proper intersection equivalence relation}}:

 \begin{definition}[Proper intersection equivalence class]\label{def:plates}
     Let $\Omega\subset\R^2$ be an axis-parallel polygonal tile and let $T\in \Tile(\Omega;\R^2)$.
     Define $\sim_{\Omega,\Uparrow}$ to be the equivalence relation on $T$ generated by the pairs $t\not= t' \in T$ such that there are horizontal edges $\omega,\omega'\in E(\Omega)$ with $\omega+t$ and $\omega'+t'$ properly intersect.

We refer to a $\sim_{\Omega,\Uparrow}$-equivalence class as a {\bf \emph{tectonic plate}} of $T$, and denote  the set of tectonic plates of $T$ by $\quakeall{T}{\Omega}$. (See Figure \ref{fig:proper_equiv}.)
 \end{definition}

\begin{figure}[ht]
    \centering
     \includegraphics[width=0.6\linewidth]{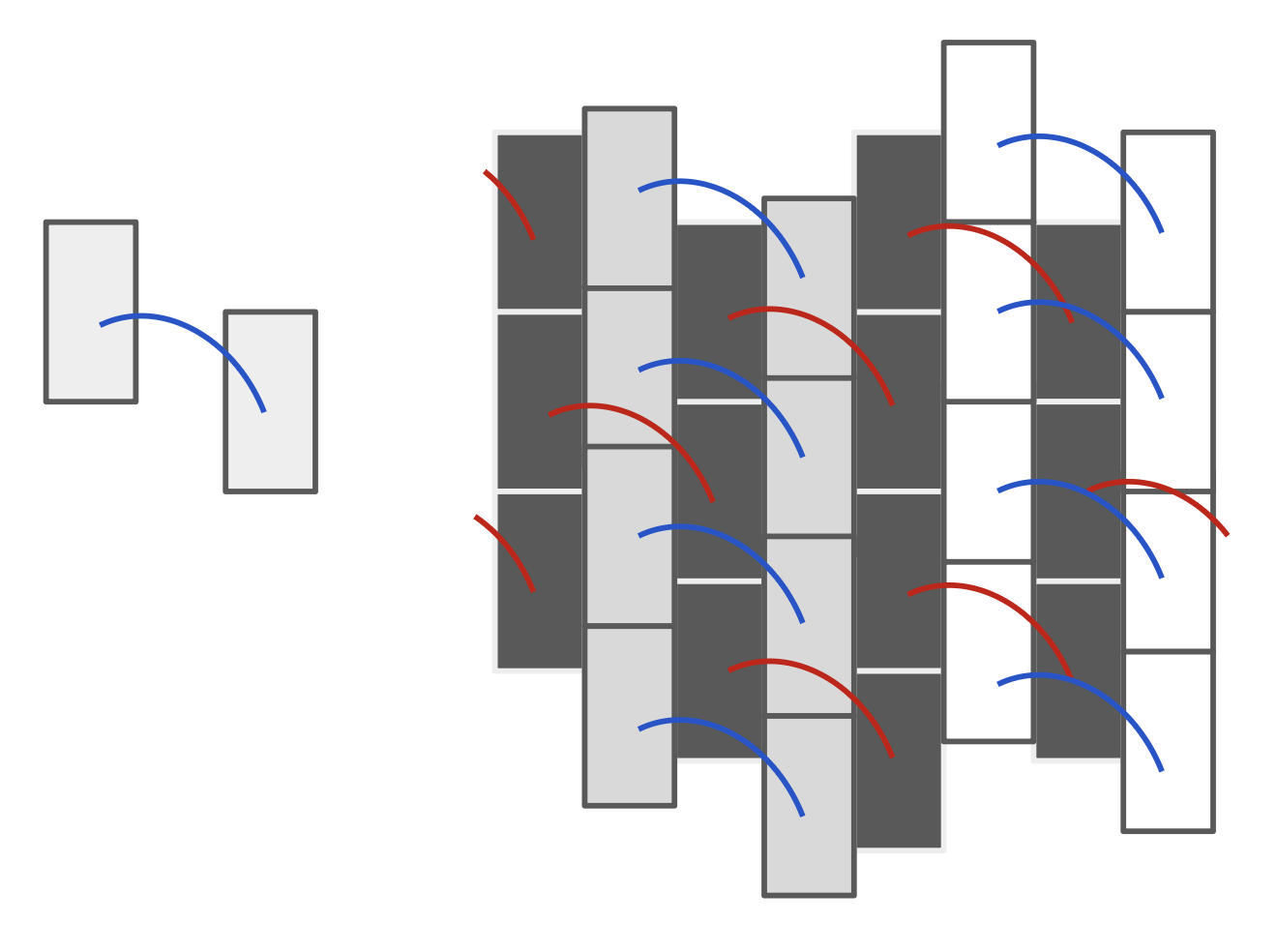}
    \caption{On the left is a tile $\Omega$. On the right, a tiling $T$ by $\Omega$, where the different colors indicate different $\sim_{\Omega,\Uparrow}$-equivalence classes.}
    \label{fig:proper_equiv}
\end{figure}

We show that the tectonic plates in $\quakeall{T}{\Omega}$ constitute a  partition of the tiling into minimal horizontal slabs. 
 
\begin{lemma}\label{lem:BasicQuake}
    Let $T$ be a tiling of $\mathbb{R}^2$ by $\Omega$ and $P\in \quakeall{T}{\Omega}$.
    Then:
    \begin{itemize}
        \item[(i)]  There exists a (unique) regular open set $X_P\subseteq \R$ such that $\Omega\oplus P=X_P\times \R$ up to null sets.
        In fact, we have $\interior(\overline{\Omega}\oplus P)=\interior(\overline{\Omega\oplus P})=X_P\times \R$.
        \item[(ii)] $P$ is minimal with respect to property (i); i.e., if $P'\subseteq P$ satisfies (i) then $P'\in \{\emptyset ,P\}$.
    \end{itemize}
\end{lemma}

\begin{remark}
    Lemma~\ref{lem:BasicQuake} shows, in particular,  that the set $\quakeall{T}{\Omega}$ is the continuous analog of the discrete tectonic plates introduced in Section \ref{sec:discrete}.
\end{remark}

\begin{proof}[Proof of Lemma~\ref{lem:BasicQuake}] 
   Suppose that $P\in \quakeall{T}{\Omega}$. If $\interior(\overline{\Omega\oplus P})$ is not $(0,1)\R$-invariant then there exists a horizontal $\omega\in E(\Omega)$ and $s\in P$ such that  $$m_1\left((\omega+s)\cap \partial (\overline{\Omega\oplus P})\right)$$ is positive. Thus, as $T\in \Tile(\Omega;\R^2)$, there must exist $t\in T\setminus P$ and a horizontal $\omega'\in E(\Omega)$ such that $\omega+s$ and $\omega'+t$ properly intersect. This gives that $t\sim_{\Omega,\Uparrow} s$, which is a contradiction since $P$ is a $\sim_{\Omega,\Uparrow}$-equivalence class and $t\not\in P$. This gives (i).

    To prove (ii), suppose that $S_0\subseteq P$ is nonempty and there exists $X_{S_0}\subseteq \R$ such that $\Omega\oplus S_0=X_{S_0}\times \R$. 
    Let $S_1\coloneqq P\setminus S_0$. If $S_1$ is nonempty then, as $\Omega\oplus P=X_P\times \R$ and $\Omega\oplus S_0=X_{S_0}\times \R$, we have that   $\Omega\oplus S_1=X_{S_1}\times \R$ for $X_{S_1}\coloneqq X_P\setminus X_{S_0}$ (which is clearly nonempty).  Since $S_0\cup S_1 = P$ and $P$ is a $\sim_{\Omega,\Uparrow}$ equivalence class, there must exist $s_0\in S_0$ and $s_1\in S_1$ such that \begin{equation}\label{eq:inter1}
        m_1\left(p_1(\Omega+s_0)\cap p_1(\Omega+s_1)\right)>0.
    \end{equation} 
    On the other hand, by our assumption that $\Omega\oplus S_i=X_{S_i}\times \R$ for both $i=0,1$ we have that the set $p_1(\Omega+s_0)\cap p_1(\Omega+s_1)$ is contained in $X_{S_0}\cap X_{S_1}$, which is empty  by construction, but this contradicts \eqref{eq:inter1}. Thus, if $S_0$ is nonempty, $S_1=P\setminus S_0$ must be empty, as stated.
\end{proof}

Using this relation, we can now define the notion of earthquake decomposition of a tiling:

\begin{definition}[Vertical columns  and earthquake decomposition]\label{def:earthquake}
    Let $\Omega\subset\R^2$ be an axis-parallel polygonal tile, $T\in \Tile(\Omega;\R^2)$, and  $P\in \quakeall{T}{\Omega}$.
\begin{itemize}
    \item We say that $P$ is a {\bf \emph{vertical column of $(\Omega,T)$}} if there exists $x\in \R$  such that $P \subseteq \{x\}\times \R$. (In particular, we have $\Omega+ P= X\times \R$, where $X=p_1(\Omega)+x$.)
    We denote by $T_c$ the union of vertical columns of $T$.
    \item We say that $P$ is in $\quake{T}{\Omega}$ if $P\subseteq T\setminus T_c$. 
    \item Finally, the {\bf \emph{(vertical) earthquake decomposition of $T$}} (with respect to $\Omega$) is the decomposition
    $$T=T_{c}\sqcup \left(\bigsqcup_{P\in \quake{T}{\Omega}} P \right).$$
    We say that the earthquake decomposition\footnote{Note that the decomposition is well-defined due to Lemma~\ref{lem:BasicQuake}.} of $T$ is {\bf \emph{non-trivial}} if $T\not\in \quake{T}{\Omega}$.
\end{itemize}
\end{definition}

For example, in Figure \ref{fig:proper_equiv}, the light gray and white classes are vertical columns in $T_c$, while the dark gray class is in $\quake{T}{\Omega}$.

\begin{remark}
    In the arguments below, we freely pass between the sets $\Omega\oplus P$, $\interior(\overline{\Omega}\oplus P)$, $\interior(\overline{\Omega\oplus P})$ and $X_P\times \R$.
    In particular, when talking about the boundary of $\Omega\oplus P$, where $P\in \quakeall{T}{\Omega}$, or $P=T_c$, we always mean the boundary of $X_P\times \R$.
\end{remark}

\begin{definition}[Finite horizontal complexity]\label{def:finhorcom}
    We say that a tiling $T$ of $\R^2$ by an axis-parallel polygonal set $\Omega$ has {\bf \emph{finite horizontal complexity}} if $p_1(T)$ is discrete and has finite local complexity (as a subset of $\mathbb{R}$).
\end{definition}

\begin{remark} 
Note that finite local complexity does not imply finite horizontal complexity.
\end{remark}

 We are now ready to state the main goal of this section: proving the horizontal periodicity of the earthquake (vertical) slabs $\Omega\oplus P$, for each tectonic plate $P\in \quake{T}{\Omega}$ of a tiling $T\in \Tile(\Omega;\R^2)$  with finite horizontal complexity.

\begin{theorem}[Periodicity of slabs of an earthquake component]\label{thm:PeriodicityMain}
    Let $T$ be a tiling of $\mathbb{R}^2$ by an axis-parallel polygonal set $\Omega$ that has finite horizontal complexity and
    $$T=T_{c}\sqcup \left(\bigsqcup_{P\in \quake{T}{\Omega}} P\right)$$
    be the earthquake decomposition of $T$.
    Then there is $\beta>0$ such that the sets
    $$\Omega\oplus T_{c} \text{ and } \Omega\oplus P$$
    are $(\beta,0)\mathbb{Z}$-periodic for every $P\in \quake{T}{\Omega}$.
\end{theorem}

\begin{remark}
    The strength of the assumption that $T$ has finite horizontal complexity is addressed in Section~\ref{sec:finhorcomp}. 
\end{remark}

\begin{remark}\label{rem:PeriodicityMain}
    By tracing the proof of Theorem~\ref{thm:PeriodicityMain} we see that the statement generalizes to tilings of any periodic set of the form $E=X\times \R$ for a regular open periodic $X\subseteq \R$. In fact, it turns out that when a regular open periodic set $X$ is a \emph{strict} subset of $\R$, \emph{all} tilings in $\Tile(\Omega;E)$ have finite horizontal complexity; see Remark \ref{rem:CharInfComplex}.
\end{remark}

The proof consists of four steps that are discussed separately in the next subsections:
\begin{itemize}
    \item First, we derive  favorable properties of some well-chosen rational approximations $\varphi$ of a tiling $T\in \Tile(\Omega;\R^2)$ with finite horizontal complexity.
    \item Then, we show that those well-chosen approximations $\varphi$ preserve  the earthquake decomposition.
    \item Using the discrete analog, Theorem~\ref{thm:DiscretePeriodicity}, of Theorem~\ref{thm:PeriodicityMain},  
we show that the rationally approximated tiling $T^\varphi$ of $\R^2$ by $\Omega^\varphi$ (or, any tiling by an  axis-parallel rational tile) satisfies Theorem~\ref{thm:PeriodicityMain}.
\item In the last step we show how to transform this periodicity of $T^\varphi\in \Tile(\Omega^\varphi;\R^2)$ back to the original tiling $T\in \Tile(\Omega; \R^2)$.
\end{itemize}

\subsection{Finite horizontal complexity}\label{subsec:finhorcomp}

We need a technical statement that connects approximations and finite horizontal complexity (cf. Theorem \ref{thm:GoodFiniteLocalComplexity}  for approximations of tilings with finite local complexity). 

\begin{lemma}\label{lem:ApproximationsOne}
    Let  $T$ be a tiling of $\mathbb{R}^2$ by an axis-parallel polygonal set $\Omega$ that has finite horizontal complexity.
    Then for every $M>0$ there is $\epsilon>0$ such that every $\epsilon$-approximation $\varphi$ for $(\Omega,T)$ satisfies:
    \begin{itemize}
        \item[(i)] For every $v,v'\in V(\Omega)+ T$ such that $p_1(v')>p_1(v)$, we have 
        \begin{equation}\label{eq:uplow}
            1/2 (p_1(v')-p_1(v))<p_1((v')^\varphi)-p_1(v^\varphi)<2(p_1(v')-p_1(v)).
        \end{equation} 
        In particular,  for every $v,v'\in V(\Omega)+T$ we have $p_1(v^\varphi)=p_1((v')^\varphi)$ if and only if $p_1(v)=p_1(v')$.
        \item[(ii)]  If $v_0,v'_0,v_1,v'_1\in V(\Omega)+T$ are such that  
        \begin{equation}\label{eq:iequal}
            p_1((v'_0)^\varphi-(v_0)^\varphi)=p_1((v'_1)^\varphi-(v_1)^\varphi)
        \end{equation}  
        and for $i=0,1$
        \begin{equation}\label{eq:Mbound}
            |p_1((v'_i)^\varphi-(v_i)^\varphi)|< M/2
        \end{equation}
        then
        $$p_1(v'_0-v_0)=p_1(v'_1-v_1).$$
     \end{itemize}
\end{lemma}
\begin{proof}
    We first prove (i).
    As $p_1(T)$ is discrete and $T$ is a tiling of $\R^2$ by $\Omega$, we can write $p_1(V(\Omega)+T)=\{x_i\}_{i\in \mathbb{Z}}$ such that $x_i<x_{i+1}$ for every $i\in \mathbb{Z}$.
    Define $Z=\{x_{i+1}-x_i:i\in \mathbb{Z}\}$.
    Then, as $p_1(T)$ has finite local complexity and $T$ is a tiling of $\R^2$ by $\Omega$, $Z$ is finite.
    Applying Lemma~\ref{lm:TechnicalForFiniteSets} with $$\eta\coloneqq 1/4 \min \{|z|\colon z\in Z\},$$ there is $\epsilon>0$ such that for every $i\in \mathbb{Z}$ and $v,v'\in V(\Omega)+T$ such that $p_1(v)=x_i$ and $p_1(v')=x_{i+1}$,
    we have that
    \begin{equation}\label{eq:uplow1}
        1/2(x_{i+1}-x_i) <p_1((v')^\varphi)-p_1(v^\varphi)<2(x_{i+1}-x_i).
    \end{equation}

    Let $v,v'\in V(\Omega)+T$ and $i<i'\in \mathbb{Z}$ be such that $p_1(v)=x_i$ and $p_1(v')=x_{i'}$.
    Pick any sequence $\{v_k\}_{k=i}^{i'}$ such that $v_i=v$, $v_{i'}=v'$  and $p_1(v_k)=x_k$ for every $i\le k\le i'$.
    Then, by \eqref{eq:uplow1}, we have
    \begin{equation}
        \begin{split}
            p_1((v')^\varphi)-p_1(v^\varphi)= & \ \sum_{k=i}^{i'-1} \left(p_1((v_{k+1})^\varphi)-p_1((v_{k})^\varphi)\right)\\
            > & \ \sum_{k=i}^{i'-1} (x_{k+1}-x_{k})/2\\ = & \ (p_1(v')-p_1(v))/2.
        \end{split}
    \end{equation}
    A similar argument gives the upper bound $p_1((v')^\varphi)-p_1(v^\varphi)< 2(p_1(v')-p_1(v))$. This concludes the proof of part (i) of the lemma.  

    We turn to prove part (ii). 
    Define
    $$Z_M \coloneqq \{p_1(v')-p_1(v):v',v\in V(\Omega)+T, \ |p_1(v')-p_1(v)|<M\}.$$
    If $Z_M=\{0\}$, then part (ii) follows immediately from part (i).
    Hence we may assume that $Z_M\neq\{0\}$.
    As before, by the assumption that $p_1(V(\Omega)+T)$ has finite local complexity, we have that $Z_M$ is finite.
    Thus, there is $\eta'>0$ such that for every $z,z'\in Z_M$ if $|z-z'|<\eta'$, then $z=z'$.
    
    By Lemma~\ref{lm:TechnicalForFiniteSets} applied with $Z\coloneqq Z_M$ and $$\eta\coloneqq 1/4 \min\{\eta',\min\{|z|\colon 0<z\in Z_M\}\}$$ and by the proof of part (i), there is $\epsilon>0$ such that every $\epsilon$-approximation $\varphi$ satisfies part (i),  and so, by \eqref{eq:uplow} and the choice of $\eta'$ for every $z\not=z'\in Z_M$ we have
    \begin{equation}\label{eq:lowereta}
        |\varphi(z)-\varphi(z')|> 1/2|z-z'| \geq \eta'/2 >\eta.
    \end{equation}

    Now, suppose that $v_0,v'_0,v_1,v'_1\in V(\Omega)+T$ are as in the statement.    
    Then, for $i=0,1$, by combining  \eqref{eq:Mbound} and \eqref{eq:uplow}, we obtain that  $p_1(v'_i-v_i)\in Z_M$. Letting   $z\coloneqq p_1(v'_0-v_0)$ and $z'\coloneqq  p_1(v'_1-v_1)$,   condition \eqref{eq:iequal} means that $|\varphi (z)-\varphi (z')|=0$. 
    Then, by \eqref{eq:lowereta}  we must have 
    that $z=z'$, i.e.,   $p_1(v'_0-v_0)=p_1(v'_1-v_1)$, which completes the proof of (ii). 
\end{proof}

\subsection{Earthquake decomposition and approximations}

We show that under the assumption that $T$ has finite horizontal complexity fine  approximations for $(\Omega, T)$ preserve the vertical earthquake decomposition.

\begin{lemma}\label{lem:PreserveEarthquake}
    Let $T$ be a tiling of $\mathbb{R}^2$ by an axis-parallel polygonal set $\Omega$ that has finite horizontal complexity and $$T=T_c\sqcup \left( \bigsqcup_{P\in \quake{T}{\Omega}} P \right)$$ be the earthquake decomposition of $T$ with respect to $\Omega$.
    Then there is $\epsilon>0$ such that  every $\epsilon$-approximation $\varphi$ for $(\Omega,T)$ preserves earthquake; i.e., we have that $T^\varphi\in \Tile(\Omega^\varphi;\R^2)$ and
    $$T^\varphi=(T_{c})^\varphi\sqcup \left(\bigsqcup_{P\in \quake{T}{\Omega}} P^\varphi\right)$$
    is the earthquake decomposition of $T^\varphi$ with respect to $\Omega^\varphi$.
\end{lemma}

\begin{proof}
Let $\epsilon>0$ be small enough to satisfy both  Corollary~\ref{cor:GoodApproxExists} and Lemma~\ref{lem:ApproximationsOne}~(i).
In particular, $\varphi$ is a good approximation and $T^\varphi\in \Tile(\Omega^\varphi;\R^2)$ by Theorem~\ref{thm:GoodApproximation}.

Let $S\subseteq  T$ be a vertical column for $\Omega$, i.e., $S\subseteq \{x\}\times \R$ for some $x\in \R$ and $\Omega+S=X\times \R$ for $X=\interior(\overline{p_1(\Omega)+x})$.
As
$$p_1(s^\varphi)=\varphi(p_1(s))=\varphi(p_1(s'))=p_1((s')^\varphi)$$
for every $s,s'\in S$ it follows  that  $S^\varphi\subseteq \{\varphi(x)\}\times \R$.
By Lemma~\ref{lem:ApproxKeepsCosetPieces}, we have that $\Omega^\varphi\oplus S^\varphi=X'\times \R$ for $X'=\interior(\overline{p_1(\Omega^\varphi)+\varphi(x)})$; thus, $S^\varphi$ is a vertical column for $\Omega^\varphi$.
This implies that $(T_c)^\varphi\subseteq  (T^\varphi)_c$.

To complete the proof we need to show that for every tectonic plate in $\quake{T}{\Omega}$ we have $P^\varphi\in \quake{T^\varphi}{\Omega^\varphi}$.

Recall that $\sim_{\Omega,\Uparrow}$ is the equivalence relation on $T$ generated by the pairs $t_0\not= t_1\in T$ such that  there are horizontal edges $$\omega_0\coloneqq \omega_{v_{0,0},v_{0,1}},\;\omega_1\coloneqq \omega_{v_{1,0},v_{1,1}}\in E(\Omega)$$ such that $\omega_0+t_0$ and $\omega_1+t_1$ \emph{intersect properly}. So, assuming without loss of generality that 
$p_1(v_{i,0})<p_1(v_{i,1})$ for both $i=0,1$, we have 
\begin{equation}\label{eq:intersect}
    p_2(\omega_0+t_0)=p_2(\omega_1+t_1) \ \text{ and } \ p_1(v_{i,0}+t_i)< p_1(v_{i',1}+t_{i'})
\end{equation} for every $(i, i')\in \{(0,1),(1,0)\}$. However, by Lemma~\ref{lem:ApproximationsOne}~(i) and the coordinate-swapped version of Lemma~\ref{lem:step0}, applied to \((\Omega^1,T^1)\), for every $(i, i')\in \{(0,1),(1,0)\}$ \eqref{eq:intersect} holds if and only if the corresponding image conditions
$$p_2((\omega_0)^\varphi+(t_0)^\varphi)=p_2((\omega_1)^\varphi+(t_1)^\varphi) \ \text{ and } \ p_1((v_{i,0})^\varphi+(t_i)^\varphi)<p_1((v_{i',1})^\varphi+(t_{i'})^\varphi)$$
hold. Thus $t_0 \sim_{\Omega,\Uparrow} t_1$ if and only if $(t_0)^\varphi \sim_{\Omega^\varphi ,\Uparrow} (t_1)^\varphi$.
This means that $S\subseteq T$ is in $\quakeall{T}{\Omega}$ if and only if $S^\varphi\subseteq T^\varphi$ is in $\quakeall{T^\varphi}{\Omega^\varphi}$.

By the definition of $\quake{T}{\Omega}$, we have that $P$ is a $\sim_{\Omega,\Uparrow}$-equivalence class of $T$ that satisfies $P\subseteq  T\setminus T_c$.
By the argument above,  we have that $P^\varphi$ is in $\quakeall{T^\varphi}{\Omega^\varphi}$.
Hence, in order to show that $P^\varphi\in \quake{T^\varphi}{\Omega^\varphi}$, it is enough to show that $P^\varphi$ is not a column of $T^\varphi$. 
Note that as $P$ is not a vertical column for $\Omega$, there are $t,t'\in P$ such that $p_1(t)\not=p_1(t')$. 
By Lemma~\ref{lem:ApproximationsOne}~(i), we conclude that $p_1(t^\varphi)\not=p_1((t')^\varphi)$ which shows that $P^\varphi$ is not a vertical column for $\Omega^\varphi$ and the proof is complete.
\end{proof}

\subsection{Periodicity of rational tectonic plates}

Aiming to prove Theorem~\ref{thm:PeriodicityMain}, we first derive a version of it, for axis-parallel sets with \emph{rational vertices},  from its discrete analog, Theorem~\ref{thm:DiscretePeriodicity}.

\begin{lemma}\label{lem:PeriodicityAproximation}
    Let $T$ be a tiling of $\mathbb{R}^2$ by a rational axis-parallel polygonal set $\Omega$.
    Then there is $b>0$ such that the sets
    $$\Omega\oplus T_{c} \text{ and } \Omega\oplus P$$
    are $(b,0)\mathbb{Z}$-periodic for every $P\in \quake{T}{\Omega}$.
\end{lemma}
\begin{proof}
    The claim is immediate if $|\quakeall{T}{\Omega}|=1$, or $T=T_c$, hence we assume that the earthquake decomposition is non-trivial.
    Since $V(\Omega)$ is contained in a coset of $\Q^2$ and is finite, by affine invariance of tilings, we may assume that $V(\Omega)$ contains origin and is a subset of $\mathbb{Z}^2$.
    Since $$\Omega\oplus T_{c} \text{ and } \Omega\oplus P$$
    are $(0,1)\mathbb{R}$-invariant for every $P\in \quake{T}{\Omega}$, 
    any vertical sliding of either $T_c$ or $P$ does not affect the horizontal periodicity of $\Omega\oplus T_{c}$ or $ \Omega\oplus P$.
    In particular, we may assume without loss of generality that there are $P_0\not=P_1\in \quakeall{T}{\Omega}$ such that $p_2(P_0)$ and $p_2(P_1)$ do not lie in the same coset of $\Q$.
    This allows to use \cite[Lemma~2.7]{rational} and slide in the direction $\vec v=(0,1)$ to achieve, without loss of generality, that $T\subseteq  \mathbb{Z}^2$.

Next, we define a finite set $F\subseteq  \mathbb{Z}^2$ to be the collection of $(a,b)\in \mathbb{Z}^2$ such that $(a,a+1)\times (b,b+1)\subseteq  \Omega$.
As we assume that $\Omega$ is axis-parallel and $V(\Omega),T\subseteq  \mathbb{Z}^2$, it follows that $T\in \Tile(F;\mathbb{Z}^2)$, $T^{\operatorname{disc}}_c$  is the union of $(0,1)$-discrete vertical columns in $T$ with respect to $F$ and $P\in \quake{T}{F}$  if and only if $P\in \quake{T}{\Omega}$.
Therefore, by Theorem~\ref{thm:DiscretePeriodicity}, we have that there exists $\ell\in \N\setminus \{0\}$ such that the sets 
$$F\oplus T_c\;\text{ and } F\oplus P$$ are $(\ell,0)\Z$-periodic, for every $P\in \quake{T}{\Omega}$.
Consequently, by setting $b=\ell$, we have that 
$$ \Omega \oplus T_c\;\text{ and } \Omega\oplus P$$
are $(b,0)\Z$-periodic, for every $P\in \quake{T}{\Omega}$, as claimed.
\end{proof}

\subsection{Periodicity of irrational tectonic plates} We are finally ready to prove Theorem~\ref{thm:PeriodicityMain}. The proof consists of applying a \emph{well-chosen} rational approximation $\varphi$ to  $(\Omega, T)$ and  using Lemma~\ref{lem:PeriodicityAproximation} for $(\Omega^\varphi,T^\varphi)$.

\begin{proof}[Proof of Theorem~\ref{thm:PeriodicityMain}]

The claim easily follows when either the earthquake decomposition of $T$ is trivial, or $T=T_c$.
Hence, we may assume that either $|\quake{T}{\Omega}|\ge 1$ and $T_c\not=\emptyset$, or $|\quake{T}{\Omega}|\ge 2$.
As $T$ has finite horizontal complexity, there is a non-empty uniformly discrete sequence $\{x_i\}_{i\in I}\subseteq  \mathbb{R}$, indexed by a countable $I\subseteq \R$ in an increasing order (that the index set $I$ inherits from $\R$), such that $\{\ell_{x_i}\}_{i\in I}$ enumerates the boundary of $\interior(\overline{\Omega\oplus T_c})$ and $\interior(\overline{\Omega\oplus P})$ for $P\in \quake{T}{\Omega}$, where $\ell_{x_i}=\{x_i\}\times \mathbb{R}$.

 First, we show that the order type of the index set $I$ is $\mathbb{Z}$ and that there is $K>0$ such that $x_{i+1}-x_i<K$ for every $i\in \mathbb{Z}$.
Observe that $\{x_i\}_{i\in I}\subseteq  p_1(V(\Omega)+T)$ which will be used throughout the argument.
Let $\varphi$ be an $\epsilon_0$-approximation, where $\epsilon_0>0$ is as in Corollary~\ref{cor:GoodApproxExists}, Lemma~\ref{lem:PreserveEarthquake} and Lemma~\ref{lem:ApproximationsOne}~(i).
In particular, $\varphi$ is a good approximation.
By Lemma~\ref{lem:ApproxKeepsCosetPieces}, we have that $\varphi(x_i)$ is well defined for every $i\in I$, $x_i-x_j$ has the same sign  as $\varphi(x_i)-\varphi(x_j)$ for every $i,j\in I$, and $\{\ell_{\varphi(x_i)}\}_{i\in I}$ coincide with the boundary of $\Omega^\varphi\oplus (T_c)^\varphi$ and $\Omega^\varphi\oplus P^\varphi$ for $P\in \quake{T}{\Omega}$.
It therefore follows from Lemma~\ref{lem:PreserveEarthquake} and Lemma~\ref{lem:PeriodicityAproximation}, applied to $\Omega^\varphi$ and $T^\varphi$, that the order type of the index set $I$ is $\mathbb{Z}$.

Moreover, observe that for every $i\in \mathbb{Z}$ we have that
\begin{equation}\label{eq:TechOne}
    (x_i,x_{i+1})\times \mathbb{R}\subseteq  \Omega\oplus S \ \text{ if and only if } \ (\varphi(x_i),\varphi(x_{i+1}))\times \mathbb{R}\subseteq  \Omega^\varphi\oplus S^\varphi,
\end{equation}
where $S=T_c$ or $S\in \quake{T}{\Omega}$.
It follows easily from Lemma~\ref{lem:PeriodicityAproximation} that $\{\varphi(x_{i})\}_{i\in \Z}\subseteq  \mathbb{R}$ is $b$-periodic for some $b>0$.
As it is also non-empty, there is $K/2$ such that $\varphi(x_{i+1})-\varphi(x_i)<K/2$ for every $i\in \mathbb{Z}$.
By Lemma~\ref{lem:ApproximationsOne} (i), we thus conclude that $x_{i+1}-x_i<K$ for every $i\in \mathbb{Z}$ as desired.

Note that $K>0$ is independent of $\varphi$ as long as $\varphi$ is $\epsilon_0$-approximation.
By choosing a finer approximation, still denoted by $\varphi$, we may additionally ensure that Lemma~\ref{lem:ApproximationsOne}(ii) holds with $M=4K$.
Applying Lemma~\ref{lem:PeriodicityAproximation} again, let $b>0$ be a common horizontal period of the image slabs for this approximation.

 Using the $b\Z$-periodicity of $\{\varphi(x_{i})\}_{i\in \Z}\subseteq  \mathbb{R}$ and  by its construction, there exists $k_0\in \mathbb{N}\setminus \{0\}$ such that both conditions:
\begin{itemize}
    \item [(a)] $\varphi(x_i)+b=\varphi(x_{i+k_0})$,
    \item [(b)] the sets $(\varphi(x_i),\varphi(x_{i+1}))\times \mathbb{R}$ and $(\varphi(x_{i+k_0}),\varphi(x_{i+1+k_0}))\times \mathbb{R}$ are subsets of the same $\Omega^\varphi\oplus S^\varphi$, where $S=T_c$ or $S\in \quake{T}{\Omega}$
\end{itemize}
hold for every $i\in \mathbb{Z}$.
Set $\beta\coloneqq x_{k_0}-x_0$ and observe that $\beta>0$ by (a) and Lemma~\ref{lem:ApproximationsOne}~(i).
To complete the proof, it suffices to show that
\begin{itemize}
    \item [(a')] $x_i +\beta = x_{i+k_0}$,
    \item [(b')] the sets $(x_i,x_{i+1})\times \mathbb{R}$ and $(x_{i+k_0},x_{i+1+k_0})\times \mathbb{R}$ are subsets of the same $\Omega \oplus S $, where $S=T_c$ or $S\in \quake{T}{\Omega}$
\end{itemize}
hold for every $i\in \mathbb{Z}$.

Note that condition (b') holds for every $i\in \mathbb{Z}$ by \eqref{eq:TechOne} and (b).
In addition, condition (a') holds for $i=0$ by the definition of $\beta$.
We claim that if (a') holds for some $i_0\in \mathbb{Z}$, then it holds for $i_0+1$ and $i_0-1$ as well.
Indeed, assume (a') holds for $i_0$ and let $i_{-1}=i_0-1$, then we have
\begin{equation*}
    \begin{split}
        x_{i_{-1}+k_0}-x_{i_{-1}}= & \ (x_{i_0+k_0}-x_{i_0})-((x_{i_0+k_0}-x_{i_0-1+k_0})-(x_{i_0}-x_{i_{-1}}))\\
        = & \ \beta -(x_{i_0+k_0}-x_{i_{-1}+k_0})+(x_{i_0}-x_{i_{-1}}).
    \end{split}
\end{equation*}
By  Lemma~\ref{lem:ApproximationsOne}~(i) and since  $x_{i+1}-x_i<K$ for every $i\in \Z$, we have that 
$$|\varphi(x_{i_0+k_0})-\varphi(x_{i_{-1}+k_0})|,|\varphi(x_{i_0})-\varphi(x_{i_{-1}})|<2K=M/2.$$
From (a), we then obtain that
$$\varphi(x_{i_0+k_0})-\varphi(x_{i_{-1}+k_0})=(\varphi(x_{i_0})+b)-(\varphi(x_{i_{-1}})+b)=\varphi(x_{i_0})-\varphi(x_{i_{-1}}).$$
Hence, by Lemma~\ref{lem:ApproximationsOne}~(ii) and our choice $M=4K$, we conclude that $x_{i_0+k_0}-x_{i_{-1}+k_0}=x_{i_0}-x_{i_{-1}}$ which, in turn, shows that (a') holds for $i_{-1}=i_0-1$. Similarly, one can show that if (a') holds for $i_0$ then it holds for $i_1=i_0+1$.

Thus, knowing that (a') holds for $i_0=0$ we conclude that (a') holds for every $i\in  \Z$ by induction. 
This completes the proof. 
\end{proof}

\section{Non-trivial earthquake decompositions}\label{sec:EarthquakeDecomposition}

In this section, we fix an axis-parallel polygonal set $\Omega$ and a non-empty set $E\subseteq  \R^2$ of the form $E=X\times \R$, where $X\subseteq \R$ is regular open and periodic, and consider the set $\Tile(\Omega;E)$.
Recall that $\Gamma(Z)$ is the group of translational symmetries of $Z\subseteq  \R^d$, and observe that $\Gamma(X)$ is non-trivial and closed, hence, either discrete, that is, of the form $\beta\Z$ for some $\beta>0$, or equal to $\R$.
In particular, $\Gamma(E)=\beta\Z\times \R$ or $\Gamma(E)=\R^2$, and the latter case occurs if and only if $E=\R^2$.

\subsection{Weak convergence}
A set $S\subseteq \R^d$ is \emph{uniformly discrete} if there exists $\delta >0$ such that the distance $\|s'-s\|$ between any two distinct elements $s,s'$ of $S$ is at least $\delta$. We say that $\delta$ is a \emph{separation constant} for $S$ and that $S$ is \emph{$\delta$-separated}.

We say that a sequence of uniformly discrete sets $S_n$ \emph{converges weakly} to $S$, in symbols $S_n\to S$, if for every $R>0$ and $\epsilon >0$, there exists $N$ such that 
$$S_n \cap \mathcal{B}_R \subseteq S+\mathcal{B}_\epsilon \ \text{ and } \ S \cap \mathcal{B}_R \subseteq S_n+\mathcal{B}_\epsilon$$
for every $n\geq N$. One can show, by a standard diagonalization argument, that any given sequence $S_n$ of uniformly discrete sets with a common separation constant $\delta$ has a subsequence $S_{n_k}$ that converges weakly (as $k\to \infty$) to some (possibly empty) set $S$, which is $\delta$-separated as well.

Observe that any $T\in \Tile(\Omega;E)$ is uniformly discrete with a common separation constant $\delta=\delta(\Omega)$. Indeed, if $t,t'\in T$ are distinct, then the packing condition implies that 
$$\|t'-t\|\geq \delta(\Omega)\coloneqq \inf\{ \|d\| \colon m_2(\Omega \cap (\Omega +d)) =0\} >0.$$

\begin{lemma}\label{lem:BasicWeak}
    The set $\Tile(\Omega;E)$ endowed with the weak convergence topology is a compact metric space and the translation action of $\Gamma(E)$ on $\Tile(\Omega;E)$ defined as $(g,T)\mapsto g+T\in \Tile(\Omega;E)$ is well-defined and continuous.
\end{lemma}
\begin{proof}
    It is a standard fact that for every $\delta>0$ the weak topology on the set of $\delta$-separated sets is compact and metrizable.
    Thus, as  $\Tile(\Omega;E)$ is closed in the weak topology in the set of $\delta(\Omega)$-separated sets, the set $\Tile(\Omega;E)$ endowed with the weak convergence topology is a compact metric space.
    The rest of the claim follows by applying the same arguments as in \cite[Section 3]{GL}.
\end{proof}

For $T\in \Tile(\Omega;E)$, we define $\orbit{T}{E}{\Omega}$ as the weak closure of the orbit of $T$ under the action of $\Gamma(E)$.
As $\Tile(\Omega;E)$ endowed with weak topology is metrizable by Lemma~\ref{lem:BasicWeak}, we have that
\begin{equation}\label{def:orbit}
    \orbit{T}{E}{\Omega}\coloneqq \{S \subseteq \R^2 \colon   \text{There exists} \ \{g_n\}_{n\in \N}\subseteq \Gamma(E) \ \text{s.t.} \ T+g_n\to S\}.
\end{equation}

\subsection{Existence of non-trivial earthquakes}

In this section, our aim is to show that infinite local complexity yields the existence of a non-trivial earthquake decomposition for some $S\in \orbit{T}{E}{\Omega}$, where $T$ is any tiling of $E$ by $\Omega$. 

\begin{theorem}\label{thm:MainEarthSplit}
    Let $T$ be a tiling of $E=X\times \R$ by an axis-parallel polygonal set $\Omega$ that does not have finite local complexity.
    Then there is $S\in \orbit{T}{E}{\Omega}$ and a decomposition $S=S_0\sqcup S_1$ such that $S_i\not=\emptyset$ for $i\in \{0,1\}$ and
    \begin{itemize}
        \item if $X\not=\R$, then $\Omega\oplus S_0$ and $\Omega\oplus S_1$ are $\R(0,1)$-invariant,
        \item if $X=\R$, then $\Omega\oplus S_0$ and $\Omega\oplus S_1$ are both either $\R(0,1)$-invariant, or $\R(1,0)$-invariant.
    \end{itemize} 

    Moreover, if $\Lambda\subseteq  \Gamma(E)$ is a lattice, then we may choose $S$ to be a weak limit of a sequence $\{T+\lambda_n\}_{n\in \mathbb{N}}$, where $\lambda_n\in \Lambda$ for every $n\in \mathbb{N}$.
\end{theorem}

We start by detecting the weak limit $S$. 

\begin{lemma}\label{pr:LocalizingEarthquake}
    Let $T$ be a tiling of $E$ by an axis-parallel polygonal set $\Omega$ that does not have finite local complexity.
    Then there are $t\in T$, $\{s_n\}_{n\in \mathbb{N}}\subseteq  T$, $\{g_n\}_{n\in \mathbb{N}}\subseteq  \Gamma(E)$ and $S\in \orbit{T}{E}{\Omega}$ such that
    \begin{enumerate}
        \item[(i)] $t\in T+g_n$ for every $n\in \mathbb{N}$,
        \item[(ii)] $(T+g_n)\to S$,
        \item[(iii)] there is $s\in S$ such that $s\not = s_n+g_n$ for every $n\in \mathbb{N}$ and $(s_n+g_n)\to s$.
    \end{enumerate}
\end{lemma}
\begin{proof}
    Recall that for $m\in \mathbb{N}$, we write $\mathcal{B}_m$ for the open ball of radius $m$ around $(0,0)$.
    By the assumption that $T$ does not have finite local complexity, there is $m\in \mathbb{N}$ such that
    $$\{\mathcal{B}_m\cap (T-t):t\in T\}$$
    is infinite.
    In particular, as $T$ is $\delta(\Omega)$-separated, there is a sequence $\{\mathcal{B}_m\cap (T-t_n)\}_{n\in \mathbb{N}}$ of finite and pairwise distinct sets.
    Consequently, using compactness and relabelling, we may pick $t'_n\in (\mathcal{B}_m+t_n)\cap T$ for every $n\in \mathbb{N}$ such that $\{t'_n-t_n\}_{n\in \mathbb{N}}$ is a convergent sequence in $\R^2$ that consists of pairwise distinct elements.
    We let $b =\lim_{n\to \infty} t'_n-t_n$ and assume without loss of generality that $t'_n-t_n\not= b$ for every $n\in \mathbb{N}$.
    We now divide into two possible cases, as follows:

    {\bf Case 1.}
    If $\Gamma(E)=\mathbb{R}^2$ or, equivalently $E=\mathbb{R}^2$, then we pick any $t\in T$ and define $g_n=t-t_n$ and $s_n=t'_n$ for every $n\in\mathbb{N}$.
    By Lemma~\ref{lem:BasicWeak} there exists  a subsequence $\{g_{n_k}\}_{k\in \N}$ such that $T+g_{n_k}\to S\in \orbit{T}{E}{\Omega}$, which gives (ii).
    By construction, we have  $$t=t_{n_k}+(t-t_{n_k})=t_{n_k}+g_{n_k}\in T+g_{n_k}$$ for every $k\in \N$, which gives (i). Moreover, as 
    $$s_{n_k}+g_{n_k}=t'_{n_k}+g_{n_k}=(t'_{n_k}-t_{n_k})+t$$ for every $k\in \mathbb{N}$, we have that $\{s_n+g_n\}_{n\in \mathbb{N}}$  consists of pairwise distinct elements that converge (to $b+t$), and thus, since $s_{n_k}+g_{n_k}\in T+g_{n_k}$ for every $k\in \N$ and $T+g_{n_k}\to S$ we have that the limit $s\coloneqq b+t$ is in $S$, and $s\not= s_n+g_n$ for every $n\in \mathbb{N}$ by the assumption, giving (iii).

     {\bf Case 2.}
    If $\R^2\not=\Gamma(E)=\Gamma(X)\times \R$, where $E=X\times \R$, then $\Gamma(X)=\beta\Z$ for some $\beta>0$.
    Let $P\subseteq  \partial(E)$ be of the form $\{x\}\times \R$ for some $x\in \partial (X)\subseteq \R$.  Note that $P+(\beta,0)\Z\subseteq \partial(E)$.

     Let $A \subseteq T$  be the set of all elements $a$ of $T$ with the property that there is $m\in \mathbb{Z}$ and a vertical edge $\omega\in E(\Omega)$ such that $\omega+a\subseteq P+(m\beta,0)$. For every $n\in\mathbb{N}$ we define $a_n\in A$ to be such that 
    $$\|a_n-t_n\|=\min \{\|a-t_n\|\colon a\in A \},$$ and let $m_n\in \Z$ and $\omega_n\in E(\Omega)$ be such that $\omega_n+a_n\subseteq P+(m_n\beta,0)$.
    Since $E(\Omega)$ is finite, by the pigeonhole principle, there is $\omega\in E(\Omega)$ such that the set $\{n\in \N\colon \omega_n=\omega\}$ is infinite. Thus, on passing to a subsequence if necessary, we may assume that $\omega_n=\omega$ for every $n\in \N$.

    Observe that, by construction, both sequences $\{a_n-t_n\}_{n\in \mathbb{N}}$ and $\{a_n-t'_n\}_{n\in \mathbb{N}}$ are bounded.
    By the assumption that the map $n\mapsto t'_n-t_n$ is injective, and possibly passing to a subsequence, we can assume that at least one of the sequences $\{a_n-t_n\}_{n\in \mathbb{N}}$ and $\{a_n-t'_n\}_{n\in \mathbb{N}}$ consists of pairwise distinct elements.
    Without loss of generality, assume that it is $\{a_n-t_n\}_{n\in \mathbb{N}}$.

    Define $t\coloneqq a_1$, $s_n\coloneqq t_n$ and $g_n\coloneqq a_1-a_n$ for every $n\in \mathbb{N}$.
    Note that $g_n= a_1-a_n\in \Gamma(E)$ as $p_1(a_1-a_n)=(m_1-m_n)\beta$.

    By Lemma~\ref{lem:BasicWeak}, there exists a subsequence $\{g_{n_k}\}_{k\in \N}$ such that 
     $(T+g_{n_k})\to S\in \orbit{T}{E}{\Omega}$, giving (ii).  Moreover, for every $n\in \N$ $$t=a_n+g_n\in A+g_n\subseteq T+g_n,$$ thus we have (i).  Finally, we have $$s_n+g_n=t_n-a_n+a_1$$ for every $n\in \mathbb{N}$.
    In particular, $\{s_{n_k}+g_{n_k}\}_{k\in \mathbb{N}}$ is a bounded sequence that consists of pairwise distinct elements.
    Thus, after passing to a further subsequence if necessary, by the definition of $S$ and since $s_n\in T$ for every $n\in \N$, there is $s\in S$ such that $s_{n_k}+g_{n_k}\to s$ and we may assume (on removing one element in the sequence if needed) that $s\not=s_n+g_n$ for every $n\in \mathbb{N}$, giving (iii).
    The claim follows.
\end{proof}

Recall that a point $x\in \partial \overline{D}$ is a proper vertex of $D\subseteq  \R^2$ if for every $\epsilon>0$, the intersection  $(\mathcal{B}_\epsilon + x) \cap \partial \overline{D}$ is \emph{not} a line segment.
We denote by $W(D)$ the set of proper vertices of $D$.

The next lemma is a special case of \cite[Theorem~1]{Kenyon}. We include a short proof for completeness.

\begin{lemma}\label{lm:HalfBalls}
    Let $\Omega$ be a polygonal set, $x\in \mathbb{R}^2$ and $U\subseteq \mathbb{R}^2$ be such that $\Omega\oplus U$ is a packing, $x\in \partial(\overline{\Omega\oplus U})$ and $x\in \partial(\Omega+u)$ for every $u\in U$.
    Suppose that there is $\epsilon_0>0$ and a sequence $U_n\subseteq  \mathbb{R}^2$, $n\in \mathbb{N}$, such that $U_n\to U$, $U_n\cap U=\emptyset$, $\Omega\oplus U_n$ is a packing and
    \begin{equation}\label{eq:UUn}
        (\mathcal{B}_{\epsilon_0}+x)\cap (\Omega\oplus U)=(\mathcal{B}_{\epsilon_0}+x)\cap (\Omega\oplus U_n)
    \end{equation}
for every $n\in \N$.    Then $x\not\in W(\Omega\oplus U)$.
\end{lemma} 
\begin{proof}
    First, we show that $x\not \in V(\Omega)+ U_n$ for all large enough $n\in \mathbb{N}$.
    To see this, assume the contrary; i.e., there is an infinite set $Z\subseteq  \mathbb{N}$ such that $x\in V(\Omega)+U_n$ for every $n\in Z$.
    As $V(\Omega)$ is finite, by the pigeonhole principle, there is an infinite set $Z'\subseteq Z$ and $v\in V(\Omega)$ such that  for every $n\in Z'$ we have $v+u_n=x$.
    In particular, $u\coloneqq u_n=u_m$ for every $n,m\in Z'$.
    Consequently, as $U_n\to U$, for every $n\in Z'$ we have $u\in U_n\cap U$, contradicting $U_n\cap U=\emptyset$.
    Thus, on removing finitely many elements of our original sequence $\{U_n\}_{n\in\N}$ if needed, we can assume that $x\not\in V(\Omega)+U_n$ for \emph{every} $n\in \N$.

    Let $b,b'\in \mathcal{B}_{\epsilon_0}+x$ be such that the segments $\omega_{x,b}$ and $\omega_{x,b'}$ are contained in $\partial(\overline{\Omega\oplus U})$. To complete the proof, it suffices to show that $\omega_{x,b}$ and $\omega_{x,b'}$ are parallel.
    
    Observe that by \eqref{eq:UUn}, $\omega_{x,b},\omega_{x,b'}\subseteq \partial(\overline{\Omega\oplus U_n})$ for every $n\in \mathbb{N}$.
    Hence, for every $n\in \N$, there are $u_n,u'_n\in U_n$, $a_n\in \omega_{x,b}$ and $a'_n\in \omega_{x,b'}$, such that 
    \begin{equation}\label{eq:bb'}
        \omega_{x,a_n}\subseteq \omega_{x,b}\cap \partial(\Omega+u_n) \ \text{ and } \ \omega_{x,a'_n}\subseteq  \omega_{x,b'}\cap \partial(\Omega+u'_n).
    \end{equation}
    Since $x\not\in V(\Omega)+U_n$, \eqref{eq:bb'} implies that $m_2((\Omega+u_n)\cap (\Omega+u'_n))$ is positive, for every  $n\in \mathbb{N}$.
    Consequently, due to the packing condition we must have $u_n=u'_n$ for every $n\in \mathbb{N}$.
    Thus, exploiting again the fact that $x\not\in V(\Omega)+u_n$, we must have that $\omega_{x,b}$ and $\omega_{x,b'}$ are parallel.
\end{proof}

\begin{lemma}\label{lem:Diagonalization}
    Let $T$ be a tiling of $E$ by an axis-parallel polygonal set $\Omega$ and $\{g_n\}_{n\in \mathbb{N}}\subseteq  \Gamma(E)$ be such that $T+g_n\to S$.
    Then there is a subsequence $\{n_k\}_{k\in \mathbb{N}}$ such that if $A$ is the collection of $s\in S$ such that $\{n_k\in \mathbb{N}:s \in T + g_{n_k}\}$  is infinite, then
    \begin{itemize}
        \item[(i)] for every $s\in A$  there is $n_s\in \{n_k\}_{k\in \mathbb{N}}$ such that $s \in T + g_{n_k}$ for every $n_k\ge n_s$;
        \item[(ii)]  $W(\Omega\oplus A)$ is empty.
    \end{itemize}
\end{lemma}
\begin{proof}
    Part (i) follows from a standard diagonalization argument. 
    
    Note that  part (ii) is equivalent to the assertion that $W(\Omega\oplus B)$ is empty, where $B\coloneqq S\setminus A$. 
    Let $x\in \partial (\overline{\Omega\oplus B})$.
    Define $$X\coloneqq \{u\in A:x\in \partial(\Omega+u)\}$$ and $$U\coloneqq \{u\in B:x\in \partial(\Omega+u)\}.$$
    For each $u\in U$, let $u_{n_k}\in (T+g_{n_k})$ be such that $u_{n_k}\to u$, and set $$U_{n_k}\coloneqq \{u_{n_k}:u\in U\}.$$
    By the definition of $B$ and the fact that $T$ is uniformly discrete, we have that $U_{n}\cap U=\emptyset$ for all large enough $n\in \{n_k\}_{k\in \N}$.
Since  $X$ is a finite set, by (i), there is $\epsilon_0>0$ and $n_X\in \{n_k\}_{k\in \N}$ such that 
    $$\mathcal{B}(x,\epsilon_0)\cap (\Omega\oplus (X\sqcup U))=\mathcal{B}(x,\epsilon_0)\cap (\Omega\oplus (X\sqcup U_{n_k}))$$
    for every $n_k\ge n_X$, which implies that
    $$\mathcal{B}(x,\epsilon_0)\cap (\Omega\oplus U)=\mathcal{B}(x,\epsilon_0)\cap (\Omega\oplus U_{n_k})$$
    for every $n_k\ge n_X$.
By Lemma~\ref{lm:HalfBalls}, we thus have that $x\not\in W(\Omega\oplus U)$.
    By the definition of $U$, we conclude that $x\not\in W(\Omega\oplus B)$, as claimed.
\end{proof}

\begin{lemma}\label{lem:TechnicalClaim}
    Let $\Omega$ be an axis-parallel polygonal set and $\Omega\oplus S$ be a packing.
    Then we have that $W(\Omega\oplus S)=\emptyset$ if and only if $\Omega\oplus S$ is of the form $X\times \R$ or $\R\times X$ for some $X\subseteq  \R$.
\end{lemma}
\begin{proof}
    As $S$ is $\delta(\Omega)$-separated and $\Omega$ is axis-parallel, we conclude that $\partial(\overline{\Omega\oplus S})$ is a union of line segments that are parallel with the horizontal and vertical axes.
    But since  $\Omega\oplus S$ is a packing and $W(\Omega\oplus S)=\emptyset$, $\partial(\overline{\Omega\oplus S})$ cannot contain both a coset of the horizontal axis and a coset of the vertical axis. The claim then follows.
\end{proof}

\subsection{Proof of Theorem~\ref{thm:MainEarthSplit}} Using the previous lemmas, we are finally ready to conclude the proof of Theorem~\ref{thm:MainEarthSplit}:

\begin{proof}[Proof of Theorem~\ref{thm:MainEarthSplit}]
    Let $T$ be a tiling of $E=X\times \R$ by an axis-parallel polygonal set $\Omega$ that does not have finite local complexity. Then we can find  $t\in T$, $\{s_n\}_{n\in \mathbb{N}}\subseteq  T$, $\{g_n\}_{n\in \mathbb{N}}\subseteq  \Gamma(E)$, $S\in \orbit{T}{E}{\Omega}$ and $s\in S$ that satisfy conditions (i), (ii) and (iii) of Lemma~\ref{pr:LocalizingEarthquake}.
     Additionally,  on passing to a subsequence if necessary, we may assume that the sequence $\{g_n\}_{n\in\N}$ satisfies the conclusion of Lemma~\ref{lem:Diagonalization}.
     Let $S_1$ be the set of all $u\in S$ such that the set $\{n\colon u\in T+g_n\}$ is infinite; and set $S_0=S\setminus S_1$.  Observe that by Lemma~\ref{pr:LocalizingEarthquake}~(i) we have that $t\in S_1$.
    Moreover, by Lemma~\ref{pr:LocalizingEarthquake}~(iii) we have that 
    $$0<\|s-(s_n+g_n)\|=\|(s-g_n)-s_n\|\to 0,$$
    which, as $s_n\in T$ and $T$ is uniformly discrete, implies that $(s-g_n)\not\in T$ for every large enough $n\in \mathbb{N}$. Thus, $s\in S_0$.
    Then, by Lemma~\ref{lem:Diagonalization} we obtain that $W(\Omega\oplus S_1)=\emptyset$. And so, by  Lemma~\ref{lem:TechnicalClaim}, the decomposition $S=S_0\sqcup S_1$   satisfies Theorem~\ref{thm:MainEarthSplit}.

    Now, suppose that $\Gamma(E)$ contains a lattice $\Lambda$. Then there is $R>0$ and a sequence $\{\lambda_n\}_{n\in \N}\subseteq  \Lambda$ such that $\|\lambda_n-g_n\|<R$ for every $n\in \N$. Thus, by compactness (using a standard diagonalization argument), there exists a subsequence $n_k$, $k\in \N$ such that the sequence $\lambda_{n_k}-g_{n_k}$ converges as $k\to \infty$ to some $\gamma\in \Gamma(E)$. Thus, 
    $$T+\lambda_{n_k}=(T+g_{n_k})+(\lambda_{n_k}-g_{n_k})\to S+\gamma.$$ And so, using the partition $S=S_0\sqcup S_1$ as defined above,  we obtain the partition $$(S+\gamma)=(S_0+\gamma)\sqcup (S_1+\gamma)$$ of $S+\gamma$, which satisfies Theorem~\ref{thm:MainEarthSplit}.
\end{proof}

\section{Finite horizontal complexity}\label{sec:finhorcomp}

In this section we discuss the structural strength of the assumption of finite horizontal complexity (Definition~\ref{def:finhorcom}).
We show the following.

\begin{theorem}\label{thm:FinLocalComplexityTrichotomy}
    Let $T$ be a tiling of $\R^2$ by an axis-parallel polygonal set $\Omega$ that does not have finite local complexity.
    Then, after possibly swapping the vertical and horizontal axes, there is $S\in \orbit{T}{\R^2}{\Omega}$ such that $S$ has finite horizontal complexity and the vertical earthquake decomposition of $S$ is non-trivial.
\end{theorem}

The proof of Theorem~\ref{thm:FinLocalComplexityTrichotomy} consists of four steps.
\begin{itemize}
    \item In the first step we provide several equivalent characterizations of the situation when $T$ admits a non-trivial earthquake decomposition but it does not have finite horizontal complexity.
    In particular, we show that in this case we must have $T_c\not=\emptyset$ and $|\quake{T}{\Omega}|=1$.
\item In the second step, we use one of these equivalent conditions to show that there is $S\in \orbit{T}{\R^2}{\Omega}$ that admits a non-trivial \emph{horizontal} earthquake decomposition.
Write $S'$ for the tiling that one gets from $S$ after swapping the horizontal and vertical axes.
\item Now we have that if $S'$ does \emph{not} satisfy  Theorem~\ref{thm:FinLocalComplexityTrichotomy} then, by step one, we must have $(S')_c\not=\emptyset$.
In step three, we show that in this case we must have that $\Omega$ is a product set.
\item Finally, in step four, we show that this is a contradiction as if $\Omega$ is a product set, then $\Omega\oplus T_c$ is $((\beta,0)\mathbb{Z})$-periodic, for some $\beta>0$, which contradicts one of the characterization in step one and thus, contradicts our assumption.
Consequently, this shows that either $T$, or $S'$ satisfies Theorem~\ref{thm:FinLocalComplexityTrichotomy} as required.
\end{itemize}

\subsection{Characterization of finite horizontal complexity}

Let $T$ be a tiling of $\R^2$ by an axis-parallel polygonal set $\Omega$. Observe that if $x\in \R$ is such that $\{x\}\times \R\subseteq  E(\Omega)+T$,  then for every $v\in V(\Omega)$ and $t\in T$ such that $p_1(v+t)> x$ there is $m\geq 1$ and a sequence of vertices $\{a_i\}_{i=1}^{m},\{b_i\}_{i=1}^{m}\subseteq  V(\Omega)$ with $p_1(b_i)> p_1(a_i)$ for every $1\le i\le m$ and 
    \begin{equation}\label{eq:seqvertex}
        p_1(v+t)-x=\sum_{i=1}^{m} p_1(b_i)-p_1(a_i).
    \end{equation}

\begin{lemma}\label{lem:Discrete}
    Let $T$ be a tiling of $\R^2$ by an axis-parallel polygonal set $\Omega$ such that the earthquake decomposition of $T$ is non-trivial.
    Then $p_1(T)$ is discrete.    
\end{lemma}
\begin{proof}
 As $T$ has a non-trivial earthquake decomposition, there is $x\in \R$ such that $\{x\}\times \R\subseteq  \partial (\overline{\Omega\oplus P})$ for some $P\in \quakeall{T}{\Omega}$.
 Note that to prove that $p_1(T)$ is discrete it suffices to show that every $\{x_\ell\}_{\ell\in \mathbb{N}}\subseteq  p_1(T)$ that converges in $\R$ has a constant tail. 
 
 Let  $\{x_\ell\}_{\ell\in \mathbb{N}}\subseteq  p_1(T)$ be a convergent sequence, and let $z\coloneqq \lim_{\ell\to \infty} x_\ell$. Without loss of generality, we may assume that there is $v\in V(\Omega)$ such that  $x\le p_1(v)+x_\ell$ for every $\ell\in \mathbb{N}$. 
 
 For every $\ell\in \N$ such that $x\neq  p_1(v)+x_\ell$, by \eqref{eq:seqvertex}, we can find finite sequences of vertices $\{a^\ell_i\}_{i=1}^{m_\ell},\{b^\ell_i\}_{i=1}^{m_\ell}\subseteq  V(\Omega)$ such that $p_1(b^\ell_i)>p_1(a^\ell_i)$ for every $1\le i\le m_\ell$ and 
    \begin{equation}\label{eq:positiveSum}
        p_1(v)+x_\ell-x=\sum_{i=1}^{m_\ell} p_1(b^\ell_i)-p_1(a^\ell_i).
    \end{equation}
   If $\ell\in \N$ such that $x =  p_1(v)+x_\ell$, then we set $m_\ell=1$ and pick \emph{identical} vertices $a_1^\ell,b_1^\ell\in V(\Omega)$, so \eqref{eq:positiveSum} is satisfied in this case as well. 
   
    We claim that the sequence $\{m_\ell\}_{\ell\in\N}$ is uniformly bounded.
    Indeed, let $R>0$ be such that  $$p_1(v)+z-x+\diam(\Omega)\le R.$$ Then,
    as $x_\ell\to z$, there is $\ell_0\in \mathbb{N}$ such that $p_1(v)+x_\ell-x\le R$ for every $\ell\geq \ell_0$.
    Unless $x =  p_1(v)+x_\ell$, the right-hand side of \eqref{eq:positiveSum} is a sum of strictly positive numbers  picked from the finite  set $p_1(V(\Omega)-V(\Omega))$, and the claim follows.
    
    Finally, since $\{m_\ell\}_{\ell\in\N}$ is uniformly bounded and $\{x_\ell\}_{\ell\in \N}$ converges, we have that $x_\ell=z$ for every large enough $\ell\in \N$, as needed.
\end{proof}

\begin{lemma}\label{lem:ColumnOrUnbounded}
    Let $T$ be a tiling of $\R^2$ by an axis-parallel polygonal set $\Omega$ and $P\in \quakeall{T}{\Omega}$ be such that there is $R>0$ with the property that
    $$\text{ either } \ p_1(\Omega\oplus P)\cap (-\infty,-R)=\emptyset \ \text{ or } \ p_1(\Omega\oplus P)\cap (R,\infty)=\emptyset.$$
    Then $P\subseteq  T_c$. 
\end{lemma}
\begin{proof}
    Let $P\in \quakeall{T}{\Omega}$ and assume, without loss of generality,that $$p_1(\Omega\oplus P)\cap (-\infty,-R)=\emptyset.$$
    Then there is $x\in p_1(P)$ such that $x\le p_1(s)$ for every $s\in P$, and $y\in \R$ such that $\{y\}\times \R\subseteq  \partial(\overline{\Omega\oplus P})$ and $y\le z$ for every $z\in p_1(\Omega\oplus P)$. To complete the proof, it suffices to show that $S=P\cap (\{x\}\times \R)$ is a vertical column, which also implies that $P=S$.
    
    We claim that there are $\epsilon,\eta>0$ such that for every good $\epsilon$-approximation $\varphi$ for $(\Omega,T)$ and every $s\in P$ we have that if $x<p_1(s)$, then $\varphi(x)+\eta<p_1(s^\varphi)$.
    Indeed, this follows from Lemma~\ref{lm:TechnicalForFiniteSets} combined with \eqref{eq:seqvertex}, as 
    $$p_1(s)-x=\sum_{i=1}^m p_1(b_i)-p_1(a_i)$$
    for a finite sequences of vertices $\{a_i\}_{i=1}^m,\{b_i\}_{i=1}^m\subseteq  V(\Omega)$ with $p_1(b_i)>p_1(a_i)$ for every $1\le i\le m$.

    Let $\varphi$ be such an approximation.
    By Lemma~\ref{lem:ApproxKeepsCosetPieces}, we have that $\Omega^\varphi\oplus P^\varphi$ is $\R(0,1)$-invariant, $\{\varphi(y)\}\times \R\subseteq \partial(\overline{\Omega^\varphi\oplus P^\varphi})$ and $\varphi(y)\le z$ for every $z\in p_1(\Omega^\varphi\oplus P^\varphi)$.
    By Lemma~\ref{lem:PeriodicityAproximation}, we must have that $P^\varphi$ is a union of vertical columns of $T^\varphi$.
    In particular, $S^\varphi$ is a vertical column of $P^\varphi$ as
    $$S^\varphi=P^\varphi\cap (\{\varphi(x)\}\times \R)$$
    is non-empty and $\varphi(x)+\eta<p_1(s^\varphi)$ for every $s\in P$ such that $x<p_1(s)$.

    The fact that $S^\varphi$ is a vertical column implies that the vertical slices
    $$\Omega^\varphi\cap (\{w\}\times \R)$$
    have the same measure for every $w\in p_1(\Omega^\varphi)\setminus p_1(V(\Omega^\varphi))$.
    As this holds true for every good $\epsilon$-approximation $\varphi$, we conclude that all the vertical slices
    $$\Omega \cap (\{w\}\times \R)$$
    have the same measure for every $w \in p_1(\Omega)\setminus p_1(V(\Omega))$.
    As $\Omega$ is axis-parallel polygonal set, $P\in \quakeall{T}{\Omega}$ and $x+\eta\le p_1(s)$ for every $s\in P\setminus S$, there is $y<y'\in \R$ such that
    $$(y,y')\times \R\subseteq  \Omega\oplus S.$$
    Combined with the fact that the vertical slices of $\Omega$ have the same measure, we conclude by a standard volume argument that $S$ is a vertical column of $T$ as desired.   
\end{proof}

\begin{lemma}[There are finitely many plates that are not vertical columns]\label{lem:FinNumberPlatesSecondVersion}
    Let $T$ be a tiling of $\R^2$ by an axis-parallel polygonal set $\Omega$.
    Then for every $P\in \quake{T}{\Omega}$, the intersection of $p_1(\Omega\oplus P)$ with every interval of length at least $8\diam(p_1(\Omega))$ has positive measure.
\end{lemma}

\begin{proof}
        Let $\alpha\coloneqq 2\diam(p_1(\Omega))$. Note that, by the translation invariance of tilings and the $\R(0,1)$-invariance of $\Omega\oplus P$ for every $P\in \quake{T}{\Omega}$, it suffices to show  that
        \begin{equation}\label{eq:BigIntersection}
            m_1\left(p_1(\Omega\oplus P)\cap (-2\alpha,2\alpha)\right)>1/2 m_1(p_1(\Omega))>0
        \end{equation}
        for every $P\in \quake{T}{\Omega}$. Suppose that this is not the case, i.e.,  there exists $P_0\in \quake{T}{\Omega}$ such that $$m_1(p_1(\Omega\oplus P_0)\cap (-2\alpha,2\alpha))\leq 1/2 m_1(p_1(\Omega)).$$ Then, by the choice of $\alpha$, we have that the intersection $p_1(\Omega\oplus P_0)\cap (-\alpha,\alpha)$ is empty. Then, we have a partition $P_0=S_+\sqcup S_-$ of $P_0$, where 
        $$S_+\coloneqq \{t\in P_0:(\Omega+t)\subseteq  [\alpha,\infty)\times \R\}$$
        and 
        $$S_-\coloneqq \{t\in P_0:(\Omega+t)\subseteq  (-\infty,-\alpha]\times \R \}.$$
        By Lemma~\ref{lem:ColumnOrUnbounded}, as $P_0$ is nonempty, neither $S_+$ nor $S_-$ is empty. However, by the choice of $\alpha$ and Lemma~\ref{lem:BasicQuake}~(i) we have that $\Omega\oplus S_i$ is $\R(0,1)$-invariant for every $i\in \{+,-\}$, which contradicts Lemma~\ref{lem:BasicQuake}~(ii).
        This gives \eqref{eq:BigIntersection}, which concludes the proof.
\end{proof}

As an immediate corollary of Lemma~\ref{lem:FinNumberPlatesSecondVersion}, we obtain:

\begin{corollary}\label{cor:FinNumberPlates}
    Let $T$ be a tiling of $\R^2$ by an axis-parallel polygonal set $\Omega$.
    Then the set $\quake{T}{\Omega}$ is finite.
\end{corollary}

For a tiling $T$ that admits a non-trivial earthquake decomposition, we define the {\bf \emph{boundary set}} $\mathcal{L}(T)\subseteq \R$ to consist of those $x\in \R$ such that there is either $P\in \quake{T}{\Omega}$ or a vertical column $P\subseteq T_c$ such that $\{x\}\times \R\subseteq \partial(\overline{\Omega+P})$.
Observe that there is $r>0$, that only depends on $\Omega$, such that if $x<x'\in \mathcal{L}(T)$, then $r<x'-x$.

\begin{lemma}\label{lem:BoundedImpliesFiniteHorizontal}
    Let $T$ be a tiling of $\R^2$ by an axis-parallel polygonal set $\Omega$ such that the earthquake decomposition of $T$ is non-trivial and the boundary set $\mathcal{L}(T)$ admits an increasing enumeration $\{x_k\}_{k\in \mathbb{Z}}$ indexed by $\Z$. Suppose that $R>0$ is such that $x_{k+1}-x_k<R$ for every $k\in \mathbb{Z}$.
    Then $T$ has finite horizontal complexity.        
\end{lemma}
\begin{proof}
    By Lemma~\ref{lem:Discrete}, we can enumerate the set $p_1(T)$ as $\{z_i\}_{i\in \mathbb{Z}}$ such that $z_i<z_{i+1}$ for every $i\in \mathbb{Z}$.
    The claim is easily seen to be equivalent with the assertion that the set
    \begin{equation}\label{eq:FinSet}
        \{z_{i+1}-z_i:i\in \mathbb{Z}\}
    \end{equation}
    is finite.
    For $i\in \mathbb{Z}$ and $v\in V(\Omega)$,
    let $k\in \mathbb{Z}$ be such that
    $$x_k<p_1(v)+z_i\; \text{ and }\; p_1(v)+z_{i+1}<x_k+2R+2\diam(\Omega).$$
    By \eqref{eq:seqvertex}, there are finite sequences of vertices
    $$\{a_\ell\}_{\ell=1}^m,\; \{b_\ell\}_{\ell=1}^m, \; \{c_j\}_{j=1}^{m'},\; \{d_j\}_{j=1}^{m'}\subseteq  V(\Omega)$$
    such that $p_1(b_\ell)>p_1(a_\ell)$ for every $1\le \ell\le m$ and $p_1(d_j)>p_1(c_j)$ for every $1\le j\le m'$ and both 
    $$p_1(v)+z_i-x_k=\sum_{\ell=1}^m p_1(b_\ell)-p_1(a_\ell) $$
    and $$p_1(v)+z_{i+1}-x_k=\sum_{j=1}^{m'} p_1(d_j)-p_1(c_j)$$ are satisfied.
    Note that we have $m,m'<M$ for some $M$ that only depends on $\Omega$ and $R$, but not on $i\in \mathbb{Z}$.
    It follows that
    $$z_{i+1}-z_i=\left(\sum_{j=1}^{m'} p_1(d_j)-p_1(c_j)\right)-\left(\sum_{\ell=1}^m p_1(b_\ell)-p_1(a_\ell)\right),$$
    where all the summands on the right hand side belong to the finite set $p_1(V(\Omega)-V(\Omega))$. This implies that \eqref{eq:FinSet} is finite as desired.
\end{proof}

\begin{lemma}[Characterization of infinite horizontal complexity]\label{lem:CharInfComplex}
    Let $T$ be a tiling of $\R^2$ by an axis-parallel polygonal set $\Omega$ with a non-trivial earthquake decomposition.
    Then the following are equivalent:
    \begin{enumerate}
        \item $T$ does not have finite horizontal complexity.
        \item $p_1(\Omega\oplus T_c)$ is non-empty and not periodic.
        \item $|\quake{T}{\Omega}|=1$ and for every $R>0$ there is $x\in \R$ such that
        \begin{equation}\label{eq:charinfcomp}
            [x-R,x+R]\subseteq  p_1(\Omega\oplus P)
        \end{equation}
        for the only $P\in \quake{T}{\Omega}$.
        \item There are horizontal edges $\omega,\omega'\in E(\Omega)$ and $\{t_n\}_{n\in \mathbb{N}}, \{s_n\}_{n\in \mathbb{N}}\subseteq  T$ such that the set
        \begin{equation}\label{eq:infiniteIntersectionTwo}
            \{\omega\cap (\omega'+s_n-t_n)\colon n\in \mathbb{N}\}
        \end{equation}
        is infinite. 
    \end{enumerate}
\end{lemma}

\begin{remark}\label{rem:CharInfComplex}
    The proof of Lemma \ref{lem:CharInfComplex} goes through for tilings of any periodic set of the form $E=X\times \R$ for a regular open periodic set $X\subseteq \R$ (where the set $[x-R,x+R]$ in \eqref{eq:charinfcomp} should be replaced with $[x-R,x+R]\cap X$ for more general statement). In particular, when $X$ is a proper subset of $\R$, we obtain that \emph{any} tiling in $\Tile(\Omega;E)$ has finite horizontal complexity, since one can show that in this case the condition (4) in Lemma~\ref{lem:CharInfComplex} cannot be satisfied.
\end{remark}

\begin{proof}
    We start by showing that (4) implies (1).
    To that end, first note that the set $p_1(\mathcal{B}_r)\cap p_1(T-t_n)$ is finite for every $r>0$ and $n\in \mathbb{N}$, as $p_1(T)$ is discrete by Lemma~\ref{lem:Discrete}.
    Taking $r>\diam(\Omega)$, we see that after possibly passing to a subsequence we have that $p_1(s_n-t_n)\in p_1(\mathcal{B}_r)\cap p_1(T-t_n)$ are pairwise distinct.
    By (4), we have that $\{p_1(s_n-t_n)\}_{n\in \mathbb{N}}$ is infinite, which shows (1).

    Next, we show that (1) implies (2).
    First, we show that $T_c\not=\emptyset$.
    If this is not the case, then, as we assume that the earthquake decomposition is non-trivial, we have $|\quake{T}{\Omega}|\ge 2$.
    Note that this implies, by Lemma~\ref{lem:FinNumberPlatesSecondVersion}, that the assumption of Lemma~\ref{lem:BoundedImpliesFiniteHorizontal} are satisfied and that is a contradiction with (1), hence $T_c\not=\emptyset$.
    Note, however, that if $p_1(\Omega\oplus T_c)$ is periodic, then $\mathcal{L}(T)$ must be relatively dense in $\R$, that is, there is $Q>0$ such that $\mathcal{B}_Q+\mathcal{L}(T)=\R$.
    Consequently, the assumptions of Lemma~\ref{lem:BoundedImpliesFiniteHorizontal} are satisfied and that contradicts (1). 
    Hence, $p_1(\Omega\oplus T_c)$ is non-periodic, giving (2).

    Next, we show that (2) implies (3).
    Suppose that (3) does not hold.
    As shown in the previous paragraph, this implies that the assumptions of Lemma~\ref{lem:BoundedImpliesFiniteHorizontal} are satisfied.
    Hence $T$ has finite horizontal complexity and by Theorem~\ref{thm:PeriodicityMain} we have that $p_1(\Omega\oplus T_c)$ is periodic (possibly empty), hence (2) does not hold.
    This shows that (2) implies (3).

    It remains to check that (3) implies (4), or equivalently, that the negation of (4) implies the negation of (3).
    The negation of (4) states that the set
    $$\{\omega\cap (\omega'+s-t):s,t\in T \ \omega,\omega' \text{ horizontal edges}\}$$
    is finite.
    Combined with Lemma~\ref{lm:TechnicalForFiniteSets} and Corollary~\ref{cor:GoodApproxExists}, we conclude that there is $\epsilon_0>0$ such that for every $0<\epsilon<\epsilon_0$ any $\epsilon$-approximation $\varphi$ for $(\Omega,T)$ is good and $t\sim_{\Omega,\Uparrow} s$ if and only if $t^\varphi \sim_{\Omega^\varphi,\Uparrow} s^\varphi$.
    In particular,
    $$T^\varphi=(T_c)^\varphi\sqcup \left(\bigsqcup_{P\in \quake{T}{\Omega}} P^\varphi\right)$$
    is the earthquake decomposition of $T^\varphi$ and it is non-trivial.
    Thus, by Lemma~\ref{lem:PeriodicityAproximation}, we have that there is $R'>0$ such that $p_1(\Omega^\varphi\oplus P^\varphi)$ does not contain any interval of length $R'$ for every $P\in \quake{T}{\Omega}$.

    Now, suppose for a contradiction that (3) holds.
    Let  $P$ be the unique element in $\quake{T}{\Omega}$ and for every $R>0$,  let $x,y\in \R$ be such that $y-x>R$ and $\{x,y\}\times \R\subseteq  \partial(\overline{\Omega\oplus P})$ and $(x,y)\times \R\subseteq  \Omega\oplus P$.
    Note that such $x$ and $y$ exist as any increasing enumeration of $\mathcal{L}(T)$ must have order type $\Z$ as the same is true for $\mathcal{L}(T^\varphi)$.
    By Lemma~\ref{lem:ApproxKeepsCosetPieces}, we have that $\varphi(x),\varphi(y)$ are well defined and
    \begin{equation}\label{eq:ContainIntervals}
        (\varphi(x),\varphi(y))\times \R\subseteq  \Omega^\varphi\oplus P^\varphi.
    \end{equation}
    As $T$ is a tiling of $\R^2$ by an axis-parallel polygonal set $\Omega$, we find a finite sequence of vertices $\{a_i\}_{i=1}^m,\{b_i\}_{i=1}^m\subseteq  V(\Omega)$ such that $p_1(b_i)>p_1(a_i)$ for every $1\le i \le m$ and
    $$y-x=\sum_{i=1}^m p_1(b_i)-p_1(a_i).$$
    As $\varphi$ is a good approximation, there is $\eta>0$, independent of $x,y\in \R$, such that
    $$\varphi(y)-\varphi(x)=\sum_{i=1}^m p_1(b^\varphi_i)-p_1(a^\varphi_i)\ge m\eta.$$
    That contradicts \eqref{eq:ContainIntervals} as we may take
    $$R=R'\diam(p_1(\Omega))/\eta$$
    which gives that $\varphi(y)-\varphi(x)\ge (R'/\eta)\eta=R'$ which is not possible due to the choice of $R'$.
\end{proof}

\subsection{Horizontal earthquake decomposition}

Let $T$ be a tiling of $\R^2$ by an axis-parallel polygonal set $\Omega$.
Analogously to Definition~\ref{def:earthquake} and only in this section, we define the notion of a \emph{horizontal column} of $(\Omega,T)$ and the \emph{horizontal earthquake decomposition} of $T$ (with respect to $\Omega$) that we denote as 
$$T=T^h_c\sqcup \left(\bigsqcup_{Q\in \quake{T}{\Omega}^h} Q\right).$$
We say that the horizontal earthquake decomposition is \emph{non-trivial} if $T\not\in \quakeall{T}{\Omega}^h$.
Similarly, we say that $S$ has \emph{finite vertical complexity} if $p_2(S)\subseteq  \R$ is discrete and has finite local complexity.
All the results derived in previous sections for earthquake decomposition and finite horizontal complexity can be applied to horizontal earthquake decomposition and finite vertical complexity by conjugating with a linear transformation that swaps the horizontal and vertical axes.

The main goal of this section is to apply Lemma~\ref{lem:CharInfComplex} to show that if $T$ has a non-trivial earthquake decomposition and does not have finite horizontal complexity then there is $S\in \orbit{T}{\R^2}{\Omega}$ that  admits a non-trivial horizontal earthquake decomposition.

\begin{lemma}\label{lem:Options}
    Let $T$ be a tiling of $\R^2$ by an axis-parallel polygonal set $\Omega$ such that the earthquake decomposition of $T$ is non-trivial and $T$ does not have finite horizontal complexity.
    Then there is $S\in\orbit{T}{\R^2}{\Omega}$ and the horizontal earthquake decomposition of $S$ is non-trivial.
    In addition, one of the following holds: 
    \begin{itemize}
        \item [(a)] $S$ has finite vertical complexity. 
        \item [(b)] There is a sequence $\{g_n\}_{n\in \mathbb{N}}\subseteq  \R^2$ such that $T+g_n\to S$ and  there are $t\in T$, $\{t_n\}_{n\in \N}\subseteq T$ such that $g_n=t-t_n$ for every $n\in \N$. Moreover, the set $C=S\cap (\R\times \{p_2(t)\})$ is a horizontal column and for every $s\in C$ there is $n_s\in \mathbb{N}$ such that $s-t+t_n\in T$ for every $n\ge n_s$.
    \end{itemize}
\end{lemma}

\begin{proof}
    By item (4) in Lemma~\ref{lem:CharInfComplex}, there are horizontal edges $\omega,\omega'\in E(\Omega)$ and
    $$\{t_n\}_{n\in \mathbb{N}}, \{t'_n\}_{n\in \mathbb{N}}\subseteq  T$$
    such that the set
        \begin{equation*}
            \{\omega\cap (\omega'+t'_n-t_n)\colon n\in \mathbb{N}\}
        \end{equation*}
    is infinite.
    Thus, on passing to subsequences if needed, we may assume without loss of generality that
    \begin{equation}\label{eq:infiniteIntersectionTwoTwo}
        (\omega+t_n)\cap (\omega'+t'_n) \text{ properly intersect for every } n\in \mathbb{N}.
    \end{equation}
    
    Passing to a further subsequence, we may assume that $d_n=t'_n-t_n$ are pairwise distinct and converge to $d$, with $d_n\neq d$ for every $n$.
    Set $t=t_0$, $t'=t'_0$, $g_n=t-t_n$, $h_n=t'-t'_n$.
    After possibly passing to subsequences, we may assume that $T+g_n\to S\in \orbit{T}{\R^2}{\Omega}$ and $T+h_n\to S'\in \orbit{T}{\R^2}{\Omega}$.
    Moreover, by Lemma~\ref{lem:Diagonalization}, we may assume (after possibly passing to a subsequence) that if $s\in S$ is such that $\{n\in \mathbb{N}:s-g_n\in T\}$ is infinite, then there is $n_s\in \mathbb{N}$ such that $s-g_n\in T$ for every $n\ge n_s$, and similarly with the sequence $\{h_n\}_{n\in \N}$ for every $s'\in S'$.

    Define $A\subseteq S$ to be the set of all $s\in S$ for which there exists $n_s\in \mathbb{N}$ such that $s-g_n\in T$ for all $n\geq n_s$ and, similarly, $A'\subseteq S'$ to be the set of all $s'\in S'$ for which there exists $n_{s'}\in \mathbb{N}$ such that $s'-h_n\in T$ for all $n\geq n_{s'}$. 
    Observe that $t\in A$ and $t'\in A'$.
    Then, setting  $u=\lim_{n\to\infty} t'_n+g_n$ and $u'=\lim_{n\to\infty} t_n+h_n$,  we have that $u\not\in A$ and $u'\not \in A'$.
    Moreover, by Lemma~\ref{lem:Diagonalization}~(ii) we have that $W(\Omega\oplus A)=\emptyset$ and $W(\Omega\oplus A')=\emptyset$.

    We claim that $\Omega\oplus A$ is a union of cosets of the horizontal axis.
    Indeed, let $v\in V(\Omega)$, $R=4\diam(\Omega)$ and define
    $$A_R=\{s\in A:(\Omega+s)\cap (\mathcal{B}_R+v+t) \not=\emptyset\}.$$
    Observe that $t\in A_R$ and $u\in (\mathcal{B}_R+v+t)\setminus A_R$.
    As $A_R$ is finite and by the definition of $A$, there is $n_0\in \mathbb{N}$ such that $n_0\ge n_s$ for every $s\in A_R$.
    As $W(\Omega\oplus A)=\emptyset$, we have that
    $$(\Omega\oplus A_R)\cap (\mathcal{B}_R+v+t)=\left((\Omega\oplus A_R)+\mathbb{R}\vec d\right)\cap (\mathcal{B}_R+v+t),$$
    for some $\vec d\in \{(0,1),(1,0)\}$.
    Shifting by $-g_n$, we get that 
    $$(\Omega\oplus (A_R-g_n))\cap (\mathcal{B}_R+v+t_n)=\left((\Omega\oplus (A_R-g_n))+\mathbb{R}\vec d\right)\cap (\mathcal{B}_R+v+t_n)$$
    for every $n\ge n_0$.
    Note that for all sufficiently large $n$, we have $t'_n\notin A_R-g_n$, since $t'_n+g_n\to u\notin A_R$ and $A_R$ is finite.
    By \eqref{eq:infiniteIntersectionTwoTwo}, this implies that the set
    $$((\Omega+t_n)+\mathbb{R}(0,1))\cap (\mathcal{B}_R+v+t_n)\cap (\Omega+t'_n)$$
    has positive measure for every $n\in \mathbb{N}$.
    Hence we must have $\vec d=(1,0)$ and the claim follows.
    An identical argument gives that $\Omega\oplus A'$ is a union of cosets of the horizontal axis as well.
    
    We have just proven that the horizontal earthquake decomposition of both $S$ and $S'$ is non-trivial.
    If in addition $S$ or $S'$ has finite vertical complexity, then (a) holds and we are done.
    Otherwise, we show that (b) holds either for $t$ and $\{t_n\}_{n\in \mathbb{N}}$, or for $t'$ and $\{t'_n\}_{n\in \mathbb{N}}$.
    
    Let $w=t'-u$. Then, as 
     $$T+g_n+w=T+(t'-t'_n)-(u-(t'_n+g_n))\to S',$$
     we have that $S+w=S'$. Moreover, 
    we claim that $(A+w)\cap A'=\emptyset$.
    Indeed, first note that
    $$t+w=t+t'-u=u'\notin A',$$
    and thus $t+w\not \in A'$.
    Let $s\in A$ and suppose for a contradiction that there is $m_s\in \mathbb{N}$ such that $s+w-h_n\in T$ for every $n\ge m_s$.
    By the definition of $A$, we have that $s-g_n\in T$ for every $n\ge n_s$. Then, since  $(s-g_n),(s+w-h_n)\in T$ for every $n\ge n_s+m_s$ and as $$-(s-g_n)+(s+w-h_n)=g_n-h_n+w\to 0,$$ we conclude that  
    $$s+w-h_n= s-g_n$$
    for all sufficiently large $n$.
    This gives that 
    $$u'-h_n=t+w-h_n=t-g_n=t_n\in T$$
    for all sufficiently large $n$, which contradicts $u'\not\in A'$.

    Finally, we have that $\Omega\oplus (A+w)$ and $\Omega\oplus A'$ are disjoint $\R(1,0)$-invariant sets.
    As we assume that $S'$ does not have a finite vertical complexity, we conclude by Lemma~\ref{lem:CharInfComplex} that
    $$\text{ either } (\R\times \{p_2(t+w)\})\cap S' \text{ or } (\R\times \{p_2(t')\})\cap S'$$
    is a horizontal column of $(\Omega,S')$.
    Note that the former case implies that $$(\R\times \{p_2(t)\})\cap S$$ is a horizontal column of $(\Omega,S)$ as $S+w=S'$.
    Hence, in the former case $t$ and $\{t_n\}_{n\in \mathbb{N}}$ work as required while in the latter we can take $t'$ and $\{t'_n\}_{n\in \mathbb{N}}$. 
\end{proof}

\begin{remark}
    In the following subsection we show that option (b) in Lemma~\ref{lem:Options} implies that $\Omega$ is a \emph{product set}.
    It is easy to check, see Lemma~\ref{lem:ProductTwo}, that in that case any tiling $T$ of $\mathbb{R}^2$ by $\Omega$ that has non-trivial earthquake decomposition has finite horizontal complexity.
    In particular, this implies that option (b) in Lemma~\ref{lem:Options} cannot occur.
\end{remark}

\subsection{Simultaneous vertical and horizontal columns}

It might be tempting to think that option (b) in Lemma~\ref{lem:Options} directly implies the existence of $S\in \orbit{T}{\R^2}{\Omega}$ that contains both vertical and horizontal columns.
Unfortunately, as $\Omega\oplus T_c$ is not periodic, it is difficult to control the existence of vertical columns when passing to the limit as in Lemma~\ref{lem:Options} (b).
To circumvent these issues, we show that (b) in Lemma~\ref{lem:Options} implies that for every good approximation $\varphi$ there is a tiling $S(\varphi)$ of $\R^2$ by $\Omega^\varphi$ that contains both vertical and horizontal columns.
This is then used in the next subsection to show that $\Omega$ is a product set.

\begin{lemma}\label{lem:SimultanousColumns}
    Let $T$ be a tiling of $\R^2$ by an axis-parallel polygonal set $\Omega$ such that the earthquake decomposition of $T$ is non-trivial and $T$ does not have finite horizontal complexity.
    Suppose that there exist $t\in T$ and $\{t_n\}_{n\in \N}\subseteq  T$ such that 
    \begin{itemize}
        \item the sequence $T+(t-t_n)$ converges to $S\in \orbit{T}{\R^2}{\Omega}$;
        \item the set $C\coloneqq S\cap (\R\times \{p_2(t)\})$ is a horizontal column for $(\Omega,S)$;
        \item for every $s\in C$ there is $n_s\in \mathbb{N}$ such that $s-t+t_n\in T$ for every $n\ge n_s$.
    \end{itemize}
    Then for every good approximation  $\varphi$ for $(\Omega,T)$ there is $S(\varphi)\in \orbit{T^\varphi}{\R^2}{\Omega^\varphi}$ that contains both vertical and horizontal columns.   
\end{lemma}
\begin{proof}
    By Theorem~\ref{thm:GoodApproximation}, we have that $T^\varphi\in \Tile(\Omega^\varphi ;\R^2)$.
    Setting $\ell_n\coloneqq t^\varphi-(t_n)^\varphi$, 
    by Lemma~\ref{lem:BasicWeak} (after possibly passing to a subsequence) we may assume that the sequence $T^\varphi+\ell_n$ converges to $S(\varphi)\in \orbit{T^\varphi}{\R^2}{\Omega^\varphi}$.
    We show that $S(\varphi)$ contains both vertical and horizontal columns, as required.

    Note that $t^\varphi=(t_n)^\varphi+\ell_n\in T^\varphi+\ell_n$ for every $n\in \mathbb{N}$, hence, $t^\varphi\in S(\varphi)$.
    For $n\in \N$ we set $g_n\coloneqq t-t_n$. Then, for $s\in C$, we have that $s-g_n\in T$, $$p_2(s-g_n)=p_2(t_n) \text{ and } p_1((s-g_n)-t_n)=p_1((s-g_m)-t_m)$$ for every $n,m\ge n_s$.
    Thus, by the definition of $T^\varphi$, we have that $(s-g_n)^\varphi\in T^\varphi$,
    $$p_2((s-g_n)^\varphi)=p_2((t_n)^\varphi) \text{ and } p_1((s-g_n)^\varphi-(t_n)^\varphi)=p_1((s-g_m)^\varphi-(t_m)^\varphi)$$
    for every $n,m\ge n_s$.
    Consequently, $(s-g_n)^\varphi+\ell_n=(s-g_m)^\varphi+\ell_m$ for every $n,m\ge n_s$.
    Set $s^\star=(s-g_n)^\varphi+\ell_n$ for any $n\ge n_s$ and note that $s^\star\in S(\varphi)$.
    Define $C^\star=\{s^\star:s\in C\}$ and observe that $C^\star\subseteq  S(\varphi)\cap (\R\times \{p_2(t^\varphi)\})$.

    \begin{claim}\label{cl:Cstar}
        $C^\star$ is a horizontal column of $(\Omega^\varphi,S(\varphi))$.
    \end{claim}
    \emph{Proof of Claim \ref{cl:Cstar}:}
    By Lemma~\ref{lem:TechnicalClaim}, it is enough to show that $W(\Omega^\varphi\oplus C^\star)=\emptyset$.
    Let $v^\varphi+s^\star$ for some $v\in V(\Omega)$ and $s\in C$.
    As $C$ is a horizontal column of $(\Omega,S)$ we have that $v+s\not \in W(\Omega\oplus C_{v,s})$, where $C_{v,s}\coloneqq \{s_1,\dots,s_k\}$ is the enumeration of $u\in C$ such that $v+s\in V(\Omega)+u$.
    For every 
    $$n\ge \max\{n_{s_i}:1\le i\le k\}$$
     we have that $C_{v,s}-g_n \subseteq T$, $$v+(s-g_n)\not\in W(\Omega\oplus (C_{v,s}-g_n))$$ and $v+(s-g_n)\in V(\Omega)+(s_i-g_n)$ for every $1\le i\le k$.
    Thus, since $\varphi$ is a good approximation, we obtain that
    $$v^\varphi+(s-g_n)^\varphi\not\in W(\Omega^\varphi \oplus (C_{v,s}-g_n)^\varphi)$$
    and consequently
    $$v^\varphi+s^\star\not\in W(\Omega^\varphi\oplus C_{v,s}^\star ),$$ 
    where $C_{v,s}^\star\coloneqq \{(s_1)^\star,\dots,(s_k)^\star\}.$
    
    Now suppose that $y\in W(\Omega^\varphi \oplus C^\star)$.
    Then, as $\Omega^\varphi$ is a polygonal set and $S(\varphi)\in \Tile(\Omega^\varphi; \R^2)$, we must have $y=v^\varphi+s^\star$ for some $v\in V(\Omega)$ and $s\in C$.
    Observe that if $y=w^\varphi+u^\star$ for some $w\in V(\Omega)$ and $u\in C$, then 
    $$ \text{either } C^\star_{v,s}=C^\star_{w,u} \text{ or } C^\star_{v,s}\cap C^\star_{w,u}=\emptyset.$$    Indeed, if $(u')^\star\in C^\star_{v,s}\cap C^\star_{w,u}$, then there are $v',w'\in V(\Omega)$ such that $v+s=v'+u'$ and $w+u=w'+u'$.
    This further implies that
    $$(v')^\varphi+(u')^\star=v^\varphi+s^\star=y=w^\varphi+u^\star=(w')^\varphi+(u')^\star,$$
    which, as $\varphi$ is a good approximation, gives that $w'=v'$.
    In particular, $v+s=w+u$, hence $C^\star_{v,s}=C^\star_{w,u}$.  Now, suppose that $A_1,\dots ,A_m$ are sets of the form $C^\star_{v,s}$,  with  $v\in V(\Omega)$ and $s\in C$ such that  $y=v^\varphi+s^\star$.
    Then, as $\Omega$ is axis-parallel, $\Omega^\varphi\oplus C^\star$ is a packing and $y\not\in W(\Omega^\varphi\oplus A_i)$ for every $1\le i\le m$, we conclude that $m\leq 2$.
    If $m=1$, then we have that $y\not\in W(\Omega^\varphi\oplus C^\star)$ as $y\not\in W(\Omega^\varphi\oplus A_1)$ and $y\not\in V(\Omega^\varphi)+u^\star$ for every $u^\star\in C^\star\setminus A_1$.
    Otherwise, $m=2$ and  $A_1\cap A_2=\emptyset$.  Thus, as $\Omega^\varphi\oplus C^\star$ is a packing and $\Omega^\varphi$ is axis-parallel, if  $y \in W(\Omega^\varphi\oplus C^\star)$ then there must be a small $\eta>0$ such that the open ball $\mathcal{B}_\eta +y$ of radius $\eta$ around $y$ is contained in $\Omega^\varphi\oplus (A_1\sqcup A_2)$. 
    But this is a contradiction with $y\in W(\Omega^\varphi\oplus C^\star)$.
    We thus have Claim \ref{cl:Cstar}.
    $\qquad \blacksquare$

    To finish the proof of Lemma \ref{lem:SimultanousColumns}, it remains to show that  $(\Omega^\varphi,S(\varphi))$ contains a vertical column.
    Combining Lemmas~\ref{lem:ApproxKeepsCosetPieces},~\ref{lem:PeriodicityAproximation} and Lemma~\ref{lem:CharInfComplex}~(2), we have that $\Omega^\varphi\oplus (T^\varphi)_c$ is non-empty and $(b,0)\mathbb{Z}$-periodic for some $b>0$.
    For every $n\in \mathbb{N}$, let $u_n\in (T^\varphi)_c$ be of minimal distance from $(t_n)^\varphi$.
    Since $\Omega^\varphi\oplus(T^\varphi)_c$ is nonempty, vertically invariant and horizontally periodic, it is relatively dense in $\R^2$.
    As $\Omega^\varphi$ is bounded, $(T^\varphi)_c$ is relatively dense as well.
    Thus $\{u_n-(t_n)^\varphi\}_{n\in\mathbb{N}}$ is bounded.
    Since
    $$u_n+\ell_n=(u_n-(t_n)^\varphi)+t^\varphi,$$
    the sequence $\{u_n+\ell_n\}_{n\in\mathbb{N}}$ is bounded.   
    Therefore, as $T^\varphi+\ell_n\to S(\varphi)$, we may without loss of generality assume that $u_n +\ell_n\to u\in S(\varphi)$.
    We claim that
    $$D=S(\varphi)\cap(\{p_1(u)\}\times \R)$$
    is a vertical column of $S(\varphi)$.
    Indeed, since for every $n$ $$p_2(T^\varphi\cap(\{p_1(u_n)\}\times \R))\in \Tile(p_2(\Omega^\varphi\cap (\{z\}\times \mathbb{R}));\R)$$
    for any fixed $z\in p_1(\Omega^\varphi)\setminus p_1(V(\Omega^\varphi))$, and using \cite[Theorem~1]{LW} together with the fact that $\Omega^\varphi$ is an axis-parallel polygonal set, there is $n_0\in \mathbb{N}$ such that, after possibly passing to a further subsequence, $D_n=D_m$ for every $n,m\ge n_0$, where
    $$D_n=\left(T^\varphi\cap(\{p_1(u_n)\}\times \R)\right)-u_n,$$
    which implies that, up to a translation, $D$ is equal to $D_n$ for every $n\ge n_0$, as needed.
\end{proof}

\subsection{Product sets}

In this section we conclude the proof of  Theorem~\ref{thm:FinLocalComplexityTrichotomy}, by analyzing tilings by \emph{polygonal product sets}. 

\begin{definition}[Polygonal product set]
    We say that an axis-parallel polygonal set $\Omega$ is a {\bf \emph{product set}}, if it has a Cartesian product  form $$A\times B=\{(a,b)\colon a\in A,\ b\in B\},$$ where $A=p_1(\Omega)$ and $B=p_2(\Omega)$
     are finite unions of intervals.
\end{definition}

First, we show that the existence of both horizontal and vertical columns in a tiling by a polygonal set $\Omega$ is only possible if $\Omega$ is a product set.

\begin{lemma}[Vertical and horizontal columns imply product structure]\label{lem:ProductOne}
    Let $T$ be a tiling of $\R^2$ by an axis-parallel polygonal set $\Omega$ and $S_0,S_1\subseteq  T$ be vertical and horizontal column of $T$ respectively.
    Then $\Omega$ is a product set.
\end{lemma}
\begin{proof}
    Observe that $X\coloneqq (\Omega\oplus S_0)\cap (\Omega\oplus S_1)$  is a product set that has positive measure, since it is of the form $p_1(\Omega)\times p_2(\Omega) + u$ for some $u\in \R^2$.
    Let $s_0\in S_0$ and $s_1\in S_1$ be such that $(\Omega\oplus s_0)\cap (\Omega\oplus s_1)\cap X$ has positive measure.
    As $T$ is a tiling and $s_0,s_1\in T$, we must have that $s_0=s_1$.
    Moreover, as $S_0$ is a subset of a coset of the vertical axis and $S_1$ is a subset of a coset of the horizontal axis, we  have  that $|S_0\cap S_1|\le 1$. 
    Combined with the fact that $X$ has positive measure, we conclude that  $X=\Omega+s_0=\Omega+s_1$, which gives the claim.
\end{proof}

Next, we prove the periodicity of columns in any tiling by a polygonal product set.

\begin{lemma}[Structure of tiling by product sets]\label{lem:ProductTwo}
    Let $T$ be a tiling of $\R^2$ by an axis-parallel polygonal set $\Omega$ that is a product set such that $T_c\not=\emptyset$.
    Then there is $\beta>0$, that only depends on $\Omega$, such that $\Omega\oplus T_c$ is $(\beta,0)\mathbb{Z}$-periodic.
\end{lemma}
\begin{proof}
    Observe that for almost every $y\in \R$ we have  that the coset $C_y=\R\times \{y\}$ of the horizontal axis satisfies $$m_1(C_y\cap (E(\Omega)+T))=0.$$ 
    Let
    $$S_{C_y}=\{t\in T:(\Omega+t)\cap C_y\not=\emptyset\}.$$
    As $\Omega$ is a product set, we have that $p_1(S_{C_y})\in \Tile(p_1(\Omega);\R)$.
    Observe that as $p_1(\Omega)$ is a finite union of intervals, the set $\Tile(p_1(\Omega),\R)$ is finite up to translations.
    Thus, by \cite[Theorem~1~(iii)]{LW}, we have that there is $\beta>0$ that is an integer multiple of $m_1(p_1(\Omega))$, that only depends on $\Omega$ and not on $y$, such that $p_1(S_{C_y})$ is $\beta\mathbb{Z}$-periodic.
    
    Let $S\subseteq  T$ be a vertical column of $(\Omega,T)$, $P=\{x\}\times \R$ be such that $S\subseteq P$, and  
    $S_k=T\cap (P+(\beta k,0))$ for every $k\in \mathbb{Z}$.
    Then, for $k\in \Z$  we have $x+\beta k\in p_1(S_{C_y})$ for almost every $y\in \R$, and thus $\Omega\oplus S_k$ covers $(\Omega\oplus S_k)+\mathbb{R}(0,1)$ up to a null set.
    Consequently, $S_k$ is a vertical column of $(\Omega,T)$ and we are done.
\end{proof}

Our next lemma shows that we can read a product structure of a tile from product structure of its good approximations.

\begin{lemma}[Product approximations imply product structure]\label{lem:ProductThree}
    Let $T$ be a tiling of $\R^2$ by an axis-parallel polygonal set $\Omega$ and suppose that there is $\epsilon>0$ such that for every good $\epsilon$-approximation $\varphi$ we have that $\Omega^\varphi$ is a product set.
    Then $\Omega$ is a product set.    
\end{lemma}
\begin{proof}
    Let $\epsilon_n\to 0$, where $0<\epsilon_n<\epsilon$, and $\varphi_n$ be a good $\epsilon_n$-approximation for every $n\in \mathbb{N}$ (recall that these exist by Corollary~\ref{cor:GoodApproxExists}).
    By the assumption, for every $n\in \mathbb{N}$ we have that $\Omega^{\varphi_n}=A_n\times B_n$ for finite union of rational intervals $A_n,B_n\subseteq  \R$.
    Since $p_1(\Omega^{\varphi_n})=A_n$ and $p_2(\Omega^{\varphi_n})=B_n$ for every $n\in \mathbb{N}$, the endpoints of these intervals are determined by $V(\Omega^{\varphi_n})$.
    Consequently, as $\varphi_n$ is a good $\epsilon_n$-approximation and $\epsilon_n\to 0$, we get that $A_n\to p_1(\Omega)$, $B_n\to p_2(\Omega)$ and $m_2(\Omega^{\varphi_n})\to m_2(\Omega)$.
    Finally, this shows the claim by the combination of
    $$m_2(\Omega^{\varphi_n})=m_1(A_n)\times m_1(B_n)\to m_2(p_1(\Omega)\times p_2(\Omega))$$
    and $\Omega\subseteq  p_1(\Omega)\times p_2(\Omega)$.
\end{proof}

\subsection{Proof of Theorem~\ref{thm:FinLocalComplexityTrichotomy}}
We can now conclude the proof of Theorem~\ref{thm:FinLocalComplexityTrichotomy}:

\begin{proof}[Proof of Theorem~\ref{thm:FinLocalComplexityTrichotomy}]
    By Theorem~\ref{thm:MainEarthSplit}, there is $S_0\in \orbit{T}{\R^2}{\Omega}$ such that, after possibly swapping the vertical and horizontal axes, the earthquake decomposition of $S_0$ is non-trivial.
    If $S_0$ has finite horizontal complexity, then we are done.
    Otherwise, we apply Lemma~\ref{lem:Options} to get $S_1\in \orbit{S_0}{\R^2}{\Omega}$ that admits non-trivial horizontal earthquake decomposition.
    If it additionally satisfies (a) in Lemma~\ref{lem:Options}, then we are done after swapping the vertical and horizontal axes.
    Otherwise, we appeal to Lemmas~\ref{lem:SimultanousColumns},~\ref{lem:ProductOne} and \ref{lem:ProductThree} to conclude that $\Omega$ is a product set.
    By Lemma~\ref{lem:ProductTwo}, we thus have that $\Omega\oplus (S_0)_c$ is periodic,  which contradicts Lemma~\ref{lem:CharInfComplex}, since we assumed   that $S_0$ does not have finite horizontal complexity.
\end{proof}

\section{Structure theorem for minimal tilings}\label{sec:MinimalTiling}

In this section we prove a structural result for \emph{topologically minimal} tilings by axis-parallel polygonal sets.

\begin{definition}[Topologically minimal tilings]\label{def:topminimal}
    Let $X\not=\emptyset$ be a regular open periodic subset of $\R$, $E=X\times \R$, $\Omega$ be a polygonal set and $T\in \Tile(\Omega;E)$.
    We say that  $T$ is {\bf \emph{topologically minimal}} if for every $S\in \orbit{T}{E}{\Omega}$ we have  $\orbit{T}{E}{\Omega}=\orbit{S}{E}{\Omega}$.
\end{definition}

A classical theorem of Birkhoff asserts that for any tiling $T$ of $E$ by $\Omega$ there exists $S\in \orbit{T}{E}{\Omega}$ that is topologically minimal.

We show that topologically minimal tilings by an axis-parallel polygonal set split, after possibly swapping the axes, into finitely many pieces each covering a periodic region and with finite local complexity, together with a union of vertical columns. This  gives Theorem \ref{thm:IntroMainTectonicPlatesFinLocComp} from the Introduction that we restate here for readers convenience.

\mainEarthquaketheorem*

We prove the theorem by analyzing the tectonic plates of the tiling $T$, and showing that under the assumptions of the theorem, there are only finitely many plates of $(T,\Omega)$ which are not vertical columns and each such tectonic plate has finite local complexity.

\begin{lemma}[Weak limits of translates preserve finite horizontal complexity]\label{lem:FinHorComplexLimits}
    Let $T$ be a tiling of $\R^2$ by an axis-parallel polygonal set $\Omega$ that has finite horizontal complexity and $S\in \orbit{T}{\R^2}{\Omega}$.
    Then $S$ has finite horizontal complexity.
\end{lemma}
\begin{proof}
    First, we show that
   \begin{equation}\label{eq:s'st't}
       p_1(s')-p_1(s)\in \{p_1(t')-p_1(t):t,t'\in T\}
   \end{equation}
    for every $s,s'\in S$.
    By the definition \eqref{def:orbit} of $\orbit{T}{\R^2}{\Omega}$ there is a sequence $\{g_n\}_{n\in \mathbb{N}}\subseteq \R^2$ such that $T+g_n\to S$, and by the definition of weak convergence there are sequences $\{t_n\}_{n\in \mathbb{N}},\{t'_n\}_{n\in \mathbb{N}}\subseteq  T$ such that $t_n+g_n\to s$ and $t'_n+g_n\to s'$.
    Therefore, by the continuity of the projection $p_1$, we have that $$p_1(t'_n)-p_1(t_n)\to p_1(s')-p_1(s).$$
    As $p_1(T)$ has finite local complexity, there must be $n_0\in \mathbb{N}$ such that $p_1(t'_n)-p_1(t_n)$ is constant for every $n\ge n_0$, thus equal to $p_1(s')-p_1(s)$, giving \eqref{eq:s'st't}.
    In particular, $p_1(S)$ is discrete since $T$ has finite horizontal complexity.
    Fix $R>0$ and observe that for every $s\in S$ we have that the set $(\mathcal{B}_R\cap p_1(S-s))$ is contained in  $$\{p_1(t)-p_1(t'):t,t'\in T \ \text{s.t.} \ |p_1(t)-p_1(t')|< R\}.$$
    Since the latter set is finite as $T$ has finite horizontal complexity, we conclude that there are only finitely many sets of the form $(\mathcal{B}_R\cap p_1(S-s))$ for $s\in S$.
    This shows that $S$ has finite horizontal complexity as required.
\end{proof}

\begin{lemma}[Weak limits of translates preserve earthquake decomposition]\label{lem:EarthquakePreserved}
    Let $T$ be a tiling of $\R^2$ by an axis-parallel polygonal set $\Omega$ that is topologically minimal and has finite horizontal complexity, $\beta>0$ be as in Theorem~\ref{thm:PeriodicityMain} and $T+g_n\to S\in \orbit{T}{\R^2}{\Omega}$, where $\{g_n\}_{n\in \mathbb{N}}\subseteq  \beta\Z\times \R$.
    Then
    \begin{equation}\label{eq:e0}
        \Omega\oplus T_c=\Omega\oplus S_c
    \end{equation}
    and there is a bijection
    $$\kappa:\quake{T}{\Omega}\to\quake{S}{\Omega}$$
    such that $\Omega\oplus P=\Omega\oplus \kappa(P)$ for every $P\in \quake{T}{\Omega}$.
\end{lemma}
\begin{proof}
    Let $E_c\coloneqq \Omega\oplus T_c$ and $E_P\coloneqq \Omega\oplus P$ for every $P\in \quake{T}{\Omega}$.
    By Theorem~\ref{thm:PeriodicityMain}, we have that \begin{equation}\label{eq:EGamma}
        \{g_n\}_{n\in \mathbb{N}}\subseteq  \Gamma(E_c)\cap \left(\bigcap_{P\in \quake{T}{\Omega}} \Gamma(E_P)\right).
    \end{equation}
    Thus, by Lemma~\ref{lem:BasicWeak} and since $T+g_n\to S$, we can partition $S$ as 
    \begin{equation}\label{eq:Decomposition}
        S=U_c\sqcup \left(\bigsqcup_{P\in \quake{T}{\Omega}} U_P\right),
    \end{equation}
    where $U_c$ is the weak limit of the sequence $T_c+g_n$, as $n\to \infty$ and for every $P\in \quake{T}{\Omega}$, $U_P$ is the weak limit of the sequence $P+g_n$, as $n\to \infty$. Moreover, as by \eqref{eq:EGamma}, for every $n\in \N$, $E_c=E_c+g_n$ and $E_P=E_P+g_n$, for every $P\in \quake{T}{\Omega}$, we must have that  $E_c=\Omega\oplus U_c$ and $E_P=\Omega\oplus U_P$ for every $P\in \quake{T}{\Omega}$. Thus, by Lemma~\ref{lem:BasicQuake} for every $P\in \quake{T}{\Omega}$ there is $P_S\in \quakeall{S}{\Omega}$ such that  $P_S\subseteq U_P$.
    
   Note that to complete the proof, it suffices to show that  $S_c=U_c$ and that \eqref{eq:Decomposition} is in fact the earthquake decomposition of $S$. Indeed, if this is the case, then setting $\kappa(P)=U_P$ gives the claim. Thus, the rest of the proof will be devoted to showing that weak limits of translates of a tiling preserve its earthquake decomposition.

    Let $s\in U_c$ and fix a sequence $\{t_n\}_{n\in \N}\subseteq T_c$ such that $t_n+g_n\to s$. Then, we have that the set  $A_n\coloneqq T\cap (\{p_1(t_n)\}\times \R)$  is a vertical column of $(\Omega,T)$ for every $n\in \mathbb{N}$.  Let $x\in \Omega$ be such that $p_1(x)\not\in p_1(V(\Omega))$, and $\Omega_x\coloneqq \Omega \cap (\{p_1(x)\}\times \R)$. Then $\Omega_x$ is a finite union of intervals (since $\Omega$ is an open polygonal set) and $A_n\in \Tile(\Omega_x; \{p_1(t_n+x)\}\times \R)$.
    By  \cite[Theorem~1]{LW} and since the sequence $\{A_n-t_n\}_{n\in \N}$ converges weakly, we conclude that there is $n_0\in \mathbb{N}$  such that $A_n=\tilde{A}+t_n$ for every $n\ge n_0$, where $\tilde{A}=A_{n_0}-t_{n_0}$.
    Hence $$\tilde{A}+s\subseteq  S\cap (\{p_1(s)\}\times \R),$$ and as $\Omega\oplus S$ is a tiling we conclude that $s\in S_c$.
    Since this is true for every $s\in U_c$, we conclude that $U_c\subseteq S_c$ and therefore 
    \begin{equation}\label{eq:e1}
        E_c =\Omega\oplus T_c =\Omega \oplus U_c \subseteq \Omega\oplus S_c. 
    \end{equation}

     By Lemma~\ref{lem:FinHorComplexLimits}, we have that $S$ has finite horizontal complexity. So, by applying Theorem~\ref{thm:PeriodicityMain} for $S\in \Tile(\Omega; \R^2)$, there is $\beta'>0$ such that $\Omega\oplus S_c$ and $\Omega\oplus P$ are $(\beta',0)\Z$-periodic for every $P\in \quake{S}{\Omega}$. As $T$ is topologically minimal, there is a sequence $\{h_n\}_{n\in \mathbb{N}}$ such that $S+h_n\to T$.
       Let $\alpha_n$ be the number in $[0,\beta')$ such that $\alpha_n =p_1(h_n) \mod \beta'\Z $.
        By compactness, on restricting to a subsequence if needed, we may assume that  $\alpha_n\to \alpha\in [0,\beta']$, and thus
        $$S+(h_n-(\alpha_n,0))\to T-(\alpha,0).$$
        Let $g_n'\coloneqq h_n-(\alpha_n,0)$. Then, by construction $g_n'\in \beta'\Z\times \R$ for every $n\in \mathbb{N}$. Therefore, repeating the argument above, we obtain that $U_c'\subseteq T_c-(\alpha,0)$, where $U_c'$ is the weak limit of $S_c+g_n'$. Therefore 
        \begin{equation}\label{eq:e2}
            \Omega \oplus S_c=\Omega\oplus U_c' \subseteq \Omega\oplus T_c -(\alpha,0) =E_c-(\alpha,0).
        \end{equation}
        Combining \eqref{eq:e1} and \eqref{eq:e2} and using the $\beta\mathbb{Z}$-periodicity of $p_1(E_c)$, we obtain that 
        $p_1(E_c)= p_1(E_c)-\alpha$ and so,  \eqref{eq:e0} holds. Moreover, as both $S_c$ and $U_c$ are subsets of the tiling $S$, we must have $U_c=S_c$.  
Thus, for every $P\in \quake{T}{\Omega}$ there is $P_S\in \quake{S}{\Omega}$ such that $P_S\subseteq U_P$;
    in particular,  $$|\quake{T}{\Omega}|\le |\quake{S}{\Omega}|.$$
        By the same argument applied to the case where the roles of $S$ and $T$ are swapped,  we infer that for every $P\in \quake{S}{\Omega}$ there is $P_T\in \quake{T}{\Omega}$ such that $P_T-(\alpha,0)\subseteq U_P'$, and thus
      $$|\quake{T}{\Omega}|\ge |\quake{S}{\Omega}|.$$
        We therefore have $$|\quake{T}{\Omega}|= |\quake{S}{\Omega}|,$$ which completes the proof.
\end{proof}

\subsection{Proof of Theorem~\ref{thm:IntroMainTectonicPlatesFinLocComp}} We can now conclude the proof of Theorem~\ref{thm:IntroMainTectonicPlatesFinLocComp}:

\begin{proof}
    The claim easily follows if, after possibly swapping the vertical and horizontal axes, we have $T=T_c$.
    Let us assume that this is not the case, and suppose first that $T$ has finite local complexity.
    By Corollary~\ref{cor:FinNumberPlates}, there is $m\ge 0$ and an enumeration of $\quake{T}{\Omega}$ as $T_1,\dots,T_m$.
    The claim follows easily in case that the earthquake decomposition of $T$ with respect to $\Omega$ is trivial by taking $T=T_1$ and $m=1$.
    Otherwise, we have that $P$ has finite local complexity for every $P\in \quake{T}{\Omega}$ and, by Lemma~\ref{lem:CharInfComplex}~(4), that $T$ has finite horizontal complexity.
    Hence the claim follows with $T=T_c\sqcup T_1\sqcup\dots\sqcup T_m$ by Theorem~\ref{thm:PeriodicityMain}.
     
    Suppose now that $T$ does not have finite local complexity.
    By Theorem~\ref{thm:FinLocalComplexityTrichotomy}, after possibly swapping the vertical and horizontal axes, there is $S\in \orbit{T}{\R^2}{\Omega}$ with finite horizontal complexity and with non-trivial vertical earthquake decomposition.
    Let $\beta_S>0$ be as in Theorem~\ref{thm:PeriodicityMain} applied to $S$.
    Since $T$ is topologically minimal, there is a sequence $(h_n)_{n\in \N}$ such that $S+h_n\to T$.
    Write $p_1(h_n)=k_n\beta_S+a_n$, where $a_n\in[0,\beta_S)$.
    After passing to a subsequence, assume $a_n\to a$.
    Then
    $$S+(h_n-(a_n,0))\longrightarrow T-(a,0),\qquad h_n-(a_n,0)\in\beta_S\Z\times\R.$$
    Lemma~\ref{lem:EarthquakePreserved} implies that $T-(a,0)$, and hence $T$, has nontrivial vertical earthquake decomposition, and by Lemma~\ref{lem:FinHorComplexLimits} we have that $T$ has finite horizontal complexity.
     
    By Theorem~\ref{thm:PeriodicityMain}, there is $\beta>0$ such that 
    $$\Omega\oplus T_{c} \text{ and } \Omega\oplus P$$
    are $\R(0,1)$-invariant and $(\beta,0)\mathbb{Z}$-periodic for every $P\in \quake{T}{\Omega}$. Since $T\neq T_c$ and by Corollary~\ref{cor:FinNumberPlates},  the set $\quake{T}{\Omega}$ has $m$ elements, where $1\leq m<\infty$; thus, $T=T_c\sqcup T_1\sqcup\dots\sqcup T_m$, where $\{T_1,\dots, T_m\}=\quake{T}{\Omega}$. Now, by Theorem~\ref{thm:PeriodicityMain}, for each $1\leq j\leq m$ $\Omega\oplus T_j$ is $(\beta,0)$-periodic. Hence, to complete the proof it suffices to show for every $P\in \quake{T}{\Omega}$ that $P$ has finite local complexity.

    Suppose, for a contradiction, that there is $P_0\in \quake{T}{\Omega}$ that does not have finite local complexity.
    Let $\Lambda$ be the lattice generated by $(0,1)$ and $(\beta,0)$ and observe that, by Lemma~\ref{lem:BasicQuake}~(i) and Theorem~\ref{thm:PeriodicityMain}, we have 
    $\Lambda\subseteq \Gamma(\Omega\oplus P_0)$.
    Thus, by Theorem~\ref{thm:MainEarthSplit}, there is a sequence $\{\lambda_n\}_{n\in \mathbb{N}}\subseteq \Lambda$ such that  $$P_0+\lambda_n\to U_{P_0},$$ where $U_{P_0}$ can be partitioned as $U_{P_0}=S_0\sqcup S_1$ with $S_i\not=\emptyset$ and $\Omega\oplus S_i$ is $\R(0,1)$-invariant for every $i\in \{0,1\}$. Moreover, since by Lemma~\ref{lem:BasicWeak}, after possibly passing to a subsequence, we may assume that $$T+\lambda_n\to S\in \orbit{T}{\R^2}{\Omega},$$ we can apply Lemma~\ref{lem:EarthquakePreserved} to conclude that $U_{P_0}\in \quake{S}{\Omega}$. However, this contradicts Lemma~\ref{lem:BasicQuake}~(ii), therefore such $P_0$ cannot exist and the statement is proven.
\end{proof}

\section{Finer structure theorem for discrete and rational tiles}\label{sec:FinerStructure}

In this section, we prove a refinement of the decomposition provided by Theorem~\ref{thm:IntroMainTectonicPlatesFinLocComp} for tilings by axis-parallel polygonal sets with rational vertices.
Recall that a vector $(a,b)\in \Z^2$ is {\bf \emph{primitive}} if $\text{gcd}(a,b)=1$, and vectors $v,w\in \Z^2\setminus \{(0,0)\}$ are {\bf \emph{incommensurable}} if there are no nonzero integers $m$ and $n$ such that $m\cdot v=n\cdot w$.

\begin{theorem}\label{thm:GoodParamDiscrete}
    Let $F$ be a finite non-empty subset of $\mathbb{Z}^2$, $E$ be an $\ell\mathbb{Z}^2$-periodic subset of $\mathbb{Z}^2$ for some positive integer $\ell$ and $A$ be a tiling of $E$ by $F$.
    Then there is a lattice $\langle h,g\rangle\subseteq  \mathbb{Z}^2$, a sequence $(k_j,\ell_j)_{j=1}^m\subseteq \Z^2$ such that $k_j\not=0$ for every $1\le j \le m$, decompositions $A=A_1\sqcup \dots\sqcup A_m$ and
    $$A_j=\bigsqcup_{Q \in L_j} A_{j,Q},$$
    where $L_j$ is the set of all cosets of $\langle  h_j\rangle$ that intersect $A_j$, $ h_j$ is a primitive element of $\mathbb{Z}^2$ such that $k_jh+\ell_jg\in \langle  h_j\rangle$ and $A_{j,Q}=A_j\cap Q$ for every $Q\in L_j$, such that
    \begin{enumerate}
        \item $A_{j,Q}$ is $\langle k_jh+\ell_jg\rangle$-periodic for every $Q\in L_j$ and $1\le j \le m$,
        \item $F\oplus A_{j,Q}+ng=F\oplus A_{j,Q+ng}$ for every $n\in \mathbb{Z}$, $Q\in L_j$ and $1\le j \le m$.
    \end{enumerate}
\end{theorem}

To prove Theorem \ref{thm:GoodParamDiscrete}, we will use the following lemma, which follows from \cite[Corollary~5.3]{GreenfeldTaoStructureZd}.

\begin{lemma}\label{lm:ShiftingSlices}
    Let $k,\ell\ge 1$ be natural numbers, and let $h$ be a primitive element of $\mathbb{Z}^2$.
    Suppose that $F$ is a finite subset of $\mathbb{Z}^2$, $A$ is $\langle \ell k h\rangle$-periodic subset of $\mathbb{Z}^2$ such that $F\oplus A$ is $\ell\mathbb{Z}^2$-periodic.
    Set $h^\perp$ to be a primitive element of $\mathbb{Z}^2$ that satisfies $h^\perp \perp h$.
    Then there is a positive integer $K$ such that 
    $$\left(F\oplus (A\cap (\langle h\rangle+x))\right)+Kh^\perp=F\oplus (A\cap (\langle h\rangle+x+Kh^\perp))$$
    holds for every $x\in \mathbb{Z}^2$. 
\end{lemma}
\begin{proof}
    Indeed, \cite[Corollary~5.3]{GreenfeldTaoStructureZd} guarantees the existence of positive integers $s,q$ and a $qks\mathbb{Z}^2$-periodic set $A'$ such that
    $$F\oplus (A\cap (\langle h\rangle+x))=F\oplus (A'\cap (\langle h\rangle+x))$$
    holds for every $x\in \mathbb{Z}^2$. Thus, setting $K=qks$, we obtain that
    \begin{equation*}
        \begin{split}
            \left(F\oplus (A\cap (\langle h\rangle+x))\right)+Kh^\perp= & \ \left(F\oplus (A'\cap (\langle h\rangle+x))\right)+Kh^\perp\\
            = & \ F\oplus ((A'\cap (\langle h\rangle+x))+Kh^\perp)\\
            = & \ F\oplus ((A'+Kh^\perp)\cap (\langle h\rangle+x+Kh^\perp))\\
            = & \ F\oplus (A'\cap (\langle h\rangle+x+Kh^\perp))\\
            = & \ F\oplus (A\cap (\langle h\rangle+x+Kh^\perp))
        \end{split}
    \end{equation*}
    for every $x\in \mathbb{Z}^2$, where in the fourth equality we have used that $A'+Kh^\perp=A'$ as $Kh^\perp \in K\mathbb{Z}^2$ and $A'$ is  $K\mathbb{Z}^2$-periodic.
\end{proof}

\begin{proof}[Proof of Theorem~\ref{thm:GoodParamDiscrete}]
    If $|F|=1$, then the claim follows by taking $h=(\ell,0)$, $g=(0,\ell)$, $m=1$, $k_1=1$, $\ell_1=0$, $h_1=(1,0)$ and $A=A_1$.
    Hence, we assume that $|F|>1$, and, after possibly shifting $F$ and $A$, that $0\in F$.
    By \cite[Theorem~1.4]{GreenfeldTaoStructureZd}, there is a lattice $\Lambda\subseteq  \ell\mathbb{Z}^2$, pairwise incommensurable vectors $h_1,\dots,h_m\in \mathbb{Z}^2\setminus \{(0,0)\}$ and a positive integer $L$ such that for every coset $\Lambda+x$ of $\Lambda$ there is $1\le j\le m$ such that the intersection of $A$ with $\Lambda+x$ is $\langle \ell L h_j\rangle$-periodic.
    Replacing each $h_j$ by its primitive parallel vector and enlarging $L$ if necessary, we may assume that every $h_j$ is primitive.

    Assign to each coset $\Lambda+x$ the smallest index $1\le j\le m$ such that the intersection of $A$ with $\Lambda+x$ is $\langle \ell L h_j\rangle$-periodic.
    For $1\le j\le m$, define $A_j$ to be the union of those intersections of $A$ with cosets of $\Lambda$ that are assigned index $j$.
    Observe that $A_j$ is $\langle \ell L h_j\rangle$-periodic for every $1\le j\le m$ and we have $A=A_1\sqcup\dots\sqcup A_m$.

    If $m=1$, then we set $N_1=1$ and observe that $F\oplus A_1=F\oplus A$ is $\ell L N_1\Z^2$-periodic by the assumption.
    If $m>1$, then by \cite[Theorem~5.4]{GreenfeldTaoStructureZd}, we have that $F\oplus A_j$ is $\ell LN_j\mathbb{Z}^2$-periodic for some integer $N_j$ for every $1\le j\le m$. Then, applying Lemma~\ref{lm:ShiftingSlices} for $F\oplus A_j$ we obtain a positive integer $K_j$ and a primitive vector $h^\perp_j$ such that $h_j\perp h^\perp_j$ and
    \begin{equation}\label{eqseven}
        \left(F\oplus (A_j\cap (\langle h_j\rangle+x))\right)+K_jh^\perp_j=F\oplus (A_j\cap (\langle h_j\rangle+x+K_jh^\perp_j))
    \end{equation}
    for every $x\in \mathbb{Z}^2$.

    Let $L_j$ be the set of cosets of $\langle h_j\rangle$ that have non-empty intersection with $A_j$ and for every $Q\in L_j$ let $A_{j,Q}=A_j\cap Q$. Set $M_j=\ell L N_j$.
    
    {\bf Claim I}
        For every $1\le j\le m$, we have that $A_{j,Q}$ is $\langle M_jh_j\rangle$-periodic and 
        $$F\oplus A_{j,Q}+ (a(M_jh_j)+b(K_j h^\perp_j))=F\oplus A_{j,Q +(a(M_jh_j)+b(K_j h^\perp_j))}$$
        for every $(a,b)\in \mathbb{Z}^2$ and $Q\in L_j$.

    \begin{proof}[Proof of Claim I]
        Let $t\in A_{j,Q}$ and $r\in \mathbb{Z}$, then we have that $t+rM_jh_j\in A_j$ as $A_j$ is $\langle M_jh_j\rangle$-periodic.
        Since $t\in A_{j,Q}=A_j\cap Q$ and $Q$ is a coset of $\langle h_j\rangle$, we obtain that $t+rM_jh_j\in Q$.
        Consequently, $t+rM_jh_j\in A_{j,Q}$.
        As $t\in A_{j,Q}$ and $r\in \mathbb{Z}$ were arbitrary, we conclude that $A_{j,Q}$ is $\langle M_jh_j\rangle$-periodic. Thus, we have
        \begin{equation*}
            \begin{split}
                F\oplus A_{j,Q}+ (a(M_jh_j)+b(K_j h^\perp_j))= & \ F\oplus (A_{j,Q}+ a(M_jh_j))+b(K_j h^\perp_j)\\
                = & \ F\oplus A_{j,Q}+ b(K_j h^\perp_j)
            \end{split}
        \end{equation*}
        and, since $Q +a(M_jh_j)=Q$, we also have 
        $$F\oplus A_{j,Q +(a(M_jh_j)+b(K_j h^\perp_j))}=F\oplus A_{j,Q +b(K_j h^\perp_j)}.$$ 
        Therefore, we may assume that $a=0$.

        Let $x\in \mathbb{Z}^2$ be such that $Q=\langle h_j\rangle+x$.
        Then we have 
        \begin{equation*}
            \begin{split}
                F\oplus A_{j,Q}+K_jh^\perp_j= & \ \left(F\oplus (A_j\cap Q)\right)+K_j h^\perp_j\\
                = & \ \left(F\oplus (A_j\cap (\langle h_j\rangle+x))\right)+K_j h^\perp_j\\
                = & \ F\oplus (A_j\cap (\langle  h_j\rangle+x+K_j h^\perp_j))\\
                = & \ F\oplus (A_{j}\cap (Q+K_j h^\perp_j))\\
                = & \ F\oplus A_{j,Q+K_j h^\perp_j}
            \end{split}
        \end{equation*}
        by \eqref{eqseven}, where the last equality follows from the fact that $A_{j}\cap (Q+K_j h^\perp_j)$ is nonempty.
        By induction, this gives the claim for every $b\ge 0$.

        Now, let $b=-1$.
        Take $Q-K_j h^\perp_j=\langle  h_j\rangle+y$ for some $y\in \mathbb{Z}^2$.
        Then, by \eqref{eqseven}  we have
        \begin{equation*}
            \begin{split}
                F\oplus (A_{j}\cap (Q-K_j h^\perp_j))+K_j h^\perp_j= & \ F\oplus (A_{j}\cap (\langle  h_j\rangle+y))+K_j h^\perp_j\\
                = & \ F\oplus (A_j\cap (\langle  h_j\rangle+y+K_j h^\perp_j))\\
                = & \ F\oplus (A_j\cap Q)=F\oplus A_{j,Q}.
            \end{split}
        \end{equation*}
        In particular, we have that $A_{j}\cap (Q-K_j h^\perp_j)$ is nonempty.
        Consequently, we obtain
        $$F\oplus (A_{j}\cap (Q-K_j h^\perp_j))=F\oplus A_{j,Q}-K_j h^\perp_j.$$
        By induction, this gives the claim for every $b< 0$.
    \end{proof}

    To complete the proof, it remains to find the lattice $\langle h,g\rangle$ and the sequence $(k_i,\ell_i)_{i=1}^m$ such that $k_i\not =0$ for every $1\le i\le m$ satisfying the statement.
Let $h,g$ be two nonzero vectors in the lattice $\bigcap_{j=1}^m \langle M_j h_j,K_j h^\perp_j\rangle$ such that $h\perp g$ and $g$ is incommensurable with $ h_j$ for every $1\le j\le m$.

        {\bf Claim II.} For every $1\le j\le m$, there is $(k_j,\ell_j)\in \mathbb{Z}^2$ such that $k_j\not= 0$ and $k_j h+\ell_j g$ is a non-zero integer multiple of $M_j h_j$.

    \begin{proof}[Proof of Claim II] 
        Let  $1\le j\le m$. Note that if such $(k_j,\ell_j)$ exists, then by the incommensurability of $g$ and $h_j$, we must have $k_j\not= 0$. Then, there exist $a,b,c,d\in \mathbb{Z}$ such that $aM_j h_j+bK_j h^\perp_j=g$ and $cM_j h_j+dK_j h^\perp_j=h$.
        Recall that $M_j,K_j>0$ and note that $b\not= 0$. 
        
        If $d=0$, then $c\not= 0$ and we can complete the proof by setting $k_j=1$ and $\ell_j=0$.
        If $d\not =0$, then we have
        $$dg-bh=daM_j h_j+dbK_j h^\perp_j-bcM_j h_j-bdK_j h^\perp_j=(da-bc)M_j h_j.$$
        The assumption $h\perp g$ implies that $dg-bh\not =0$.
        Consequently $(da-bc)M_j\not = 0$ and we can set $k_j=-b$ and $\ell_j=d$ to complete the proof of the claim.       
    \end{proof}

    Note that part (1) of Theorem \ref{thm:GoodParamDiscrete} immediately follows from the Claim II.
    To see part (2), let $1\le j\le m$, $Q\in L_j$ and $n\in \mathbb{Z}$.
    Write $g=aM_j h_j+bK_j h^\perp_j$.
    Then, by Claim I, we have
    \begin{equation*}
        \begin{split}
            F\oplus A_{j,Q}+ng= & \ F\oplus A_{j,Q}+(naM_j h_j+nbK_j h^\perp_j)\\
            = & \ F\oplus A_{j,Q+n(aM_j h_j+bK_j h^\perp_j)}=F\oplus A_{j,Q+ng},
        \end{split}
    \end{equation*}
    as needed.   
\end{proof}

As an immediate corollary of Theorem~\ref{thm:GoodParamDiscrete}, we get a decomposition for tilings by axis-parallel polygonal sets with rational vertices. 

\begin{corollary}\label{cor:MainRational}
    Let $X\not=\emptyset$ be a regular open periodic subset of $\R$, $T\subseteq  \mathbb{R}^2$ be a tiling of $E=X\times \mathbb{R}$ by an axis-parallel polygonal set $\Omega$ such that $T, V(\Omega)\subseteq  \frac{1}{d}\mathbb{Z}^2$ for some $0\not=d\in \mathbb{N}$.
    Then there is a lattice $\langle\lambda,\gamma\rangle\subseteq  \mathbb{Q}^2$ and vectors $(k_i,\ell_i)_{i=1}^n\subseteq \mathbb{Z}^2$ such that $k_i\not=0$ for every $1\le i\le n$, decompositions $T=T_1\sqcup\dots \sqcup T_n$ and 
    $$T_i=\bigsqcup_{Q\in L_i} T_{i,Q},$$
    where for every $1\le i\le n$ the index set $L_i$ is equal to the set of cosets of the one-dimensional linear subspace of $\mathbb{R}^2$ that is generated by $k_i\lambda+\ell_i\gamma$ that intersect $T_i$ and $T_{i,Q}=Q\cap T_i$ for every $Q\in L_i$, such that
    \begin{enumerate}
        \item $T_{i,Q}$ is $(k_i\lambda+\ell_i\gamma)$-periodic for every $Q\in L_i$ and $1\le i\le n$,
        \item for every $m\in \mathbb{Z}$ and $Q\in L_i$ we have that
        $$(\Omega\oplus T_{i,Q})+m\gamma=\Omega\oplus T_{i,Q+m\gamma}$$
        for every $1\le i\le n$.
    \end{enumerate} 
\end{corollary}

\section{Periodicity of tilings of finite local complexity}\label{sec:PeriodicityFLC}

In this section we prove the following theorem, on the structure of tilings by axis-parallel tiles that have finite local complexity.

\begin{theorem}[Periodicity of tilings of finite local complexity]\label{thm:MainLocallyFin}
Let $X\not=\emptyset$ be a regular open periodic subset of $\R$, $\Omega$ be an axis-parallel polygonal set and $E=X\times \R$. Then:
\begin{itemize}
    \item[(i)] Any $T\in \Tile(\Omega;E)$ that has finite local complexity is weakly periodic.
    \item[(ii)] If there is $T\in \Tile(\Omega;E)$ that has finite local complexity, then there is $S\in \Tile(\Omega;E)$ that is periodic.
\end{itemize}
\end{theorem}

Throughout this section, we assume that $X\subseteq \R$, $\Omega$ and $E=X\times \R$ and $T$ are as in Theorem~\ref{thm:MainLocallyFin}.

\subsection{Good parametrization}\label{subsec:DefinitionDecomposition}
\newcommand{\wlambda}{\widetilde{\lambda}}
\newcommand{\wgamma}{\widetilde{\gamma}}
In some of the arguments in the proof we need to distinguish cases based on properties of the polygonal set $\Omega$.
Recall that $\Omega$ {\bf \emph{can tile a column}} if there is $S\subseteq  \mathbb{R}(0,1)$ such that $\Omega\oplus S=p_1(\Omega)\times \mathbb{R}$.

As the first step in the proof we apply Theorem~\ref{thm:GoodFiniteLocalComplexity} with large enough $M>0$.
For most of the arguments in this section, we require that 
\begin{equation}\label{eq:LargeParameter}
    M>\max\left\{189\diam(\Omega),\frac{1}{\diam(\Omega)},\sqrt{6},4|E(\Omega)|\left(\frac{\diam(\Omega)}{\min\{|\omega|:\omega\in E(\Omega)\}}+1\right)\right\},
\end{equation}
which clearly depends only on $\Omega$, and not on $T$.
Additional lower bounds on $M>0$, that also only depends on $\Omega$, will come from Proposition~\ref{pr:NotLinkedOne} and Proposition~\ref{pr:NotLinkedThree}, see the discussion after the corresponding proofs and in the beginning of Section~\ref{subsubsec:ProofMainLemma}.

Take $\epsilon>0$ as in Theorem~\ref{thm:GoodFiniteLocalComplexity}.
By a combination of Theorem~\ref{thm:GoodFiniteLocalComplexity}~(1) and (2) and Corollary~\ref{cor:MainRational}, we have that for any $\epsilon$-approximation $\varphi$ there is a lattice $\langle\wlambda,\wgamma\rangle\subseteq  \mathbb{Q}^2$, $(k_i,\ell_i)_{i=1}^n\subseteq \mathbb{Z}^2$ with $k_i\not=0$ for every $1\le i\le n$, such that there are decompositions $T^\varphi=T^\varphi_1\sqcup\dots \sqcup T^\varphi_n$ and 
    $$T^\varphi_i=\bigsqcup_{Q\in L_i} T^\varphi_{i,Q},$$
where for every $1\le i\le n$ the index set $L_i$ is the set of cosets of the one-dimensional linear subspace $D_i\subseteq \mathbb{R}^2$ generated by $k_i\wlambda+\ell_i\wgamma$ that intersect $T^\varphi_i$, that is, $D_i=\R(k_i\wlambda+\ell_i\wgamma)$, and for every $Q\in L_i$ the set $T^\varphi_{i,Q}=Q\cap T^\varphi_i$ is the corresponding intersection, and 
\begin{enumerate}
    \item $T^\varphi_{i,Q}$ is $(k_i\wlambda+\ell_i\wgamma)$-periodic for every $Q\in L_i$ and $1\le i\le n$,
    \item for every $m\in \mathbb{Z}$ and $Q\in L_i$ we have that
    $$(\Omega^\varphi\oplus T^\varphi_{i,Q})+m\wgamma=\Omega^\varphi\oplus T^\varphi_{i,Q+m\wgamma}$$
    for every $1\le i\le n$.
\end{enumerate}
Letting $T_{i,Q}=\{t\in T:t^\varphi\in T^\varphi_{i,Q}\}$ for all $1\le i\le n$ and $Q\in L_i$, we have
\begin{equation}\label{eq:decT}
    T_i=\{t\in T:t^\varphi\in T^\varphi_i\}=\bigsqcup_{Q\in L_i} T_{i,Q}, \quad 1\le i\le n.
\end{equation}
In the remainder of this section, we show that the collections $T_{i,Q}$ satisfy properties analogous to (1.) and (2.) that hold for $T^\varphi_{i,Q}$.
Note that $T_{i,Q}$ is not necessarily a subset of some coset of $D_i$.
Let $N$ be the the least common multiple of $|k_1|,\dots,|k_n|$ and define\footnote{Recall that the set of proper vertices $W(D)$ of a set $D\subseteq  \mathbb{R}^2$ is defined in Section~\ref{sec:Definitions}.} $W_{i,Q}=W(\Omega\oplus T_{i,Q})$.

\subsection{Columns}

The following proposition allows the reduction of Theorem \ref{thm:MainLocallyFin} to the case when $T$ does not contain any columns (i.e., $T_c$ is empty). 

\begin{proposition}\label{pr:WLOGcolumns} 
    Let $T$ be a tiling of $E$ by $\Omega$ that has finite local complexity and assume that $T_c\not=\emptyset$.
    Then $p_1(\Omega\oplus T_c)$ and $p_1(T_c)$ are $\beta\Z$-periodic for some $\beta>0$ and $T_c$ is $(0,\alpha)\Z$-periodic for some $\alpha>0$.
\end{proposition}
\begin{proof}
By Theorem \ref{thm:PeriodicityMain} (and Remark \ref{rem:PeriodicityMain}), we have that the set $\Omega\oplus T_c$ is $(\beta',0)\Z$-periodic, for some $\beta'>0$, thus $p_1(\Omega\oplus T_c)$ is $\beta'\Z$-periodic.
Observe that by the definition of $T_c$, we have that $p_1(T_c)$ is a tiling of $p_1(\Omega\oplus T_c)$ by $p_1(\Omega)$.
Adapting Newman’s pigeonhole-principle argument \cite{Newman} to tilings of periodic subsets of $\mathbb R$ by finite unions of intervals, we obtain that $p_1(T_c)$ is $\beta\mathbb Z$-periodic for some $\beta>0$.
It remains to argue that if $p_1(\overline{\Omega\oplus T_c})\not =\R$, then $\beta\in \Q\beta'$.
This follows by a routine application of the ergodic theorem.

In addition, letting $x \in \Omega$ be such that $p_1(x)\not\in p_1(V(\Omega))$ and applying \cite[Theorem 1]{LW} to the region $$\Omega_x \coloneqq p_2(\Omega\cap (\{p_1(x)\}\times \R))\subseteq \R,$$
we see that there exists $\alpha>0$ such that every $p_2(P)\in \Tile(\Omega_x ;\R)$, where $P= \left(\{x_{P}\}\times \R\right) \cap T_c$ for some $x_P\in p_1(T_c)$, is $\alpha\Z$-periodic.
In turn, we have that $T_c$ is $(0,\alpha)\Z$-periodic, as a union of $(0,\alpha)\Z$-periodic sets.
\end{proof}

In order to prove part (i) of Theorem~\ref{thm:MainLocallyFin}, using Proposition \ref{pr:WLOGcolumns}, we may assume that $T_c=\emptyset$ without loss of generality, since the set $E\setminus (\Omega\oplus T_c)$ is periodic of the form $X'\times \mathbb{R}$ for some open set $X'\subseteq  \mathbb{R}$.
We note that reducing part (ii) of Theorem~\ref{thm:MainLocallyFin} to the case $T\not=T_c$ requires an additional argument, as will be discussed separately later in Subsection~\ref{subsec:ProofOfThmLocallyFin}.

\subsection{Gaps and periodicity}
In this section, we assume that $T_c=\emptyset$.

\begin{definition}[Gaps]
    We say that $T_{i,Q}$ \emph{has gaps} if the equivalence relation on $T_{i,Q}$  that is the transitive closure of the set
    $$\{(t,t')\in T_{i,Q}\times T_{i,Q}: \|t-t'\| <60\diam(\Omega)\}$$
    has more than one equivalence class. 

    Let $T_g \subseteq  T$ be the set of all $t\in T$ such that there exist $1\le i\le n$, $Q\in L_i$ and $m\in \Z$ such that $T_{i,Q}$ has gaps and $t\in T_{i,Q+m\widetilde{\gamma}}$.    
\end{definition}

\begin{lemma}[Rationalization of $T_g$ is periodic]\label{lem:GapsPeriodicity}
    We have that $(T_g)^\varphi$ is $\langle N\widetilde{\lambda},N\widetilde{\gamma}\rangle$-periodic.
\end{lemma}

To prove Lemma \ref{lem:GapsPeriodicity}, we will use the following proposition.

\begin{proposition}\label{pr:GapsOne}
    Suppose that $T_{i,Q}$ has gaps and let $S\subseteq \R^2$ be such that $\Omega^\varphi\oplus S=\Omega^\varphi\oplus T^\varphi_{i,Q}$.
    Then $S=T^\varphi_{i,Q}$.
\end{proposition}
\begin{proof}
    Recall that $T^\varphi_{i,Q}\subseteq  Q$, where $Q$ is a coset of the linear subspace $D_i$.
    Consider the canonical enumeration $T^\varphi_{i,Q}=(t^\varphi_j)_{j\in \mathbb{Z}}$.
    Since $T_{i,Q}$ has gaps, there must be $j\in \Z$ such that $\|t^\varphi_j-t^\varphi_{j+1}\| \ge 30\diam(\Omega)$.
    Indeed, if $\|t^\varphi_j-t^\varphi_{j+1}\| < 30\diam(\Omega)$ for every $j\in \mathbb{Z}$,
    by Theorem~\ref{thm:GoodFiniteLocalComplexity}~(10) and \eqref{eq:LargeParameter}, we have $\|t_j-t_{j+1}\| < M$, since $\|t^\varphi_j-t^\varphi_{j+1}\| < 30\diam(\Omega)\le M/2$ for every $j\in \Z$.
    By Theorem~\ref{thm:GoodFiniteLocalComplexity}~(9) and \eqref{eq:LargeParameter}, we would then have 
    $$\|(t_j^\varphi-t_{j+1}^\varphi)-(t_j-t_{j+1})\| <1/M,$$
    which implies that $\|t_j-t_{j+1}\| <30\diam(\Omega)+1/M<31\diam(\Omega)$, contradicting our assumption that $T_{i,Q}$ has gaps.
    Thus, we can fix $j_0\in \Z$ such that $\|t^\varphi_{j_0}-t^\varphi_{j_0+1}\|$ is bounded below by  $30\diam(\Omega)$, which in turn, by Theorem~\ref{thm:GoodFiniteLocalComplexity}~(3), is bounded below by $15\diam(\Omega^\varphi)$. Thus, as $T^\varphi_{i,Q}\subseteq  Q$, we have that the sets
    $$F_0=\Omega^\varphi\oplus \{t^\varphi_{j}:j\le j_0\} \quad \text{and} \quad F_1=\Omega^\varphi\oplus \{t^\varphi_{j}:j>j_0\}$$
    are of distance at least $2\diam(\Omega^\varphi)$ from each other.
    
    Now consider the linear subspace $D^\perp_i$ that is orthogonal to $D_i$ and let $P$ be its coset intersecting $\overline{F_1}$ and such that $p=P\cap Q$ is minimal (in the canonical order induced on $Q$).
    Then it follows that $t^\varphi_{j_0+1}$ is the only possible choice of a translation of $\Omega^\varphi$ that satisfies both $\Omega^\varphi+t^\varphi_{j_0+1}\subseteq  F_1$ and $P\cap \overline{F_1}\subseteq  \overline{\Omega^\varphi+t^\varphi_{j_0+1}}$. Thus, by induction, we have that the tilings of $F_0$ and $F_1$ by $\Omega^\varphi$ are unique. Hence if $\Omega^\varphi \oplus S =\Omega^\varphi \oplus T^\varphi_{i,Q}$  then we must have $S=T^\varphi_{i,Q}$.
\end{proof}

We can now conclude the proof of Lemma~\ref{lem:GapsPeriodicity}.

\begin{proof}[Proof of Lemma~\ref{lem:GapsPeriodicity}]
    Suppose that $T_{i,Q}$ has gaps.
    By Corollary~\ref{cor:MainRational} and the construction of the decomposition, we have that 
    $$\Omega^\varphi\oplus T^\varphi_{i,Q}+m\widetilde{\gamma}=\Omega^\varphi\oplus T^\varphi_{i,Q+m\widetilde{\gamma}}$$
    for every $m\in \Z$.
    Hence,for every $m\in \Z$ 
    $$\Omega^\varphi\oplus T^\varphi_{i,Q}=\Omega^\varphi\oplus \left(T^\varphi_{i,Q+m\widetilde{\gamma}}-m\widetilde{\gamma}\right),$$
    which, by Proposition~\ref{pr:GapsOne}, implies  that $T^\varphi_{i,Q}+m\widetilde{\gamma}=T^\varphi_{i,Q+m\widetilde{\gamma}}$.
    Combined with the fact that $T^\varphi_{i,Q}$ is $(k_i\widetilde\lambda+\ell_i\widetilde\gamma)$-periodic (recall $k_i\not=0$), we conclude that $\bigsqcup_{m\in \mathbb{Z}} T^\varphi_{i,Q+m\widetilde{\gamma}}$ is $\langle N\widetilde{\lambda},N\widetilde{\gamma}\rangle$-periodic.
    As this is true for every $T_{i,Q}$ that has gaps, $T^\varphi_g$ is $\langle N\widetilde{\lambda},N\widetilde{\gamma}\rangle$-periodic, as a union of $\langle N\widetilde{\lambda},N\widetilde{\gamma}\rangle$-periodic sets.
\end{proof}

\subsection{No gaps}

In this section, we assume that $T_c=\emptyset$, and let $T_{ng}\coloneqq T\setminus T_g$.

\begin{lemma}[Structure of $T_{ng}$]\label{lem:MainShifts}
    Let $1\le i\le n$ and $Q\in L_i$ be such that $T_{i,Q}\subseteq  T_{ng}$ and $W_{i,Q}$ is nonempty.
    Then there is a nonzero vector $\alpha_{i,Q+\mathbb{Z}\widetilde{\gamma}}\in \mathbb{R}^2$ such that $T_{i,Q+m\widetilde{\gamma}}$ is $\langle\alpha_{i,Q+\mathbb{Z}\widetilde{\gamma}}\rangle$-periodic and
    $$\Omega\oplus T_{i,Q}+\beta_{Q,m}=\Omega\oplus T_{i,Q+m\widetilde{\gamma}}$$  for every $m\in \mathbb{Z}$,
    where $\beta_{Q,m}=w_0-v_0$ for any $w_0\in W_{i,Q+m\widetilde{\gamma}}$ and $v_0\in W_{i,Q}$ such that $(w_0-v_0)^\varphi=m\widetilde{\gamma}$.
\end{lemma}

The proof of this Lemma is concluded in Section \ref{subsubsec:ProofMainLemma}. In the proof, we distinguish two cases depending on whether $D_i\in \{\mathbb{R}(1,0),\mathbb{R}(0,1)\}$ or not.
We begin with a definition.

\begin{definition}[Links]
    Let $S\subseteq  T$.
    For $r>0$, we define the \emph{$r$-link equivalence relation on $W(\Omega\oplus S)$} to be the equivalence relation that is generated by the pairs $v,w\in W(\Omega\oplus S)$ such that $\|v-w\|<r$.
    We denote the $r$-link equivalence class of $v\in W(\Omega\oplus S)$ as $[v]_r$.
    We say that $S$ is \emph{$r$-linked} if the $r$-link equivalence relation on $W(\Omega\oplus S)$ consists of a single non-empty equivalence class.
\end{definition}

We extend the definitions of gaps and links to arbitrary nonempty packings $\Omega\oplus S$ in the same way.
Note that while the definitions of gaps and links have similar flavor, these notions are different.
For instance, if $\Omega$ can tile a column, then it is easy to produce an example of a vertical packing $S\subseteq  \mathbb{R}(0,1)$ such that $S$ has no gaps, but it is not $r$-linked for any $r>0$.

The following observation will be our main tool.

\begin{claim}\label{cl:LinkedEquivalence}
    Let  $1\le i\le n$, $Q\in L_i$, $m\in \mathbb{Z}$. Suppose that $w_0\in W_{i,Q+m\widetilde{\gamma}}$ and $v_0\in W_{i,Q}$ are such that $(w_0-v_0)^\varphi=m\widetilde{\gamma}$.
    Then we have that $u+\beta\in W_{i,Q+m\widetilde{\gamma}}$ and 
    \begin{equation}\label{eq:GoodPointsZero}
        \mathcal{B}(u,(\Omega\oplus T_{i,Q}))=\mathcal{B}(u+\beta,(\Omega\oplus T_{i,Q+m\widetilde{\gamma}}))
    \end{equation}
    for every $u\in [v_0]_{M/3}$, where $\beta=w_0-v_0$ and $M>0$ as in \eqref{eq:LargeParameter}. 
    Moreover, if $v\in [v_0]_{M/3}$, then the unique $w\in W_{i,Q+m\widetilde{\gamma}}$ that satisfies $(w-v)^\varphi=m\widetilde{\gamma}$ is equal to $v+\beta$.
\end{claim}
\begin{proof}
    By Theorem~\ref{thm:GoodFiniteLocalComplexity}~(7), we have that
    $$\mathcal{B}(v_0,\Omega\oplus T_{i,Q})=\mathcal{B}(w_0,\Omega\oplus T_{i,Q+m\widetilde{\gamma}}).$$
    Write $X_0\subseteq  W_{i,Q}$ for the set of $v\in W_{i,Q}$ such that $v+\beta\in W_{i,Q+m\widetilde{\gamma}}$ and that satisfy \eqref{eq:GoodPointsZero}.
    We show that $[v_0]_{M/3}\subseteq  X_0$.

    Suppose for a contradiction that $[v_0]_{M/3}\setminus X_0\not=\emptyset$.
    As $v_0\in X_0$, we find $v\in X_0$ and $v'\in [v_0]_{M/3}\setminus X_0$ such that $\|v-v'\| <M/3$.
    Then, from Theorem~\ref{thm:GoodFiniteLocalComplexity}~(9) and \eqref{eq:LargeParameter} it follows that $\|v^\varphi-(v')^\varphi\| <M/2$.
    Write $w=v+\beta$ and observe that $w^\varphi=v^\varphi+m\widetilde{\gamma}$.
    Set also $w'\in W_{i,Q+m\widetilde{\gamma}}$ such that $(w')^\varphi=(v')^\varphi+m\widetilde{\gamma}$.
    Observe that by Theorem~\ref{thm:GoodFiniteLocalComplexity}~(7)
    $$\mathcal{B}(v',\Omega\oplus T_{i,Q})=\mathcal{B}(w',\Omega\oplus T_{i,Q+m\widetilde{\gamma}}).$$
    We claim that $w'=v'+\beta$.
    Indeed, by Theorem~\ref{thm:GoodFiniteLocalComplexity}~(13), we have that $v'-v=w'-w$.
    Hence $w'-v'=w-v=\beta$, so $v'\in X_0$, which contradicts out initial assumption.

    It remains to argue that $\beta$ does not depend on the choice of $v\in [v_0]_{M/3}$.
    Let $v\in [v_0]_{M/3}$ and $w\in W_{i,Q+m\widetilde{\gamma}}$ be such that $m\widetilde{\gamma}=w^\varphi-v^\varphi$.
    Then by the previous argument, we have that $v+\beta=w'\in W_{i,Q+m\widetilde{\gamma}}$ for some $w'\in W_{i,Q+m\widetilde{\gamma}}$, and
    $$(w')^\varphi=v^\varphi+(w^\varphi_0-v_0^\varphi)=v^\varphi+m\widetilde{\gamma}=w^\varphi.$$
    By Theorem~\ref{thm:GoodFiniteLocalComplexity}~(11), we conclude that $w'=w$.
    In particular, $w-v=\beta$ as desired.
\end{proof}

The following observation will be used in all the arguments.
We extend in the natural way the definition of local combinatorial structure, Definition~\ref{def:LocalCombStructure}, to all $x\in \partial F$, where $F=\Omega\oplus S$ for some $\emptyset\not =S\subseteq \R^2$.

\begin{claim}\label{cl:NecessaryForRegions}
    Let $F_0,F_1\subseteq  \mathbb{R}^2$ be non-empty regular open sets and $S_i\in\Tile(\Omega; F_i)$, $i=0,1$.
    Then $F_0+\beta=F_1$ for some $\beta\in \mathbb{R}^2$ if and only if $\partial(F_0)+\beta=\partial(F_1)$ and $\mathcal{B}(x,F_0)=\mathcal{B}(x+\beta,F_1)$ for every $x\in \partial(F_0)$.
\end{claim}
\begin{proof}
    The implication from left to right is clear.
    For the converse, suppose that $\beta\in \mathbb{R}^2$ is such that $\partial(F_0)+\beta=\partial(F_1)$ and $\mathcal{B}(x,F_0)=\mathcal{B}(x+\beta,F_1)$ for every $x\in \partial(F_0)$.
    We may assume that $F_0,F_1\not=\mathbb{R}^2$ as in the case the claim follows as well.
    Let $x\in F_0$.
    As $F_0=\Omega\oplus S_0\not=\mathbb{R}^2$ and $\Omega$ is a polygonal set, there is a line segment $L=[x,y)\subseteq  F_0$ that connects $x$ and some $y\in \partial(F_0)$.
    By the assumption, we have that $y+\beta\in\partial(F_1)$ and $(\overline{L}+\beta)\cap \partial(F_1)=\{y+\beta\}$.
    It follows from $\mathcal{B}(y,F_0)=\mathcal{B}(y+\beta,F_1)$ that $L+\beta\subseteq  F_1$, which implies that $x+\beta\in F_1$.
    Analogous argument shows that if $x\in F_1$, then $x-\beta\in F_0$ and the proof is finished.    
\end{proof}

\subsubsection{\bf The case $D_i\not\in \{\mathbb{R}(1,0),\mathbb{R}(0,1)\}$}

\begin{proposition}\label{pr:NotLinkedTwo}
    Suppose that $D_i\not\in \{\mathbb{R}(1,0),\mathbb{R}(0,1)\}$ and let $Q\in L_i$ be such that $T_{i,Q}\subseteq  T_{ng}$.
    Then $T_{i,Q+m\widetilde{\gamma}}$ is $M/3$-linked for every $m\in \mathbb{Z}$ for every $M$ as in \eqref{eq:LargeParameter}.
\end{proposition}
\begin{proof}
    First, observe that the assumption $T_{i,Q}\subseteq  T_{ng}$ implies that $T_{i,Q}$ has no gaps.
    Let $L'$ be a coset of $D_i$ that maximizes the distance from $Q$ and has non-empty intersection with $\overline{\Omega^\varphi\oplus T^\varphi_{i,Q}}$.
    Note that $L'$ exists as $T^\varphi_{i,Q}\subseteq  Q$ and $\Omega$ is bounded.
    Since $D_i\not\in \{\mathbb{R}(1,0),\mathbb{R}(0,1)\}$, $L'$ has a maximal distance from $Q$, and $T^\varphi_{i,Q}\subseteq  Q$, it must be the case that there is $v_0\in V(\Omega)$ such that for every $t\in T_{i,Q}$ we have that $v_0^\varphi+t^\varphi\in L'$ and $v_0^\varphi+t^\varphi\in W(\Omega^\varphi\oplus T^{\varphi}_{i,Q})$.

    By Theorem~\ref{thm:GoodFiniteLocalComplexity} parts (3) and (7), we have that $v_0+t\in W_{i,Q}$ for every $t\in T_{i,Q}$.
    As $T_{i,Q}$ does not have gaps, we conclude that $v_0+t'\in [v_0+t]_{60\diam(\Omega)}$ for every $t,t'\in T_{i,Q}$.
    Using the fact that $W_{i,Q}\subseteq  V(\Omega)+T_{i,Q}$, we conclude that $T_{i,Q}$ is $63\diam(\Omega)$-linked, and hence it is $M/3$-linked.
    Finally, as $T_{i,Q}\subseteq  T_{ng}$ implies that $T_{i,Q+m\widetilde{\gamma}}\subseteq  T_{ng}$ for every $m\in \mathbb{Z}$, the claim follows.
\end{proof}

\begin{proposition}\label{pr:NoGapsTwo}
    Let $1\le i\le n$ and $Q\in L_i$ be such that $D_i\not\in \{\mathbb{R}(1,0),\mathbb{R}(0,1)\}$ and $T_{i,Q}\subseteq  T_{ng}$.
    Then
    $$\Omega\oplus T_{i,Q}+\beta_{Q,m}=\Omega\oplus T_{i,Q+m\widetilde{\gamma}}$$
    for every $m\in \mathbb{Z}$, where $\beta_{Q,m}=w_0-v_0$ for any $w_0\in W_{i,Q+m\widetilde{\gamma}}$ and $v_0\in W_{i,Q}$ such that $(w_0-v_0)^\varphi=m\widetilde{\gamma}$.
\end{proposition}
\begin{proof}
    Let $m\in \mathbb{Z}$. Observe that the assumption $D_i\not\in \{\mathbb{R}(1,0),\mathbb{R}(0,1)\}$ combined with Theorem~\ref{thm:GoodFiniteLocalComplexity}~(1), (7) and (11) imply that any maximal connected line segment $P\subseteq  \partial(\Omega\oplus T_{i,Q+m\widetilde{\gamma}})$ is finite and must be of the form $P_{a,b}$ for some $a,b\in W_{i,Q+m\widetilde{\gamma}}$.
    By Proposition~\ref{pr:NotLinkedTwo}, we have that $[v_0]_{M/3}=W_{i,Q}$ and $[w_0]_{M/3}=W_{i,Q+m\widetilde{\gamma}}$.
    By Claim \ref{cl:LinkedEquivalence}, we have that $u\in W_{i,Q}$ if and only if $u+\beta_{Q,m}\in W_{i,Q+m\widetilde{\gamma}}$, and
    \begin{equation}\label{eq:GoodPoints}
        \mathcal{B}(u,(\Omega\oplus T_{i,Q}))=\mathcal{B}(u+\beta_{Q,m},(\Omega\oplus T_{i,Q+m\widetilde{\gamma}}))
    \end{equation}
    for every $u\in W_{i,Q}$.
    Moreover, $\beta_{Q,m}$ does not depend on the choice of $v_0$ and $w_0$. 
    As every line segment in $\partial(\Omega\oplus T_{i,Q})$ and $\partial(\Omega\oplus T_{i,Q+m\widetilde{\gamma}})$ is finite, we conclude that $\Omega\oplus T_{i,Q}+\beta_{Q,m}=\Omega\oplus T_{i,Q+m\widetilde{\gamma}}$ by Claim~\ref{cl:NecessaryForRegions}.
\end{proof}

\newcommand{\cB}{\mathcal{B}}

\subsubsection{\bf The case $D_i\in \{\mathbb{R}(1,0),\mathbb{R}(0,1)\}$}
We assume, without loss of generality, that $D_i=\mathbb{R}(0,1)$.
The difficulty in proving the analogue of Proposition~\ref{pr:NoGapsTwo} is that $T_{i,Q}$ need not be $M/3$-linked, and $\partial(\Omega\oplus T_{i,Q})$ might contain arbitrarily long, even infinite, line segments.

We start with distinguishing two cases, depending on whether $\Omega$ can tile a column or not. 
We say that $S$ is a {\bf \emph{vertical packing}} if $S\subseteq  \mathbb{R}^2$ lies in a coset of $\mathbb{R}(0,1)$ and $\Omega\oplus S$ is a packing.

\begin{proposition}\label{pr:NotLinkedOne}
    Suppose that $\Omega$ cannot tile a column.
    Then there is $\ell_\Omega>0$, that only depends on $\Omega$, such that every vertical packing $S\subseteq  \mathbb{R}(0,1)$ that does not have gaps is $\ell_\Omega$-linked.
\end{proposition}

\begin{proof}
    Let $\ell>100\diam(\Omega)$ and suppose that there is a vertical packing $S_\ell$ that does not have gaps and that is not $\ell$-linked.
    By the definition, there are $v,w\in W(\Omega\oplus S_\ell)$ such that $w\not\in [v]_\ell$.
    We may without loss of generality assume that $S_\ell\subseteq  \mathbb{R}(0,1)$, $p_2(w)=-\ell/2$, $p_2(v)-p_2(w)\ge \ell-2\diam(\Omega)$ and there is no $w'\in W(\Omega\oplus S_\ell)$ such that $p_2(w)<p_2(w')<p_2(v)$.
    Define
    $$X\coloneqq\left(p_1(\Omega\oplus S_\ell)\times \mathbb{R}\right)\cap \left(\mathbb{R}\times [p_2(w),p_2(v)]\right).$$
    Then we either have $X\cap (\Omega\oplus S_\ell)$ is a null set, or $X\subseteq  \Omega\oplus S_\ell$.
    The former case is impossible, as $S_\ell$ does not have gaps and $\ell>100\diam(\Omega)$.
    Now observe that the latter implies that $$p_1(\Omega)\times [-\ell/3,\ell/3]\subseteq  \Omega\oplus S_\ell,$$
    since $\ell/3\le \ell/2-2\diam(\Omega)$.
    By compactness, we see that if no $\ell>0$ satisfies the requirements, then any subsequential limit of $S_\ell$ witnesses that $\Omega$ can tile a column, a contradiction.
\end{proof}

In what follows, we will assume that if $\Omega$ cannot tile a column, then the parameter $M$ in \eqref{eq:LargeParameter} additionally satisfies $M>3\ell_\Omega$, where $\ell_\Omega$ is as in Proposition \ref{pr:NotLinkedOne}. 

\begin{proposition}\label{pr:NoGapsNoColumn}
    Suppose that $\Omega$ cannot 
    tile a column and that $D_i= \mathbb{R}(0,1)$ for some $1\leq i\leq n$. If  $Q\in L_i$ is such that $T_{i,Q}\subseteq  T_{ng}$, then
    $$\Omega\oplus T_{i,Q}+\beta_{Q,m}=\Omega\oplus T_{i,Q+m\widetilde{\gamma}}$$
    for every $m\in \mathbb{Z}$, where $\beta_{Q,m}=w_0-v_0$ for any $w_0\in W_{i,Q+m\widetilde{\gamma}}$ and $v_0\in W_{i,Q}$ such that $(w_0-v_0)^\varphi=m\widetilde{\gamma}$.
\end{proposition}
\begin{proof}
    Observe first that $T_{i,Q}$ is a vertical packing.
    Indeed, if $s,t\in T_{i,Q}$ satisfy $\|s-t\|<60\diam(\Omega)$, then by Theorem~\ref{thm:GoodFiniteLocalComplexity}~(9) and the choice of $M$, we have $\|s^\varphi-t^\varphi\|<M/2$.
    Since $p_1(s^\varphi)=p_1(t^\varphi)$, Theorem~\ref{thm:GoodFiniteLocalComplexity}~(12) gives $p_1(s)=p_1(t)$.
    As $T_{i,Q}$ has no gaps, we conclude that $p_1$ is constant on $T_{i,Q}$.

    By Proposition~\ref{pr:NotLinkedOne}, we have that $T_{i,Q+m\widetilde{\gamma}}\subseteq  T_{ng}$ is $\ell_\Omega$-linked for every $m\in \mathbb{Z}$.
    Thus, we have that $[v_0]_{M/3}=W_{i,Q}$ and $[w_0]_{M/3}=W_{i,Q+m\widetilde{\gamma}}$.
    Additionally, recall that $W_{i,Q}\not=\emptyset$ as $\Omega$ cannot tile a column.

    By Claim \ref{cl:LinkedEquivalence}, we have that $\beta_{Q,m}$ is well defined, and has the property that $u\in W_{i,Q}$ if and only if $u+\beta_{Q,m}\in W_{i,Q+m\widetilde{\gamma}}$, and
    \begin{equation}\label{eq:GoodPointsTwo}
        \mathcal{B}(u,(\Omega\oplus T_{i,Q}))=\mathcal{B}(u+\beta_{Q,m},(\Omega\oplus T_{i,Q+m\widetilde{\gamma}}))
    \end{equation}
    for every $u\in W_{i,Q}$.    

    Unlike in the proof of Proposition~\ref{pr:NoGapsTwo}, we cannot directly apply Claim~\ref{cl:NecessaryForRegions} as $\partial(\Omega\oplus T_{i,Q})$ might contain infinite line segments disjoint from $W_{i,Q}$.
    Observe that any maximal infinite line segment contained in $\partial(\Omega\oplus T_{i,Q})$ is a coset of $\mathbb{R}(0,1)$.
    This follows from the fact that $T_{i,Q}\subseteq  T_{ng}$, Theorem~\ref{thm:GoodFiniteLocalComplexity}, and the fact that every maximal infinite line segment contained in $\partial(\Omega^\varphi\oplus T^\varphi_{i,Q})$ is a coset of $\mathbb{R}(0,1)$ due to the periodicity of $T^\varphi_{i,Q}$.
    
    We handle this issue by the following auxiliary claim, whose proof is postponed after this proof.
   
    \begin{claim}\label{cl:Shifts}
        We have $p_1(\Omega\oplus T_{i,Q})+p_1(\beta_{Q,m})=p_1(\Omega\oplus T_{i,Q+m\widetilde{\gamma}})$ for every $m\in \mathbb{Z}$.
    \end{claim}

    By \eqref{eq:GoodPointsTwo}, the boundary segments of $\Omega\oplus T_{i,Q}$ that intersect $W_{i,Q}$ correspond under translation by $\beta_{Q,m}$ with the boundary segments of $\Omega\oplus T_{i,Q+m\widetilde{\gamma}}$ that intersect $W_{i,Q+m\widetilde{\gamma}}$.
    Hence, it remains to consider a coset $P\subseteq\partial(\Omega\oplus T_{i,Q})$ of $\mathbb R(0,1)$ such that $W_{i,Q}\cap P=\emptyset$.
    Observe that $p_1(\Omega\oplus T_{i,Q})=p_1(\Omega)+p_1(T_{i,Q})$ is a finite union of finite intervals.
    Suppose that $P\subseteq  \partial(\Omega\oplus T_{i,Q})$ is a coset of $\mathbb{R}(0,1)$ such that $W_{i,Q}\cap P=\emptyset$.
    It follows that $p_1(P)$ is an end point of one of the finite intervals of $p_1(\Omega\oplus T_{i,Q})$.
    This implies, by Claim~\ref{cl:Shifts} and \eqref{eq:GoodPointsTwo}, that $P+\beta_{Q,m}\subseteq  \partial(\Omega\oplus T_{i,Q+m\widetilde{\gamma}})$ and $(P+\beta_{Q,m})\cap W_{i,Q+m\widetilde{\gamma}}=\emptyset$.
    Additionally, we have that
    $$\mathcal{B}(x,\Omega\oplus T_{i,Q})=\mathcal{B}(x+\beta_{Q,m},\Omega\oplus T_{i,Q+m\widetilde{\gamma}})$$
    for every $x\in P$.
    Combined with \eqref{eq:GoodPointsTwo}, we conclude by Claim~\ref{cl:NecessaryForRegions} that $\Omega\oplus T_{i,Q}+\beta_{Q,m}=\Omega\oplus T_{i,Q+m\widetilde{\gamma}}$ as desired.
\end{proof}

\begin{proof}[Proof of Claim~\ref{cl:Shifts}:]
    Suppose first that $\partial(\Omega\oplus T_{i,Q})$ does not contain a coset of $\mathbb{R}(0,1)$.
    Note that in this case there are $v,w\in W_{i,Q}$ such that $p_1(v)\le p_1(\Omega\oplus T_{i,Q})\le p_1(w)$. 
    By \eqref{eq:GoodPointsTwo}, we have that $v+\beta_{Q,m},w+\beta_{Q,m}\in W_{i,Q+m\widetilde{\gamma}}$, which implies the claim as $p_1(v+\beta_{Q,m})\le p_1(\Omega\oplus T_{i,Q+m\widetilde{\gamma}})\le p_1(w+\beta_{Q,m})$.
    
    Next, suppose that $\partial(\Omega\oplus T_{i,Q})$ contains a coset of $\mathbb{R}(0,1)$, say $P$.
    Let $v_0\in V(\Omega)$ and $t\in T_{i,Q}$ be such that $v_0+t\in P$, and define $m_\Omega=m_1(p_2(\overline{\Omega}\cap (\{p_1(v_0)\}\times \mathbb{R})))$.
    As $T_{i,Q}$ is not a column, there is $v'\in V(\Omega)$ that admits a sequence $\{s_n\}_{n\in \N}\subseteq \R$ that converges to $p_1(v')$ such that
    $$0<\lim_{n\to \infty} m_1(p_2(\overline{\Omega}\cap (\{s_n\}\times \mathbb{R})))<m_\Omega.$$
    Let $v_1\in V(\Omega)$ have this property and be additionally such that $p_1(v_1)$ is minimal. 
    By a simple volume argument, we have that 
    $$(p_1(v_1+T_{i,Q})\times \mathbb{R})\cap W_{i,Q}\not=\emptyset$$
    and $p_1(v_1+T_{i,Q})\le p_1(W_{i,Q})$.
    
    By Theorem~\ref{thm:GoodFiniteLocalComplexity}~(7), we have that $\partial(\Omega\oplus T_{i,Q+m\widetilde{\gamma}})$ contains a coset of $\mathbb{R}(0,1)$ as well.
    Another application of the volume argument together with Theorem~\ref{thm:GoodFiniteLocalComplexity}~(7) and the fact that $\Omega^\varphi\oplus T^\varphi_{i,Q}+m\widetilde{\gamma}=\Omega^\varphi\oplus T^\varphi_{i,Q+m\widetilde{\gamma}}$ shows that
    $$(p_1(v_1+T_{i,Q+m\widetilde{\gamma}})\times \mathbb{R})\cap W_{i,Q+m\widetilde{\gamma}}\not=\emptyset$$
    and $p_1(v_1+T_{i,Q+m\widetilde{\gamma}})\le p_1(W_{i,Q+m\widetilde{\gamma}})$.
    This gives $p_1(v_1+T_{i,Q})+p_1(\beta_{Q,m})=p_1(v_1+T_{i,Q+m\widetilde{\gamma}})$ by \eqref{eq:GoodPointsTwo}, which implies that
    $$p_1(T_{i,Q})+p_1(\beta_{Q,m})=p_1(T_{i,Q+m\widetilde{\gamma}})$$
    and the claim follows.
\end{proof}

We are left with the case when $\Omega$ can tile a column.
Recall that we assume that $T_{i,Q}$ is not a column, therefore by the periodicity of $T^\varphi_{i,Q}$ and Theorem~\ref{thm:GoodFiniteLocalComplexity}~(7), we have that all maximal connected line segments of $\partial(\Omega\oplus T_{i,Q})$ are finite.
However, a priori, there is no bound that depends only on $\Omega$, on the maximal length of such line segments, so we cannot guarantee that $W_{i,Q}$ is $M'$-linked for any $M'>0$ fixed at the beginning of the construction.
On a high level, this is remedied by the elementary observation that any tiling of a sufficiently large interval in 
$\R$ by translates of a finite union of intervals must exhibit a nontrivial period.

We start with an auxiliary claim.

\begin{proposition}\label{pr:NotLinkedThree}
    Suppose that $\Omega$ can tile a column.
    Then there is $r_\Omega>0$, that only depends on $\Omega$, such that if $S$ is a vertical packing that does not have gaps and is not $r_\Omega$-linked, then for every $x\in W(\Omega\oplus S)$ there is $v\in W(\Omega\oplus S)\cap [x]_{r_\Omega}$ such that $p_1(v)\le p_1(\Omega\oplus S)$. 
\end{proposition}
\begin{proof}
    Suppose for a contradiction that there is no $r>0$ with this property.
    In particular, for every sufficiently large $n\in \mathbb{N}$, we can find a vertical packing $S_n\subseteq  \mathbb{R}(0,1)$ that does not have gaps and is not $n$-linked, and $x_n\in W(\Omega\oplus S_n)$ such that $p_1(v)\not\leq p_1(\Omega\oplus S_n)$ for every $v\in W(\Omega\oplus S_n)\cap [x_n]_{n}$.
    Observe that, as $S_n$ is not $n$-linked, there is $y_n\in [x_n]_{n}$ that either maximizes, or minimizes $p_2$, i.e., $$p_2(y_n)\in \{\max Y_n, \min Y_n\},$$ where $Y_n=\{p_2(y)\colon y \in W(\Omega\oplus S_n)\cap [x_n]_{n}\}$.
    Let us assume without loss of generality that $p_2(y_n)=\max Y_n$.
    After possibly shifting $S_n$, we may assume that $p_2(y_n)=0$.
    Additionally, since $V(\Omega)$ is finite, after possibly passing to a subsequence, we may assume that $y_n=y=\hat{v}+\hat{s}$ for some $y\in \mathbb{R}^2$, $\hat{v}\in V(\Omega)$ and $\hat{s}\in S_n$ for every $n\in \mathbb{N}$.
    As $S_n$ does not have gaps, it follows that $$p_1(\overline{\Omega}\oplus S_n)\times [0,n-\diam(\Omega)]\subseteq  \overline{\Omega}\oplus S_n $$ for every large enough $n\in \mathbb{N}$.  
    Let $S$ be a subsequential weak limit of $(S_n)_{n\in \mathbb{N}}$.
    Observe that
    $$p_1(\overline{\Omega}\oplus S)\times [0,\infty)\subseteq  \overline{\Omega}\oplus S$$
    and that $\hat{s}\in S$.
    
    Suppose first that $y\not\in W(\Omega\oplus S)$.
    This can only happen if there are $s\in S$ and $s_n\in S_n$ for every $n\in \mathbb{N}$ such that $\|s_n-s\|\to 0$, $y\in V(\Omega)+s$ and $y\not\in V(\Omega)+s_n$.
    Observe that $s\not=\hat{s}$ and $\|s-\hat{s}\|\le 2\diam(\Omega)$.
    Let $w\in V(\Omega)$ be such that $p_1(w)\le p_1(\Omega)$.
    As $s_n\not=s$ for every $n\in \mathbb{N}$, we get that the set
    $$\{p_2(w+\hat{s})-p_2(w+s_n):n\in \mathbb{N}\}$$
    is infinite.
    In particular as $E(\Omega)$ is finite, we find arbitrarily large $n\in \mathbb{N}$ such that $\|s-s_n\|\le \diam(\Omega)$ and $W(\Omega\oplus S_n)\cap (P_{\hat{s},s_n}+w)\not=\emptyset$.
    Let $v\in W(\Omega\oplus S_n)\cap (P_{\hat{s},s_n}+w)$.
    Then we have that
    \begin{equation*}
        \begin{split}
            \|v-y\|\le & \ \|v-(w+s_n)\|+\|(w+s_n)-(w+s)\|+\|(w+s)-(w+\hat{s})\|+\|(w+\hat{s})-y\|\\
            \le & \ 5\diam(\Omega).
        \end{split}
    \end{equation*}
    In particular, $v\in [y]_{n}$ in $S_n$ for sufficiently large $n\in\N$, and that is a contradiction.

    Suppose, then, that $y\in W(\Omega\oplus S)$. To complete the proof, it then suffices to show that in this case there is $v\in W(\Omega\oplus S)$ such that $p_1(v)\le p_1(\Omega\oplus S)$ and $\|v-y\|\le 2\diam(\Omega)$, as this would imply that the sets $S_n$ don't converge weakly to $S$.
    Indeed, for such a vertex $v$ we find a sequence of vertices $v_n\in W(\Omega\oplus S_n)$ for sufficiently large $n\in\N$ such that $p_1(v_n)=p_1(v)$ and $v_n\to v$, contradicting the choice of $y$.

    Write $(t_k)_{k\in \mathbb{N}}$ for the increasing enumeration (with respect to the second coordinate) of those $t\in S$ such that $\Omega+t$ intersects $p_1(\Omega\oplus S)\times (0,\infty)$ and set $U_k=\{t_{k'}\}_{k'\ge k}$ for every $k\in \mathbb{N}$.
    As $\Omega$ is an axis-parallel polygonal set, by the pigeonhole principle (see, e.g., \cite{Newman}) there exist $k_0< k_1\in \N$ such that 
    $t_{k_0+j}+\alpha=t_{k_1+j}$ for every $j\in \mathbb{N}$, where $\alpha=t_{k_1}-t_{k_0}$; let $k_0$ be the minimal index for which there exists $k_1$ satisfying this property.
    We claim that $k_0=0$.

    Suppose for a contradiction that $k_0>0$.
    By periodicity, we have that 
    \begin{equation}\label{eq:peralpha}
        \Omega\oplus (U_{k_0}+\alpha)=\Omega\oplus U_{k_1}.
    \end{equation}
    For $i=0,1$, let $\kappa_i$ be a maximal non-trivial horizontal line segment that is contained in $\partial(\Omega\oplus U_{k_i})$ with the property that for every other non-trivial horizontal line segment $\kappa'\subseteq  \partial(\Omega\oplus U_{k_i})$ we have that $p_2(\kappa')\le p_2(\kappa_i)$.
    Observe that the existence of such $\kappa_1$ follows from the assumption that there is $a\ge 0$ such that $$p_1(\Omega\oplus S)\times (a,\infty)\subseteq  \overline{\Omega}\oplus U_{k_1}$$ and $$p_1(\Omega\oplus S)\times (-\infty,-a)\cap(\overline{\Omega}\oplus U_{k_1})=\emptyset.$$
    By periodicity, we have that $p_2(\kappa_0)=p_2(\kappa_1-\alpha)$, $\kappa_1-\alpha\subseteq  \partial(\Omega\oplus U_{k_0})$ and 
    \begin{equation}\label{eq:kappaminusalphaslice}
        p_1(\Omega\oplus S)\times (p_2(\kappa_1-\alpha),\infty)\subseteq  \overline{\Omega}\oplus U_{k_0}.
    \end{equation}
    Now, as $p_2(\alpha)>0$, there is $k_0\le l<k_1$ such that $\overline{\Omega}+t_{l}$ intersects $\kappa_1$ in a non-trivial line segment.
    By construction, we have $p_2(t_k)\le p_2(t_{k+1})$ for every $k\in \mathbb{N}$ and $$p_1(\kappa_1)\times [p_2(\kappa_1),\infty)\subseteq  \overline{\Omega}\oplus U_{k_1},$$ so it must be the case that $l=k_1-1$ as otherwise $\Omega\oplus (U_{k_1}\cup\{t_{k_1-1}\})$ is not a packing.
    Moreover,  observe that since $$(\Omega+ t_{k_1-1})\cap (\Omega\oplus U_{k_1})=\emptyset,$$  by \eqref{eq:peralpha}, we have that $\Omega+ (t_{k_1-1}-\alpha)$ does not intersect $\Omega\oplus U_{k_0}$.
    Thus, by \eqref{eq:kappaminusalphaslice}, we obtain that $p_2(t_{k_0-1})\le p_2(t_{k_1-1}-\alpha)$.
    Finally, observe that for every $s\in S\setminus U_{k_0}$
    $$\max_{z\in\overline{\Omega}} p_2(z+s)\le \max_{z\in\overline{\Omega}} p_2(z+t_{k_0-1}),$$
    which implies that $\Omega+ s$ does not intersect the rectangle
    $$p_1(\kappa_1)\times (p_2(\kappa_1-t_{k_1-1}+t_{k_0-1}),p_2(\kappa_0)).$$
    This can only happen if $p_2(\kappa_0)\leq 0$ or $t_{k_1-1}-t_{k_0-1}=\alpha$.
    Note that by the minimality of $k_0$ the latter is impossible, we then must have $p_2(\kappa_0)\leq 0$, which, by the maximality of $p_2(\kappa_0)$, implies that
    $$p_1(\Omega\oplus S)\times (0,\infty)\subseteq \Omega\oplus U_{k_0}.$$
    And so, by the construction of $\{t_k\}_{k\in \mathbb{N}}$, we must have that $U_{k_0}=U_0=\{t_k\}_{k\in \mathbb{N}}$ and thus $k_0=0$.

    Let $v\in W(\Omega\oplus U_0)$ be such that $p_1(v)\le p_1(\Omega\oplus S)$ and
    $$p_2(v)=\max\{p_2(v'):v'\in W(\Omega\oplus U_0) \text{ and } p_1(v')\le p_1(\Omega\oplus S)\}.$$
    Such $v$ is well defined as $p_1(\Omega\oplus S)\times (0,\infty)\subseteq  \overline{\Omega}\oplus U_0$ and
    $$p_1(\Omega\oplus S)\times (-\infty,-\diam(\Omega))\cap(\overline{\Omega}\oplus U_0)=\emptyset.$$
    Observe that $\|v-y\|\le 2\diam(\Omega)$.
    So it remains to show that $v\in W(\Omega\oplus S)$.
    Note that this is clearly true when $U_0=S$, so we may assume that $S\setminus U_0\not=\emptyset$.
    
    From periodicity, we have that $p_1(\Omega\oplus S)\times (p_2(\alpha),\infty)\subseteq  \overline{\Omega}\oplus U_{k_1}$, and $v+\alpha\in W(\Omega\oplus U_{k_1})$.
    In particular, as $y+\alpha\not\in W(\Omega\oplus U_0)$ and $v+\alpha\not\in W(\Omega\oplus U_0)$, we must have that $y+\alpha\not\in W(\Omega\oplus (U_{k_1}\cup\{t_{k_1-1}\}))$ and $v+\alpha\not\in W(\Omega\oplus (U_{k_1}\cup\{t_{k_1-1}\}))$.
    Let $s\in S$ be such that $$p_2(s)=\max\{p_2(s'):s'\in S\setminus U_0\}.$$
    Then it follows from periodicity that $p_2(s)\le p_2(t_{k_1-1}-\alpha)$.
    As $y\in W(\Omega\oplus S)$, we have that $p_2(s)<p_2(t_{k_1-1}-\alpha)$.
    Finally, since $\Omega\oplus (U_0\cup\{t_{k_1-1}-\alpha\})$ is a packing and $v\not\in W(\Omega\oplus (U_0\cup\{t_{k_1-1}-\alpha\}))$, we obtain that $v\in W(\Omega\oplus (U_0\cup\{s\}))$.
    By the definition of $s$, we conclude that $v\in W(\Omega\oplus S)$, as needed.
\end{proof}

In what follows, we will assume that if $\Omega$ can tile a column, then the parameter $M$ in \eqref{eq:LargeParameter} additionally satisfies $M>3r_\Omega$, where $r_\Omega$ is as in Proposition \ref{pr:NotLinkedThree}.

\begin{proposition}\label{pr:NoGapsNoColumnButColumn}
    Suppose $\Omega$ can tile a column, and let $1\le i\le n$, $Q\in L_i$ be such that $D_i=\mathbb{R}(0,1)$ and $T_{i,Q}\subseteq  T_{ng}$ is not a column.
    Then for every $m\in \mathbb{Z}$
    $$\Omega\oplus T_{i,Q}+\beta_{Q,m}=\Omega\oplus T_{i,Q+m\widetilde{\gamma}},$$
     where $\beta_{Q,m}=w_0-v_0$ for any $w_0\in W_{i,Q+m\widetilde{\gamma}}$, $v_0\in W_{i,Q}$ such that $(w_0-v_0)^\varphi=m\widetilde{\gamma}$.
\end{proposition}
\begin{proof}
    Note that as $T_{i,Q}$ is not a column, we may define $\beta_{Q,m}=w_0-v_0$ as in the statement since $W_{i,Q}\not=\emptyset$.
    By the periodicity of $T^\varphi_{i,Q}$, Theorem~\ref{thm:GoodFiniteLocalComplexity}~(7) and the assumption that $T_{i,Q}$ is not a column, we have that all maximal connected line segments of $\partial(\Omega\oplus T_{i,Q})$ are bounded. Thus,  by Claim~\ref{cl:NecessaryForRegions}, it suffices to show that $u\in W_{i,Q}$ if and only if $u+\beta_{Q,m}\in W_{i,Q+m\widetilde{\gamma}}$, and 
    \begin{equation}\label{eq:GoodPointsThree}
        \mathcal{B}(u,(\Omega\oplus T_{i,Q}))=\mathcal{B}(u+\beta_{Q,m},(\Omega\oplus T_{i,Q+m\widetilde{\gamma}}))
    \end{equation}
    for every $u\in W_{i,Q}$.

    Let $X\subseteq  W_{i,Q}$ be the set of those $u\in W_{i,Q}$ such that $u+\beta_{Q,m}\in W_{i,Q+m\widetilde{\gamma}}$ and that satisfy \eqref{eq:GoodPointsThree}.
    Then $X$ is non-empty as $v_0\in W_{i,Q}$.
    Moreover, by Claim~\ref{cl:LinkedEquivalence}, we have that if $u\in X$, then $[u]_{M/3}\subseteq  X$.

     Observe that for every $a,b,c\in W_{i,Q}$ such that $p_2(a)\le p_2(b)\le p_2(c)$ we have
    \begin{equation}\label{eq:prelim}
         a\in [c]_{M/3} \Rightarrow b\in [c]_{M/3}.
    \end{equation}
    Indeed, let $a=z_0,\dots,z_m=c\subseteq  W_{i,Q}$ be a sequence that witnesses that $a\in [c]_{M/3}$.
    By our assumption that $p_2(a)\le p_2(b)\le p_2(c)$, there must be an index $0\le j\le m-1$ such that $p_2(z_j)\le p_2(b)\le p_2(z_{j+1})$.
    We thus have
    $$\min\{\|z_j-b\|,\|z_{j+1}-b\|\}\le \diam(\Omega)+M/6< M/3$$
    by \eqref{eq:LargeParameter}, which shows that $b\in [c]_{M/3}$.

    Now, suppose for a contradiction that there is $v\in W_{i,Q}\setminus X$.
    By Proposition~\ref{pr:NotLinkedThree}, we may assume that $p_1(v)\le p_1(\Omega\oplus T_{i,Q})$.
    Additionally, we may without loss of generality assume that $p_2(v)>p_2(v_0)$, and that $p_2(v)$ is minimal with these properties.
    Consider now the maximal connected line segment $I$ that is a subset of
    $$\partial(\Omega\oplus T_{i,Q})\cap \{y:p_2(y)\le p_2(v) \text{ and } p_1(y)=p_1(v)\}$$
    and that contains $v$.
    As $I$ is bounded, it must be of the form
    $$P_{a,v}=\{y:p_2(a)\le p_2(y)\le p_2(v) \text{ and } p_1(y)=p_1(v)=p_1(a)\}.$$
    Suppose first that $a=v$, then as $T_{i,Q}$ does not contain gaps, we find $v'\in [v]_{M/3}$, where $v'\in W_{i,Q}$ is such that $p_1(v')=p_1(v)$ and $p_2(v')<p_2(v)$.
    This implies that either $p_2(v')\le p_2(v_0)\le p_2(v)$, or $p_2(v_0)\le p_2(v')\le p_2(v)$.
    The former case implies that $v_0\in [v]_{M/3}$ by \eqref{eq:prelim}, and the latter contradicts the minimality of $p_2(v)$. Therefore, we must have that $a\not=v$, which implies that $P_{a,v}\subseteq  \partial(\Omega\oplus T_{i,Q})$ and $a\in W_{i,Q}$.

    We claim that $a\in X$.
    Indeed, if $a\not\in X$, then by the minimality of $p_2(v)$, we deduce that $p_2(a)\le p_2(v_0)$.
    Proposition~\ref{pr:NotLinkedThree} gives the existence of $z\in [v_0]_{M/3}$ such that $p_1(z)\le p_1(\Omega\oplus T_{i,Q})$.
    As $W(\Omega\oplus T_{i,Q})\cap P_{a,v}=\{a,v\}$, we have that either $p_2(z)\le p_2(a)\le p_2(v_0)$ or $p_2(v_0)\le p_2(v)\le p_2(z)$. By \eqref{eq:prelim}, the former case implies that $a\in X$ and the latter implies that $v\in X$, both contradicting our assumptions. 

    Observe that the following claim gives the desired contradiction.

    \begin{claim*}
        Let $a,v\in W_{i,Q}$ be such that $a\in X$, $p_1(a)=p_1(v)\le p_1(\Omega\oplus T_{i,Q})$ and $P_{a,v}\subseteq  \partial(\Omega\oplus T_{i,Q})$. Then $v\in X$.     
    \end{claim*}

    Indeed, this claim implies that $W_{i,Q}+\beta_{Q,m}\subseteq  W_{i,Q+m\widetilde{\gamma}}$ and \eqref{eq:GoodPointsThree} holds for every $u\in W_{i,Q}$.
    Repeating the same argument with swapping the roles of $T_{i,Q}$ and $T_{i,Q+m\widetilde{\gamma}}$ implies that $u\in W_{i,Q}$ if and only if $u+\beta_{Q,m}\in W_{i,Q+m\widetilde{\gamma}}$.
    In particular $\beta_{Q,m}$ is well-defined.
    As every line segment in $\partial(\Omega\oplus T_{i,Q})$ and $\partial(\Omega\oplus T_{i,Q+m\widetilde{\gamma}})$ is finite, we conclude that $\Omega\oplus T_{i,Q}+\beta_{Q,m}=\Omega\oplus T_{i,Q+m\widetilde{\gamma}}$ by Claim~\ref{cl:NecessaryForRegions}.

    \bigskip
    \emph{Proof of the Claim:}
        Let $b,w\in W_{i,Q+m\widetilde{\gamma}}$ be such that $a^\varphi+m\widetilde{\gamma}=b^\varphi$ and $v^\varphi+m\widetilde{\gamma}=w^\varphi$.
        Observe that $v^\varphi-a^\varphi=w^\varphi-b^\varphi$, and that $p_1(w)=p_1(b)$ by Theorem~\ref{thm:GoodFiniteLocalComplexity}~(12) as $T_{i,Q+m\widetilde{\gamma}}\subseteq  T_{ng}$.
        By the assumption that $a\in X$ and Theorem~\ref{thm:GoodFiniteLocalComplexity}~(11), we have that $a+\beta_{Q,m}=b$.
        Observe that by Theorem~\ref{thm:GoodFiniteLocalComplexity}~(7), to complete the proof of the claim it suffices to show that $v+\beta_{Q,m}=w$.
        
        Recall that
        $$m_\Omega=m_1(p_2((\{p_1(x)\}\times \mathbb{R})\cap\overline{\Omega})),$$
        where $x\in V(\Omega)$ such that $p_1(x)\le p_1(\Omega)$.
        Observe that we can write
        $$\|v-a\|=km_\Omega+\sum_{\omega\in A} |\omega|,$$
        for some $k\in \mathbb{N}$ and a multiset $A$ of vertical edges on the extreme left boundary of $\Omega$ that satisfies
        $$|A|\le 2\diam(\Omega)/\min\{|\omega|:\omega\in E(\Omega)\}+2.$$
        To see this, note that if $t\in T_{i,Q}$ is such that $((\{p_1(x)\}\times \mathbb{R})\cap\overline{\Omega})+t$ does not contribute by a multiple of $m_\Omega$ then $((\{p_1(x)\}\times \mathbb{R})\cap\overline{\Omega})+t$ intersects either the line segment from $a$ to $a+(0,\diam(\Omega))$ or the line segment from $v-(0,\diam(\Omega))$ to $v$.
        Similarly, we can write
        $$\|w-b\|=k'm_\Omega+\sum_{\omega'\in B} |\omega'|,$$
        where $k'\in \mathbb{N}$ and $|B|\le 2\diam(\Omega)/\min\{|\omega|:\omega\in E(\Omega)\}+2$. 
        
        As $(w)^\varphi-(b)^\varphi=(v)^\varphi-(a)^\varphi$, we have by linearity of $\varphi$ that
        \begin{equation*}
            k\cdot m_{\Omega^\varphi}+\sum_{\omega\in A}|\omega^\varphi|=k'\cdot m_{\Omega^\varphi}+\sum_{\omega'\in B}|(\omega')^\varphi|.
        \end{equation*}
        We may assume, without loss of generality, that $k\ge k'$, and write
        \begin{equation}\label{eq:NumberSummands}
            (k-k')\cdot m_{\Omega^\varphi}+\sum_{\omega\in A}|\omega^\varphi|=\sum_{\omega'\in B}|(\omega')^\varphi|.
        \end{equation}
        Note that
        $$\sum_{\omega'\in B}|(\omega')^\varphi|\le m_{\Omega^\varphi} |B|\le 2m_{\Omega^\varphi}\left(\diam(\Omega)/\min\{|\omega|:\omega\in E(\Omega)\}+1\right).$$
        This implies that
        $$k-k'\le 2\left(\diam(\Omega)/\min\{|\omega|:\omega\in E(\Omega)\}+1\right).$$
        Note that if we write $m_\Omega=\sum_{i=1}^q |\omega_i|$ as a sum of distinct vertical edges of $\Omega$, then $q\le |E(\Omega)|$.
        We conclude that each side of \eqref{eq:NumberSummands} can be expressed as a sum of vertical edge lengths over an index set that is bounded in size by
        $$4|E(\Omega)|\left(\diam(\Omega)/\min\{|\omega|:\omega\in E(\Omega)\}+1\right).$$
        By \eqref{eq:LargeParameter} and Theorem~\ref{thm:GoodFiniteLocalComplexity}~(14), we conclude that
        $$(k-k')\cdot m_{\Omega}+\sum_{\omega\in A}|\omega|=\sum_{\omega'\in B}|\omega'|.$$
        Adding $k'm_\Omega$ to both sides, we get
        $$\|v-a\|=\|w-b\|,$$
        which implies that $v-a=w-b$ as  $p_1(v)=p_1(a)$ and $p_1(w)=p_1(b)$.
        By the assumption that $a+\beta_{Q,m}=b$, we have $v+\beta_{Q,m}=w$, as claimed.
        $\qquad \blacksquare$

        \bigskip
This completes the proof of Proposition \ref{pr:NoGapsNoColumnButColumn}. 
\end{proof}

\subsubsection{\bf Proof of Lemma~\ref{lem:MainShifts}}\label{subsubsec:ProofMainLemma}
    Let us first discuss the assumptions on the parameter $M>0$.    
    Recall that we assume in \eqref{eq:LargeParameter} that 
    \begin{equation}
        M>\max\left\{189\diam(\Omega),\frac{1}{\diam(\Omega)},\sqrt{6},4|E(\Omega)|\left(\frac{\diam(\Omega)}{\min\{|\omega|:\omega\in E(\Omega)\}}+1\right)\right\}.
    \end{equation}
    In case that $\Omega$ cannot tile a column, we assume additionally that $M>3\ell_\Omega$, where $\ell_\Omega$, that only depends on $\Omega$, is defined in Proposition~\ref{pr:NotLinkedOne}.
    If $\Omega$ can tile a column, then we assume additionally that $M>3r_\Omega$, where $r_\Omega$, that only depends on $\Omega$, is defined in Proposition~\ref{pr:NotLinkedThree}.
    As Propositions~\ref{pr:NoGapsNoColumn} and~\ref{pr:NoGapsNoColumnButColumn} are stated only for vertical axis $D_i=\mathbb{R}(0,1)$, we need to consider the analogous parameters $\ell_{\Omega^o}$ and $r_{\Omega^o}$ for $\Omega^o$ which is defined as the rotation of $\Omega$ by $90$ degrees.
    Finally, we require that $M>3\ell_{\Omega^o}$ in case that $\Omega^o$ cannot tile a column, and $M>3r_{\Omega^o}$ in case $\Omega^o$ can tile a column.
    Note that such $M>0$ clearly only depends on $\Omega$ and not on $T$.

    By Propositions~\ref{pr:NoGapsTwo},~\ref{pr:NoGapsNoColumn} and~\ref{pr:NoGapsNoColumnButColumn} (applied to the rotated version $\Omega^o$ in case $D_i=\mathbb{R}(1,0)$), it is enough to argue that there is a universal period $\alpha_{i,Q+\mathbb{Z}\widetilde{\gamma}}$ for $T_{i,Q+m\widetilde{\gamma}}$ that is independent of $m\in \mathbb{Z}$.
    We start by finding a period $\alpha_{i,Q+m\widetilde{\gamma}}\in \mathbb{R}^2\setminus \{0\}$ for each fixed $m\in \mathbb{Z}$.

    Without loss of generality, we may assume that $m=0$.
    Let $(t_j)_{j\in \mathbb{Z}}$ be an enumeration of $T_{i,Q}$ such that $(t^\varphi_j)_{j\in \mathbb{Z}}$ is the canonical enumeration of $T^\varphi_{i,Q}$.

    \begin{claim}\label{cl:Enumeration}
        We have $\|t^\varphi_j-t^\varphi_{j+1}\| <90\diam(\Omega)$ for every $j\in \mathbb{Z}$.
    \end{claim}
    \begin{proof}
        Suppose for a contradiction that there is $j_0\in \mathbb{Z}$ such that $\|t^\varphi_{j_0}-t^\varphi_{j_0+1}\| \ge 90\diam(\Omega)$.
        Define $F_0=\{t_j:j\le j_0\}$ and $F_1=\{t_j:j> j_0\}$.
        By the assumption $T_{i,Q}\subseteq  T_{ng}$ we have that $T_{i,Q}$ does not have gaps.
        It follows that there must be $i_0\le j_0$ and $i_1>j_0$ such that $\|t_{i_1}-t_{i_0}\|\le 60 \diam(\Omega)$.
        By \eqref{eq:LargeParameter} and Theorem~\ref{thm:GoodFiniteLocalComplexity}~(9), we conclude that 
        $$\|t^\varphi_{i_1}-t^\varphi_{i_0}\|\le 60\diam(\Omega)+1/M\le 61\diam(\Omega).$$
        Finally, we obtain that
        $$\|t^\varphi_{j_0+1}-t^\varphi_{j_0}\|\le \|t^\varphi_{i_1}-t^\varphi_{i_0}\|< 90\diam(\Omega),$$
        as claimed.
    \end{proof}

    Write $j_1\in \mathbb{Z}$ for the unique non-zero index such that
    \begin{equation}\label{eq:Periodic}
        t_{\ell}^\varphi+(k_i\widetilde{\lambda}+\ell_i\widetilde{\gamma})=t_{j_1+\ell}^\varphi
    \end{equation}
    for every $\ell\in \mathbb{Z}$.
    Set $\alpha_{i,Q}=t_{j_1}-t_{0}$ and observe that $\alpha_{i,Q}$ is a non-zero vector by Theorem~\ref{thm:GoodFiniteLocalComplexity}~(11) and \eqref{eq:Periodic}.
    Fix $\ell\in \mathbb{Z}$ and observe that by Claim~\ref{cl:Enumeration}, Theorem~\ref{thm:GoodFiniteLocalComplexity}~(13), \eqref{eq:LargeParameter} and \eqref{eq:Periodic}, we have that
    $$t_{j_1+\ell}-t_{j_1+(\ell-1)}=t_{\ell}-t_{\ell-1}.$$
    By induction, we obtain that $t_{\ell}+\alpha_{i,Q}=t_{j_1+\ell}$ holds for every $\ell\in \mathbb{Z}$.

    Fix $m\in \mathbb{Z}$.
    The same argument shows the existence of a non-zero vector $\alpha_{i,Q+m\widetilde{\gamma}}$ such that $T_{i,Q+m\widetilde{\gamma}}+\alpha_{i,Q+m\widetilde{\gamma}}=T_{i,Q+m\widetilde{\gamma}}$.
    Write $(s_j)_{j\in \mathbb{Z}}$ for the enumeration of $T_{i,Q+m\widetilde{\gamma}}$ such that $(s^\varphi_j)_{j\in \mathbb{Z}}$ is the canonical enumeration of $T^\varphi_{i,Q+m\widetilde{\gamma}}$, and $j_2\in \mathbb{Z}$ for the non-zero index as in\eqref{eq:Periodic}.
    It remains to show that $\alpha_{i,Q}=\alpha_{i,Q+m\widetilde{\gamma}}$.

    Without loss of generality, we may assume that there are $v,w\in V(\Omega)$ such that $v+t_0\in W_{i,Q}$ and $w+s_0\in W_{i,Q+m\widetilde{\gamma}}$ are such that $\beta_{Q,m}=(w+s_0)-(v+t_0)$.
    By Proposition~\ref{pr:NoGapsTwo},~\ref{pr:NoGapsNoColumn} and~\ref{pr:NoGapsNoColumnButColumn} and Theorem~\ref{thm:GoodFiniteLocalComplexity}~(7), there is $u\in W_{i,Q+m\widetilde{\gamma}}$ such that
    $$u=v+t_0+\alpha_{i,Q}+\beta_{Q,m}.$$
    We have
    \begin{equation*}
        \begin{split}
            u^\varphi= & \ (v+t_0)^\varphi+(k_i\widetilde{\lambda}+\ell_i\widetilde{\gamma})+(w+s_0)^\varphi-(v+t_0)^\varphi\\
            = & \ (k_i\widetilde{\lambda}+\ell_i\widetilde{\gamma})+(w+s_0)^\varphi\\
            = & \ (w+s_{j_2})^\varphi.
        \end{split}
    \end{equation*}
    By Theorem~\ref{thm:GoodFiniteLocalComplexity}~(11), we have that
    $$w+s_0+\alpha_{i,Q+m\widetilde{\gamma}}=w+s_{j_2}=u=v+t_0+\alpha_{i,Q}+\beta_{Q,m}=w+s_0+\alpha_{i,Q}$$
    and the proof is completed by setting $\alpha_{i,Q+\mathbb{Z}\wgamma}=\alpha_{i,Q}$.

\subsection{Proof of Theorem~\ref{thm:MainLocallyFin}}\label{subsec:ProofOfThmLocallyFin}
Let $X\not=\emptyset$ be a regular open periodic subset of $\R$, $\Omega$ be an axis-parallel tile, $E=X\times \R$, and suppose $T\in \Tile(\Omega; E)$ has finite local complexity.
The proof is split into three parts: In the first {\bf (I)}, we focus on the case when $T=T_c$; in the second {\bf (II)}, we focus on the case when the earthquake decomposition of $T$ is trivial; and finally, in the third part {\bf (III)}, we conclude the proof from parts {\bf (I)} and {\bf (II)}.

{\bf (I)}.
Suppose that, after possibly swapping the vertical and horizontal axes, $T=T_c$.
By Proposition~\ref{pr:WLOGcolumns}, we have that $p_1(T)$ is $\langle \beta\rangle$-periodic for some $\beta>0$, and $T$ is $(0,\alpha)$-periodic for some $\alpha>0$.
This proves Theorem~\ref{thm:MainLocallyFin}~(i) for the case $T=T_c$.
Moreover, if we set $m_\Omega=m_1(p_2(\overline{\Omega}\cap (\{p_1(x)\}\times\mathbb{R})))$, where $x\in V(\Omega)$ is such that $p_1(x)\le p_1(\Omega)$, then $\alpha\in \mathbb{Q}m_\Omega$.

We prove a strengthening of Theorem~\ref{thm:MainLocallyFin}~(ii) for the case $T=T_c$ that will be useful later.
Let $z>0$.
Observe that $p_1(T)$ is a union of finitely many cosets of $\langle \beta\rangle$, and let $\{x_j\}_{j=1}^\ell$ be an enumeration of their representatives.
For each $1\le j\le \ell$, set $T_j=T_c \cap(\{x_j\}\times \mathbb{R})$ and define
$$S^{z,j}_c=\bigsqcup_{m\in \mathbb{Z}} (T_j+m(\beta,z)).$$
Then we have that $S^z_c=\bigcup_{j=1}^\ell S^{z,j}_c$ is $\langle(0,\alpha),(\beta,z)\rangle$-periodic and $\Omega\oplus T=\Omega\oplus S^z_c$.

{\bf (II)}.
Suppose that the vertical earthquake decomposition of $T$ is trivial and, if $X=\mathbb R$, that its horizontal earthquake decomposition is also trivial.
We distinguish two subcases based on whether $T_{ng}$ is empty or not.

{\bf (IIa)}.
Suppose that $T_{ng}$ is nonempty, i.e., $T\neq T_g$.
Observe that by Lemma \ref{lem:TechnicalClaim} and Theorem~\ref{thm:GoodFiniteLocalComplexity}~(8), we have that $W_{i,Q}$ is nonempty for every $1\le i\le n$ and $Q\in L_i$.
Take any $1\le i_0\le n$, $Q_0\in L_{i_0}$ such that $T_{i_0,Q_0}\subseteq T_{ng}$, and use Lemma~\ref{lem:MainShifts} to define $\gamma=\beta_{Q_0,N}$.
That is, $\gamma=w_0-v_0$, where $v_0\in W_{i_0,Q_0}$, $w_0\in W_{i_0,Q_0+N\wgamma}$ and $(w_0-v_0)^\varphi=N\wgamma$.
Observe that by Lemma~\ref{lem:MainShifts}, we have
$$\gamma+\left(\Omega\oplus T_{i_0,Q_0}\right)=\Omega\oplus T_{i_0,Q_0+N\wgamma},$$
and $\gamma\not=0$ by Theorem~\ref{thm:GoodFiniteLocalComplexity}~(11) as $(v_0)^\varphi\not=(w_0)^\varphi$.

\begin{claim}\label{cl:OneClaim}
    We have that
    \begin{enumerate}
    \item $T_g$ is $\langle\gamma\rangle$-periodic,
    \item $\beta_{N,Q}=\gamma$ for every $1\le i\le n$ and $Q\in L_i$ such that $T_{i,Q}\subseteq T_{ng}$.
    In particular, $\gamma+(\Omega\oplus T_{i,Q})=\Omega\oplus T_{i,Q+N\wgamma}$.
\end{enumerate}
\end{claim}
\begin{proof}
Let $S_0$ be the union of all elements $t\in T_g$ such $\gamma+t\in T_g$, and $t\in T_{i,Q}$ that satisfies (2).
Observe that $S_0\not=\emptyset$ as $T_{i_0,Q_0}\subseteq  S_0$, and that, by Lemma~\ref{lem:MainShifts}, there is $S'_0\subseteq  T$ such that $\Omega\oplus S'_0=\gamma+\Omega\oplus S_0$.
We show that $T=S_0$.

By the assumptions in (II), combined with Lemma~\ref{lem:TechnicalClaim} and the fact that $S_0$ is nonempty, it is enough to show that $W(\Omega\oplus S_0)$ is empty.
Suppose for a contradiction that $w\in W(\Omega\oplus S_0)$.
Then
$$\gamma+w\in W(\gamma+\Omega\oplus S_0)\cap (V(\Omega)+T)=W(\Omega\oplus S'_0)\cap (V(\Omega)+T),$$
and $(\gamma+w)^\varphi-w^\varphi=N\wgamma$.
As $\Omega$ is an axis-parallel polygonal set,  either there are $t\in T_g\setminus S_0$ and $v\in V(\Omega)$ such that $w=v+t$, or there is $Q\in L_i$ for some $1\le i\le n$ such that $w\in W_{i,Q}$ and $T_{i,Q}\subseteq  T\setminus S_0$.

Consider first the former case.
By Lemma~\ref{lem:GapsPeriodicity}, there is $t'\in T_g$ such that $(v+t')^\varphi-(v+t)^\varphi=N\wgamma$.
This implies that
$$(v+t')^\varphi=N\wgamma+(v+t)^\varphi=N\wgamma+w^\varphi=(\gamma+w)^\varphi=(\gamma+v+t)^\varphi.$$
By Theorem~\ref{thm:GoodFiniteLocalComplexity}~(11), we conclude that $v+t'=\gamma+v+t$ as $\gamma+w\in V(\Omega)+T$.
This shows that $\gamma+t=t'\in T_g$, contradicting the choice of $t\in T\setminus S_0$.

For the latter case, note that by the definition we have $\beta_{N,Q}+w\in W_{i,Q+N\wgamma}$ and $(\beta_{N,Q}+w)^\varphi=N\wgamma+w^\varphi$.
This implies that $(\beta_{N,Q}+w)^\varphi=(\gamma+w)^\varphi$, and consequently $\beta_{N,Q}+w=\gamma+w$ by Theorem~\ref{thm:GoodFiniteLocalComplexity}~(11) as $\gamma+w\in V(\Omega)+T$.
This contradicts the choice of $i$ and $Q$, hence $W(\Omega\oplus S_0)$ is empty and so $S_0=T$ as desired.

The fact that $T=S_0$ implies (2) and $\gamma+T_g\subseteq T_g$.
The exact same argument for $-\gamma=\beta_{Q+N\wgamma,-N}=v_0-w_0$ shows that $T_g$ is $\langle \gamma\rangle$-periodic.
\end{proof}

Recall that in the proof of Lemma~\ref{lem:MainShifts}, we define $\alpha_{i_0,Q_0+\mathbb{Z}\wgamma}$ such that $(\alpha_{i_0,Q_0+\mathbb{Z}\wgamma})^\varphi=k_{i_0}\wlambda+\ell_{i_0}\wgamma$.
This motivates the following definition
$$\lambda=\frac{N}{k_{i_0}}\left(N\alpha_{i_0,Q_0+\mathbb{Z}\wgamma}-\ell_{i_0}\gamma\right).$$

\begin{claim}\label{cl:TwoClaim}
    We have
    \begin{enumerate}
        \item $\lambda+T_g=T_g$,
        \item $\lambda=\frac{N}{k_i}\left(N\alpha_{i,Q+\mathbb{Z}\wgamma}-\ell_i \gamma\right)$ for every $1\le i\le n$ and $Q\in L_i$ such that $T_{i,Q}\subseteq T_{ng}$.
    \end{enumerate}
\end{claim}
\begin{proof}
Define $S_1$ to be the union of those $t\in T_g$ such that $\lambda+t\in T_g$, and $t\in T_{i,Q}$ for which (2) holds.
Then $S_1\not=\emptyset$ as $T_{i_0,Q_0}\subseteq  S_1$.
Observe that if $T_{i,Q}\subseteq T_{ng}\cap S_1$, then
$$T_{i,Q+\mathbb{Z}\wgamma}=\bigsqcup_{m\in \mathbb{Z}} T_{i,Q+m\wgamma}\subseteq  S_1$$
by Lemma~\ref{lem:MainShifts}.
In particular, as $\Omega\oplus T_{i,Q+\mathbb{Z}\wgamma}$ is $\langle \alpha_{i,Q+\mathbb{Z}\wgamma},\gamma\rangle$-periodic by Claim~\ref{cl:OneClaim} and Lemma~\ref{lem:MainShifts}, we have that $\Omega\oplus T_{i,Q+\mathbb{Z}\wgamma}$ is $\langle \lambda\rangle$-periodic, whenever $T_{i,Q}\subseteq  S_1$.
Together with the definition of $S_1\cap T_g$, we conclude that
\begin{equation}\label{eq:lambda}
    \lambda+\Omega\oplus S_1=\Omega\oplus S'_1
\end{equation}
for some $S'_1\subseteq  T$.
We show that $T=S_1$.
By the assumptions in (II), by Lemma~\ref{lem:TechnicalClaim} and since $S_1\not=\emptyset$, it is enough to show that $W(\Omega\oplus S_1)=\emptyset$.
Suppose for a contradiction that $w\in W(\Omega\oplus S_1)$.
Then we have
$$\lambda+w\in W(\lambda+\Omega\oplus S_1)\cap (V(\Omega)+T)=W(\Omega\oplus S'_1)\cap (V(\Omega)+T).$$
Observe that by the linearity of $\varphi$
\begin{equation}\label{eq:Somesing}
    ((\lambda+w)-w)^\varphi=N^2\wlambda.
\end{equation}
As $\Omega$ is an axis-parallel polygonal set, there is either $t\in T_g\setminus S_1$ and $v\in V(\Omega)$ such that $w=v+t$, or $1\le i\le n$ and $Q\in L_i$ such that $w\in W_{i,Q}$ and $T_{i,Q}\subseteq  T\setminus S_1$.

In the former case, there is $t'\in T_g$ such that $((v+t')-(v+t))^\varphi=N^2\wlambda$ by Lemma~\ref{lem:GapsPeriodicity}.
This implies that
$$(v+t')^\varphi=N^2\wlambda+(v+t)^\varphi=N^2\wlambda+w^\varphi=(\lambda+w)^\varphi=(\lambda+v+t)^\varphi.$$
By Theorem~\ref{thm:GoodFiniteLocalComplexity}~(11), we conclude that $v+t'=\lambda+v+t$ as $\lambda+w\in V(\Omega)+T$.
This shows that $\lambda+t=t'\in T_g$, contradicting the choice of $t\in T\setminus S_1$.

For the latter case, note that
$$\frac{N}{k_i}(N\alpha_{i,Q+\mathbb{Z}\wgamma}-\ell_i \gamma)+\Omega\oplus T_{i,Q}=\Omega\oplus T_{i,Q-(N^2\ell_i/k_i)\wgamma},$$
which implies that $\frac{N}{k_i}(N\alpha_{i,Q+\mathbb{Z}\wgamma}-\ell_i \gamma)+w\in W_{i,Q-(N^2\ell_i/k_i)\wgamma}$.
We write
\begin{equation*}
    \begin{split}
        \left(\left(\frac{N}{k_i}(N\alpha_{i,Q+\mathbb{Z}\wgamma}-\ell_i \gamma)+w\right)-w\right)^\varphi= & \ \frac{N}{k_i}(N(k_i\wlambda+\ell_i\wgamma)-N\ell_i\wgamma)\\
        = & \ N^2\wlambda \\
        = & \ \left((\lambda+w)-w\right)^\varphi
    \end{split}
\end{equation*}
by \eqref{eq:Somesing}.
By Theorem~\ref{thm:GoodFiniteLocalComplexity}~(11), we conclude that $\lambda=\frac{N}{k_i}(N\alpha_{i,Q+\mathbb{Z}\wgamma}-\ell_i \gamma)$ as desired.

This shows (2) and that $\lambda+T_g\subseteq T_g$.
The exact same argument for $-\lambda$ shows that $T_g$ is $\langle \lambda\rangle$-periodic.
\end{proof}

Now we are ready to finish the proof of Theorem~\ref{thm:MainLocallyFin}.
We start by showing that $\langle \lambda,\gamma\rangle$ is a lattice.
Note that, by the definition, we have $\gamma^\varphi=N\wgamma$, and $\lambda^\varphi=N^2\wlambda$.
It follows that $\lambda,\gamma\not= (0,0)$, and $\lambda\not\in \mathbb{Q}\gamma$ as $\langle \widetilde{\lambda}, \widetilde{\gamma}\rangle$ is a lattice and $\varphi$ is a $\mathbb{Q}$-homomorphism.
Let $v_0\in W_{i_0,Q_0}$ be the vertex from the definition of $\gamma$ before Claim~\ref{cl:OneClaim}.
Then it follows from Claim~\ref{cl:OneClaim}, and Lemma~\ref{lem:MainShifts} and Claim~\ref{cl:TwoClaim} that
\begin{equation}\label{eq:PeriodicPoint}
    \langle\lambda,\gamma\rangle+v_0\subseteq  V(\Omega)+T.
\end{equation}
As a consequence, we get that $\lambda\not\in (\mathbb{R}\setminus \mathbb{Q})\gamma$ as this would imply by ergodicity that there are sequences of integers $k_n,\ell_n$ such that
$$0<\|(k_n\lambda+v_0)-(\ell_n\gamma+v_0)\| \to 0,$$
which contradicts the uniform discreteness of $V(\Omega)+T$, which follows from the finite local complexity of $T$.
Hence $\langle\lambda,\gamma\rangle$ is a lattice.

For $1\le i\le n$ and $Q\in L_i$, define
$$T_{i,Q+\mathbb{Z}\widetilde{\gamma}}=\bigsqcup_{m\in \mathbb{Z}} T_{i,Q+m\widetilde{\gamma}}.$$
By Claim~\ref{cl:OneClaim}, and Lemma~\ref{lem:MainShifts} and Claim~\ref{cl:TwoClaim}, each of the sets $\Omega\oplus T_{i,Q+\mathbb{Z}\widetilde{\gamma}}$ is $\langle \lambda,\gamma\rangle$-periodic, and for every $Q,Q'\in L_i$ the sets $\Omega\oplus T_{i,Q+\mathbb{Z}\widetilde{\gamma}}$ and $\Omega\oplus T_{i,Q'+\mathbb{Z}\widetilde{\gamma}}$ are disjoint (up to null sets) unless $Q+\mathbb{Z}\widetilde{\gamma}= Q'+\mathbb{Z}\widetilde{\gamma}$.
Thus, for each $1\le i\le n$ there are only finitely many sets of the form $T_{i,Q+\mathbb{Z}\widetilde{\gamma}}$.
By Claims~\ref{cl:OneClaim} and~\ref{cl:TwoClaim}, $T_g$ is $\langle \lambda,\gamma\rangle$-periodic.
Combined with Lemma~\ref{lem:MainShifts}, this shows Theorem~\ref{thm:MainLocallyFin}~(i).

Finally, we modify $T$ to a $\langle\lambda,\gamma\rangle$-periodic tiling $S$.
For every $1\le i\le n$ and each orbit $Q+\mathbb Z\widetilde\gamma$ such that $T_{i,Q+\mathbb Z\widetilde\gamma}\subseteq T_{ng}$,
fix a representative $Q$.
Define $S_{i,Q+m\widetilde{\gamma}}=T_{i,Q}+\beta_{Q,m}$, where $\beta_{Q,m}$ is as in Lemma~\ref{lem:MainShifts}, and $S_{i,Q+\mathbb{Z}\widetilde{\gamma}}=\bigsqcup_{m\in \mathbb{Z}} S_{i,Q+m\widetilde{\gamma}}$.
Set $S$ to be the union of $T_g$ and $S_{i,Q+\mathbb{Z}\widetilde{\gamma}}$, where $T_{i,Q+\mathbb{Z}\widetilde{\gamma}}\subseteq  T_{ng}$.
It follows from Lemma~\ref{lem:MainShifts} that $S$ is a tiling of $E$ by $\Omega$, and from Claim~\ref{cl:OneClaim}, Lemma~\ref{lem:MainShifts} and Claim~\ref{cl:TwoClaim} that $S$ is $\langle\lambda,\gamma\rangle$-periodic.

{\bf (IIb)}.
Suppose that $T_{ng}=\emptyset$, that is, $T=T_g$.
Then $T^\varphi$ is $\langle N\wlambda,N\wgamma\rangle$-periodic by Lemma~\ref{lem:GapsPeriodicity}.
Let $t_0\in T$.
By the definition, there are $s_0,s_1\in T$ such that $(s_0-t_0)^\varphi=N\wlambda$ and $(s_1-t_0)^\varphi=N\wgamma$.
As the elements of $\{s_0,s_1,t_0\}$ are pairwise distinct by Theorem~\ref{thm:GoodFiniteLocalComplexity}~(11), we conclude that $\{\lambda,\gamma,(0,0)\}$ are pairwise distinct, where $\lambda=s_0-t_0$ and $\gamma=s_1-t_0$.
We claim $\langle\lambda,\gamma\rangle$ is a lattice and $T$ is $\langle\lambda,\gamma\rangle$-periodic, which will prove Theorem~\ref{thm:MainLocallyFin}.

Define
$$S=\{t\in T:\lambda+t,\gamma+t\in T\},$$
and observe that $S$ is non-empty.
We show that $S=T$.
By the assumptions in (II), it is enough to show that $W(\Omega\oplus S)=\emptyset$.
The argument to see that is the same as the corresponding argument in {\bf (IIa)}.
Similarly, if we define
$$S'=\{t\in T:-\lambda+t,-\gamma+t\in T\},$$
then we have $S'=T$.
This shows that $T$ is $\langle\lambda,\gamma\rangle$-periodic.

The fact that $\langle\lambda,\gamma\rangle$ is a lattice follows as in {\bf (IIa)} by noting that
\begin{equation}\label{eq:PeriodicPointTwo}
    \langle\lambda,\gamma\rangle+v_0\subseteq  V(\Omega)+T
\end{equation}
holds for any $v_0\in V(\Omega)+T$.

{\bf (III)}.
For the general case, we may without loss of generality assume that $T$ admits a non-trivial (vertical) earthquake decomposition
$$T=T_c\sqcup \left(\bigsqcup_{P\in \quake{T}{\Omega}} P\right).$$
By Lemma~\ref{lem:CharInfComplex} and the assumption that $T$ has finite local complexity, we conclude that $T$ has finite horizontal complexity.
By Theorem~\ref{thm:PeriodicityMain} and Remark~\ref{rem:PeriodicityMain}, we find $\beta>0$ such that $\Omega\oplus T_c$ and $\Omega\oplus P$ for every $P\in\quake{T}{\Omega}$ are $(\beta,0)$-periodic.
As the set $\quake{T}{\Omega}$ is necessarily finite, we immediately get by {\bf (I)} and {\bf (II)} that Theorem~\ref{thm:MainLocallyFin}~(i) holds.

Observe that since $T$ has finite local complexity, $T_c$ and each $P\in\quake{T}{\Omega}$ has finite local complexity as well.
Hence, we may apply {\bf (I)} to $T_c$ and {\bf (II)} to each $P\in\quake{T}{\Omega}$.

In the former case, we get for each $z>0$ a $\langle(0,\alpha_c),(\beta_c,z)\rangle$-periodic set $S^z_c$ such that $\Omega\oplus T_c=\Omega\oplus S^z_c$.
Moreover, we have that $\alpha_c\in \mathbb{Q}m_\Omega$, as defined in {\bf (I)}, and it is easy to see that $\beta_c\in \mathbb{Q}\beta$.

In the latter case, we find for each $P\in \quake{T}{\Omega}$ a lattice $\langle\lambda_P,\gamma_P\rangle$, a $\langle\lambda_P,\gamma_P\rangle$-periodic set $S_P$ such that $\Omega\oplus P=\Omega\oplus S_P$, and $v_P\in V(\Omega)+P$ such that $\langle\lambda_P,\gamma_P\rangle+v_P\subseteq  V(\Omega)+P$.
The last claim follows from \eqref{eq:PeriodicPoint} and \eqref{eq:PeriodicPointTwo}.

\begin{claim}\label{cl:Lattice}
    We have that $\bigcap_{P\in \quake{T}{\Omega}} \langle\lambda_P,\gamma_P\rangle$ is a lattice.
\end{claim}
\begin{proof}
    Note that
    $$X=\bigcup_{P\in \quake{T}{\Omega}} \langle\lambda_P,\gamma_P\rangle+v_P\subseteq  V(\Omega)+T.$$
    As $T$ has finite local complexity and $V(\Omega)$ is finite, it follows that $X$ has finite local complexity.
    
    Let $\mathbb{P}\subseteq  \quake{T}{\Omega}$ be a maximal subset such that $\langle\lambda,\gamma\rangle=\bigcap_{P\in \mathbb{P}} \langle\lambda_P,\gamma_P\rangle$ is a lattice.
    Suppose for a contradiction that there is $P\in \quake{T}{\Omega}\setminus \mathbb{P}$.
    As $X$ has finite local complexity, we must have that $p(X)$, the projection of $X$ to the torus $\mathbb{R}^2/\langle\lambda,\gamma\rangle$, is finite.
    This is because $\langle\lambda,\gamma\rangle+v_{P'}\subseteq  V(\Omega)+P'$ for any $P'\in \mathbb{P}$.
    In particular, the projection $p(\langle\lambda_P,\gamma_P\rangle+v_P)$ is finite, which implies by the Weyl equidistribution theorem that $\langle\lambda_P,\gamma_P\rangle\cap \langle\lambda,\gamma\rangle$ is a lattice and that is contradiction.
\end{proof}

By Claim~\ref{cl:Lattice}, we may assume that there is a lattice $\langle\lambda,\gamma\rangle$ such that $\langle\lambda,\gamma\rangle=\langle\lambda_P,\gamma_P\rangle$ for every $P\in \quake{T}{\Omega}$.
This proves Theorem~\ref{thm:MainLocallyFin}~(ii) in the case $T_c=\emptyset$.

Finally, suppose that $T_c\not=\emptyset$.
As each of the sets $\Omega\oplus P$ is of the form $X_P\times \mathbb{R}$, where $X_P\not=\mathbb{R}$ is an open $\langle\beta\rangle$-periodic set, we may additionally assume that $\lambda=(0,\alpha)$, and $\gamma=(\beta',z')$, where $\beta'\in \mathbb{Q}\beta\setminus \{0\}$ and $z'>0$.
Moreover, as $\Omega$ can tile a column, it must be the case that $\alpha\in \mathbb{Q}m_\Omega$.
This is because $m_\Omega=m_1(p_2(\overline{\Omega}\cap(\{x\}\times \mathbb{R})))$ for every $x\in p_1(\Omega)\setminus p_1(V(\Omega))$.
Choose $z=(\beta_c/\beta')z'$ and take $S_c^z$ as in (I).
Since $\beta_c/\beta'\in\mathbb Q$ and $\alpha_c/\alpha\in\mathbb Q$, the intersection
$$\langle(0,\alpha_c),(\beta_c,z)\rangle\cap \langle(0,\alpha),(\beta',z')\rangle$$
is a lattice.
The union of $S_c^z$ and the sets $S_P$ is therefore a periodic tiling of $E$, completing the proof.

\section{Proofs of the main results}\label{sec:app}

\subsection{Periodicity} 

Theorem~\ref{thm:IntroQuasiPTC} follows immediately from the next two results.

\begin{theorem}\label{thm:MainWeakPeriodic}
    Let $T$ be a tiling of $\mathbb{R}^2$ by an axis-parallel polygonal set $\Omega$ that is topologically minimal. Then $T$ is weakly periodic.
\end{theorem}
\begin{proof}
    By Theorem~\ref{thm:IntroMainTectonicPlatesFinLocComp}, we may assume without loss of generality that $T=T_c\sqcup T_1\sqcup\dots\sqcup T_m$, where $T_j$ has finite local complexity, and $\Omega\oplus T_j$ is periodic and $\mathbb{R}(0,1)$-invariant for every $1\le j\le m$.
    It follows that $\Omega\oplus T_c$ is periodic as well.
    The claim follows by Theorem~\ref{thm:MainLocallyFin}~(i) applied to $T_j$ for every $1\le j\le m$, and by the fact that the set $T_c$ is always $\Z(0,\alpha)$-periodic for some $\alpha\in \Q m_\Omega$ by the one-dimensional result \cite{LW}.
\end{proof}

\begin{theorem}\label{thm:MostPeriodic}
    Let $\Omega$ be an axis-parallel polygonal set.
If $\Tile(\Omega,\R^2)$ is nonempty, then there is a tiling $S\in \Tile(\Omega,\R^2)$ that is quasi periodic and $\langle \vec v\rangle$-periodic for some $\vec v\in \mathbb{R}^2\setminus \{(0,0)\}$.
\end{theorem}
\begin{proof}
    After passing to a topologically minimal tiling and applying Theorem~\ref{thm:IntroMainTectonicPlatesFinLocComp}, we may assume that $T=T_c\sqcup T_1\sqcup\dots\sqcup T_m$, where $T_j$ has finite local complexity, $\Omega\oplus T_j$ is periodic and $\mathbb{R}(0,1)$-invariant for every $1\le j\le m$, and $\Omega\oplus T_c$ is periodic and $\mathbb{R}(0,1)$-invariant.
    As $p_1(T_c)$ is periodic, we may additionally assume, after possibly replacing the columns by a horizontal translates of a fixed single column, that $T_c$ has finite local complexity.
    Applying Theorem~\ref{thm:MainLocallyFin}~(ii) separately to $T_c$ and $T_j$ for every $1\le j\le m$, there are periodic $S_c$ and $S_j$ for $1\le j\le m$ such that $\Omega\oplus T_c=\Omega\oplus S_c$ and $\Omega\oplus T_j=\Omega\oplus S_j$ for every $1\le j \le m$.
    Define $S=S_c\sqcup S_1\sqcup\dots \sqcup S_m$, and observe that $S$ is quasi periodic.

    It remains to find $\vec v$ for which is $S$ $\langle \vec v\rangle$-periodic.
    Let $\langle \lambda_c,\gamma_c\rangle$ and $\langle \lambda_j,\gamma_j\rangle$ be the lattices that makes $S_c$ and $S_j$ periodic for every $1\le j\le m$.
    Our goal is to show that the intersection of these lattices is non-trivial.
    This is clearly true when either $T_c=T$ or $T_c=\emptyset$ and $m=1$.
    Hence, we may assume that either $m\ge 2$ and $T_c=\emptyset$, or $m\ge 1$ and $T_c\not=\emptyset$.
    
    We discuss the former case, as the latter case can be proved in the same way, and is, in fact, simpler.
    As each of the sets $\Omega\oplus T_j$ is $\mathbb{R}(0,1)$-invariant and periodic for $1\le j\le m$, where $m\ge 2$, we conclude by the ergodic theorem and after possible passing to sublattices of finite index that $\lambda_j=(0,\alpha_j)$ and $\gamma_j=(\beta,z_j)$ for some $\alpha_j,\beta>0$ and $z_j\ge 0$ for every $1\le j\le m$.

    \begin{claim}\label{cl:CommonPeriod}
        We have $\alpha_j\in \mathbb{Q}\alpha_1$ for every $1\le j\le m$.
    \end{claim}
    \begin{proof}[\emph{Proof of Claim \ref{cl:CommonPeriod}:}]
        Let $\varphi$ be a good approximation of $S$.
        Following the notation after Theorem~\ref{thm:GoodApproximation} and the beginning of the proof of Theorem~\ref{thm:GoodApproximation}, we have that $S^1\in \Tile(\Omega^1;\R^2)$, where we recall that $\Omega^1$ and $S^1$ is the modification of $\Omega$ and $S$ given by applying $\varphi$ only to the first coordinates of $V(\Omega)$ and $S$.
        It can be easily checked that $S^1=S^1_1\sqcup \dots\sqcup S_m^1$, $S^1_j$ is $\langle(0,\alpha_j),(\varphi(\beta),z_j)\rangle$-periodic and $\Omega^1\oplus S^1_j$ is $\mathbb{R}(0,1)$-invariant.
        
        Let $X^1_j\subseteq  \mathbb{R}$ be such that $\Omega^1\oplus S^1_j=X^1_j\times \mathbb{R}$, and observe that $X^1_j$ is $\varphi(\beta)$-periodic, and $m_1(X^1_j\cap [0,\varphi(\beta)])\in \mathbb{Q}$ for every $1\le j\le m$.
        Additionally, there is a positive integer $\ell_j$ such that
        $$\ell_j\cdot m_2(\Omega^1)=m_1(X^1_j\cap [0,\varphi(\beta)])\cdot \alpha_j$$
        for every $1\le j\le m$.
        In particular, we have
        \begin{equation}\label{eq:Ratios}
            \alpha_j= \frac{\ell_j\cdot m_2(\Omega^1)}{m_1(X^1_j\cap [0,\varphi(\beta)])}=\frac{\ell_j\cdot m_1(X^1_1\cap [0,\varphi(\beta)])}{\ell_1\cdot m_1(X^1_j\cap [0,\varphi(\beta)])}\alpha_1\in \mathbb{Q}\alpha_1
        \end{equation}
        for every $1\le j\le m$, which finishes the proof of the claim.
    \end{proof}
    
    By Claim~\ref{cl:CommonPeriod}, we have that $\bigcap_{j=1}^m \langle \alpha_j\rangle\not =\{0\}$.
    Any $\vec v=(0,\alpha)$, where $\alpha\in \bigcap_{j=1}^m \langle \alpha_j\rangle\setminus \{0\}$, then works as required.  
\end{proof}

\subsection{Decidability}\label{subsec:Decidability}

In this section, we prove Theorem\ref{thm:decidability}.
In fact, our periodicity results allow us to deduce that the translational monotiling problem for axis-parallel polygonal sets is decidable relative to a comparison oracle.
More precisely, we consider the following problem.

\begin{tcolorbox}[
    title={Tiling with a comparison oracle},
    fonttitle=\bfseries,
    coltitle=black,
    colbacktitle=black!6,
    colback=white,
    colframe=black!65,
    boxrule=0.5pt,
    sharp corners,
    breakable,
    left=3mm,
    right=3mm,
    top=2mm,
    bottom=2mm,
    before skip=1em,
    after skip=1em
]
\textbf{Input.}
An axis-parallel polygonal set $\Omega$, specified by its
vertex set $V(\Omega)$ and edge set $E(\Omega)$.

\medskip
\textbf{Oracle access.}
Given $i\in\{1,2\}$ and finite sequences
$v_1,\dots,v_m$ and $w_1,\dots,w_n$ of vertices of $\Omega$,
the comparison oracle determines whether
\[
    \sum_{j=1}^{m} p_i(v_j)
    \le
    \sum_{k=1}^{n} p_i(w_k).
\]
Repetitions are allowed, and empty sums are interpreted as zero.

\medskip
\textbf{Required output.}
\textsc{Yes} if $\Omega$ tiles $\mathbb{R}^2$ by translations,
and \textsc{No} otherwise.

\medskip
The algorithm must halt on every input after finitely many
computational steps and oracle queries.
\end{tcolorbox}
\vspace{+0.1cm}

\begin{theorem}\label{thm:DecidabilityComparison}
    The problem whether a given axis-parallel polygonal set $\Omega$ tiles $\mathbb{R}^2$ by translations is decidable relative to the comparison oracle described above.
\end{theorem}

Since the comparison oracle is computable for real algebraic coordinates given by exact representations,
Theorem~\ref{thm:DecidabilityComparison} has the following immediate corollary.

\begin{tcolorbox}[
    title={Tiling with algebraic vertex coordinates},
    fonttitle=\bfseries,
    coltitle=black,
    colbacktitle=black!6,
    colback=white,
    colframe=black!65,
    boxrule=0.5pt,
    sharp corners,
    breakable,
    left=3mm,
    right=3mm,
    top=2mm,
    bottom=2mm,
    before skip=1em,
    after skip=1em
]
\textbf{Exact representation.}
A real algebraic number $\alpha$ is represented by a nonzero
polynomial $f\in\mathbb{Z}[x]$ and rational numbers $a<b$
such that $\alpha$ is the unique real root of $f$ in $(a,b)$.

\medskip
\textbf{Input.}
An axis-parallel polygonal set $\Omega$ with real algebraic
vertex coordinates, specified by its vertex set $V(\Omega)$
and edge set $E(\Omega)$.
Each vertex coordinate is given by an exact representation
as above.

\medskip
\textbf{Required output.}
\textsc{Yes} if $\Omega$ tiles $\mathbb{R}^2$ by translations,
and \textsc{No} otherwise.

\medskip
The algorithm must halt on every input after finitely many
computational steps.
\end{tcolorbox}
\vspace{+0.1cm}

As a corollary, we obtain Theorem \ref{thm:decidability} that we restate here for readers convenience.

\mainDecidabilitytheorem*

\begin{proof}
    For real algebraic coordinates given by defining polynomials and rational isolating intervals, exact arithmetic and order comparisons are computable; see \cite[Chapter~10]{BPR}.
    Hence, the claim follows from Theorem~\ref{thm:DecidabilityComparison}.
\end{proof}

\begin{remark}
    LLMs pointed out that the comparison oracle is computable for real algebraic coordinates given by exact representations.
\end{remark}

In the remainder of this section, we prove Theorem~\ref{thm:DecidabilityComparison} by constructing oracle-verifiable finite certificates for both tileability and non-tileability, together with effective searches for these certificates.
We start with a couple of definitions and observations.

Let $\Omega$ be an axis-parallel polygonal set.
We may assume that $(0,0)\in V(\Omega)$, as any query for the comparison oracle for $\Omega-v_0$ for any fixed $v_0\in V(\Omega)$ can be simulated by the comparison oracle for $\Omega$.

For each $k\in\mathbb{N}$, define
$$k\star V(\Omega)=\underbrace{(V(\Omega)-V(\Omega))+\cdots+(V(\Omega)-V(\Omega))}_{k\text{ summands}},$$
with the convention that $0\star V(\Omega)=\{0\}$.
Let $A_i=\min\{p_i(w-v):v,w\in V(\Omega) \ p_i(v)<p_i(w)\}$ and $B_i=\max\{p_i(w-v):v,w\in V(\Omega) \ p_i(v)<p_i(w)\}$ for $i\in \{1,2\}$.
Note that $A_i>0$ and $B_i>0$ can be determined by querying the comparison oracle.

\subsubsection{\bf Certificate for non-tileability.}
For each $N\in \mathbb{N}$, we define $D_N=[0,NA_1]\times [0,NA_2]$.

\begin{proposition}\label{pr:CertificateNonTile}
    Suppose that $\Tile(\Omega,\R^2)\not=\emptyset$.
    Then for every $N\in \N$ such that $B_i\le NA_i$ for $i\in \{1,2\}$ there is
    $$T_N\subseteq (16N)^2\star V(\Omega)$$
    such that $D_N\subseteq \overline{\Omega}\oplus T_N$. 
\end{proposition}
\begin{proof}
    Let $T$ be a tiling of $\R^2$ by $\Omega$ such that $(0,0)\in T$ and $N\in \N$ be such that $B_i\le NA_i$ for $i\in \{1,2\}$.

    Define $W=D_N+[-B_1,B_1]\times[-B_2,B_2]$ and $U=W+[-B_1,B_1]\times[-B_2,B_2]$, and let
    $$T'=\{t\in T:(\Omega+t)\cap U\not=\emptyset\}.$$
    Observe that $U\subseteq\overline{\Omega}\oplus T'$ and that each $\Omega+t$, for $t\in T'$, contributes by volume at least $A_1A_2$ to $U+[-A_1,A_1]\times[-A_2,A_2]$.
    Hence, as
    $$U+[-A_1,A_1]\times[-A_2,A_2]\subseteq [-4NA_1,4NA_1]\times [-4NA_2,4NA_2],$$
    we conclude that $|T'|\le 256N^2$.
    
    Define $\simeq_{T'}$ to be the equivalence relation on $T'$ generated by the pairs $t,s\in T'$ such that
    $$(0,0)\in V(\Omega)-V(\Omega)+t-s.$$
    Write $T'=T'_1\sqcup\dots\sqcup T'_m$ for the decomposition of $T'$ into $\simeq_{T'}$-equivalence classes.
    Observe that, after possibly swapping the vertical and horizontal axes, for each $1\le i\le m$ we have that $\partial(\overline{\Omega\oplus T'_i})\cap \interior{U}$ consists of vertical line segments.    
    After possibly permuting the indices, we may additionally assume that
    \begin{equation}\label{eq:SharingBoundary}
        (\partial(\overline{\Omega\oplus T'_{i+1}})\cap \interior{U})\cap \left(\partial\left(\overline{\Omega\oplus \bigsqcup_{j=1}^i T'_j}\right)\cap \interior{U}\right)\not=\emptyset \text{ for every } 1\le i\le m-1,
    \end{equation}
    and that $(0,0)\in T'_1$.
    
    By the definition of $B_2$, of $U$ and $W$, and by the structure of the boundary of $\overline{\Omega\oplus T'_i}$ within $U$, we have that for any sequence $\alpha=(\alpha_i)_{i=1}^m$ such that $0\le \alpha_i\le B_2$ we have 
    $$\bigl((\Omega+T'_i)+(0,\alpha_i)\bigr)\cap W=(\Omega+T'_i)\cap W\text{ up to null sets}.$$
    Hence setting
    $$S_{\alpha}=\bigcup_{i=1}^m\left(T'_i+(0,\alpha_i)\right),$$
    that need not be a packing in general, and
    $$T_{\alpha}=\{s\in S_\alpha:(\Omega+s)\cap D_N\not=\emptyset\},$$
    we have $D_N\subseteq \overline{\Omega}\oplus T_{\alpha}$.
    In particular, $\Omega\oplus T_\alpha$ is a packing.

    Recursively, starting with $\alpha_1=0$ and $S_1=T'$, define $\alpha_i$ and $S_i$ such that $0\le \alpha_i\le B_2$, 
    $$S_{i}=\bigcup_{j=1}^{i} (T'_j+(0,\alpha_j))\sqcup \bigcup_{j=i+1}^m T'_j$$
    such that $\bigsqcup_{j=1}^{i} (T'_j+(0,\alpha_j))$ is contained in a single $\simeq_{S_{i}}$-equivalence class.
    This is possible by \eqref{eq:SharingBoundary} and our choice of $B_2$.
    Set $T_N=\{s\in S_m:(\Omega+s)\cap D_N\not=\emptyset\}$.
    As $B_2\le NA_2$, we have that $D_N\subseteq \overline{\Omega}\oplus T_N$.
    Observe that $(0,0)\in S_m$ as $(0,0)\in T'_1$.
    
    It remains to show that $T_N\subseteq (16N)^2\star V(\Omega)$.
    Since $\alpha_1=0$ and $(0,0)\in T'_1$, we have $(0,0)\in S_m$.
    Moreover, the $\simeq_{S_m}$-equivalence relation consists of a single equivalence class $S_m$.
    Thus every $s\in S_m$ can be joined to $(0,0)$ by a path of at most $|S_m|-1$ elements of the form $V(\Omega)-V(\Omega)$ by the definition of $\simeq_{S_m}$.
    Since
    $$|S_m|\le |T'|\le 256N^2,$$
    we obtain
    $$T_N\subseteq S_m\subseteq (16N)^2\star V(\Omega),$$
    as required.
\end{proof}

\begin{tcolorbox}[
    title={Certificate for non-tileability},
    fonttitle=\bfseries,
    coltitle=black,
    colbacktitle=black!6,
    colback=white,
    colframe=black!65,
    boxrule=0.5pt,
    sharp corners,
    breakable,
    left=3mm,
    right=3mm,
    top=2mm,
    bottom=2mm,
    before skip=1em,
    after skip=1em
]
\textbf{Search.}
Consider, in increasing order, all positive integers $N$ such
that $B_i\le NA_i$ for $i\in\{1,2\}$, and set $D_N=[0,NA_1]\times[0,NA_2]$.

\medskip
\textbf{Finite test.}
Using the comparison oracle, check whether there is a subset
$S\subseteq (16N)^2\star V(\Omega)$ such that $\Omega+S$ is a packing and
$$D_N\subseteq\overline{\Omega}+S.$$
Since $(16N)^2\star V(\Omega)$ is finite, this test requires only finitely many
computational steps and oracle queries.

\medskip
\textbf{Output.}
If no such subset exists, output \textsc{No} and halt.
Otherwise, proceed to the next $N$.

\medskip
\textbf{Guarantee.}
By Proposition~\ref{pr:CertificateNonTile}, the algorithm
never outputs \textsc{No} when $\Omega$ tiles $\mathbb{R}^2$.
Conversely, compactness guarantees that, if $\Omega$ does
not tile $\mathbb{R}^2$, some $D_N$ cannot be covered by
any packing of translates of $\Omega$.
Thus the algorithm halts exactly when $\Omega$ does not tile.
\end{tcolorbox}

\subsubsection{\bf Certificate for tileability.}
Assume that $\Tile(\Omega,\R^2)\not=\emptyset$.
Recall that, given a packing $\Omega\oplus T$, we denote by $\simeq_T$ the equivalence relation on $T$ generated by the pairs
$t,s\in T$ such that $(0,0)\in V(\Omega)-V(\Omega)+t-s$.
Observe that if the earthquake decomposition of $T\in\Tile(\Omega,\R^2)$ with respect to $\Omega$ is non-trivial
and $P\in\quake{T}{\Omega}$, then $P$ is contained in a $\simeq_T$-equivalence class.
This follows from the minimality of tectonic plates, since the region covered by each $\simeq_T$-equivalence class is either $\R(0,1)$-invariant or $\R(1,0)$-invariant.

To describe the algorithm for certifying tileability, we distinguish two cases and run the corresponding searches in
parallel.
The first case corresponds to full periodicity.
It is not hard to check that if the PTC holds for this class of tiles, then the first case would suffice.

We call $\Omega$ {\bf \emph{nicely periodic}} if there is $T\in\Tile(\Omega,\R^2)$ such that $\simeq_T$ consists of a single equivalence class and $T$ is $\Gamma=\langle a,b\rangle$-periodic for some linearly independent vectors $a,b\in\R^2$.

For such a tiling, we may assume that $(0,0)\in T$, and after passing to a finite-index sublattice, we may choose the generators $(a_1,a_2)$ and $(b_1,b_2)$ satisfying
\begin{equation}\label{eq:CertificateBasis}
    a_1,b_2>0,\qquad 4|b_1|\le a_1,\qquad 4|a_2|\le b_2,\qquad \text{and}\qquad\overline{\Omega}\subseteq [-a_1,a_1]\times[-b_2,b_2].
\end{equation}
Set
$$W_{a,b}:=[-3a_1,3a_1]\times[-3b_2,3b_2].$$
Observe that the fundamental parallelogram of $\langle a,b\rangle$ is contained in $W_{a,b}$ by \eqref{eq:CertificateBasis}.
Consequently, there is a finite set
$T'\subseteq T\cap W_{a,b}$ such that
$$T=T'+\langle a,b\rangle.$$
Since $\simeq_T$ consists of a single equivalence class, there is $k\in\mathbb N$ such that
$$a,b\in k\star V(\Omega), \qquad T'\subseteq k\star V(\Omega).$$

\begin{tcolorbox}[
    title={Certificate for tileability: nicely periodic case},
    fonttitle=\bfseries,
    coltitle=black,
    colbacktitle=black!6,
    colback=white,
    colframe=black!65,
    boxrule=0.5pt,
    sharp corners,
    breakable,
    left=3mm,
    right=3mm,
    top=2mm,
    bottom=2mm,
    before skip=1em,
    after skip=1em
]
\textbf{Search.}
For $k\ge 1$, enumerate $a,b\in k\star V(\Omega)$ satisfying
\eqref{eq:CertificateBasis} and all nonempty $Q\subseteq k\star V(\Omega)\cap W_{a,b}$.

\medskip
\textbf{Finite test.}
Set
\[
    F=Q+\{ra+tb:r,t\in\mathbb Z,\ |r|,|t|\le20\}.
\]
Using the comparison oracle, check whether $\Omega+F$
is a packing and $W_{a,b}\subseteq\overline{\Omega}+F$.

\medskip
\textbf{Output.}
If both conditions hold, output \textsc{Yes} and halt.
Otherwise, continue the search.

\medskip
\textbf{Guarantee.}
If the test passes, set $S=Q+\langle a,b\rangle$.
\begin{itemize}
    \item  Since $W_{a,b}$ contains the fundamental parallelogram and $W_{a,b}\subseteq\overline{\Omega}+F$, we have that $\Omega+S=\R^2$ up to a null set.
    \item If $m_2((\Omega+s)\cap (\Omega+s'))>0$ for some $s,s'\in S$, then we may assume after translating by a lattice vector that $s=q$ and $s'=q'+ra+tb$, where $q,q'\in Q$ and $r,t\in\mathbb Z$.
    Thus $ra+tb\in q-q'+\Omega-\Omega$, and by the conditions above, we get that
    $$|ra_1+tb_1|\le8a_1, \qquad |ra_2+tb_2|\le8b_2.$$
    By \eqref{eq:CertificateBasis},
    $$|r|\le8+\frac{|t|}{4}, \qquad |t|\le8+\frac{|r|}{4},$$
    whence $\max\{|r|,|t|\}\le32/3<20$.
    Consequently, $s=s'$ by the finite test above.
\end{itemize}
\end{tcolorbox}

Note that if there is $T\in\Tile(\Omega;\mathbb R^2)$ such that $T=T_c$, then $\Omega$ is nicely periodic.
Indeed, $T=T_c$ implies that $p_1(T)$ is a tiling of $\R$ by $p_1(\Omega)$.
By \cite{LW}, $p_1(T)$ is $\beta\Z$-periodic for some $\beta>0$.
After replacing $\beta$ by a multiple and translating $\Omega$ and $T$, we may assume that
$\beta\ge 2\diam(p_1(\Omega))$, $0\le p_1(V(\Omega))$, and $0\in p_1(T)$.

Fix an increasing enumeration $x_1,\dots,x_k$ of $p_1(T)\cap[0,\beta)$ and a vertical column
$C\subseteq\{0\}\times\R$ for $\Omega$.
By \cite{LW}, $C$ is $(0,\alpha)\Z$-periodic for some $\alpha>0$.
Set $y_1=0$.
Observe that there are $y_2,\dots,y_k\in\R$ such that, setting
$$S'=\bigsqcup_{i=1}^k \bigl(C+(x_i,y_i)\bigr),$$
the relation $\simeq_{S'}$ consists of a single equivalence class.

Finally, by the choice of $\beta$, there is $\gamma\in\mathbb R$ such that
$$(V(\Omega)+S')\cap(V(\Omega)+S'+(\beta,\gamma))\not=\emptyset.$$
Then
$$S=\bigsqcup_{n\in\mathbb Z}\bigl(S'+n(\beta,\gamma)\bigr)$$
is a tiling of $\mathbb R^2$ by $\Omega$, is $\langle(0,\alpha),(\beta,\gamma)\rangle$-periodic, and $\simeq_S$ consists of a single equivalence class.

Next, we describe the algorithm for the case when $\Omega$ is not nicely periodic.
We start with a strengthening of Theorem~\ref{thm:MostPeriodic}.

\begin{proposition}\label{pr:CertificateTile}
    Suppose that $\Tile(\Omega,\R^2)\not=\emptyset$ and $\Omega$ is not nicely periodic.
    Then, after possibly swapping vertical and horizontal axes, there is $U\in \Tile(\Omega,\R^2)$ and $k\in \N$ such that $U$ can be written as $U=U_1\sqcup  \dots \sqcup U_m$ for some $m\in \N$ together with $\alpha,\beta>0$ and $\gamma_i\ge0$ for $1\le i\le m$ such that
    \begin{enumerate}
        \item $(0,\alpha),(\beta,\gamma_i)\in k\star V(\Omega)$ for every $1\le i\le m$,
        \item if we set $\Gamma_i=\langle (0,\alpha),(\beta,\gamma_i)\rangle$, then $U_i$ is $\Gamma_i$-periodic and $\Omega\oplus U_i$ is $\mathbb{R}(0,1)$-invariant for every $1\le i\le m$,
        \item we have
            $$\{u\in U:(\overline{\Omega}+u)\cap [0,\beta]\times[0,\alpha]\not=\emptyset\}\subseteq k\star V(\Omega).$$
    \end{enumerate}
\end{proposition}
\begin{proof}
    Let $S\in \Tile(\Omega,\R^2)$ be as in the proof of Theorem~\ref{thm:MostPeriodic}.
    That is, after possibly swapping the vertical and horizontal axes, we have a decomposition $S=S_c\sqcup S_1\sqcup\dots\sqcup S_m$ for some $m\in\N$, where the fact that $m\ge 1$ follows from the assumption that $\Omega$ is not nicely periodic together with the observation before this proposition.
    Additionally, there are lattices $\Gamma_c$ and $\Gamma_i$ for $1\le i\le m$ such that $S_c$ is $\Gamma_c$-periodic and $S_i$ is $\Gamma_i$-periodic for every $1\le i\le m$, and $\Omega\oplus S_c$ and $\Omega\oplus S_i$ are $\mathbb{R}(0,1)$-invariant for every $1\le i\le m$.

    Suppose that $m=1$ and $S_c=\emptyset$, so that $S=S_1$ is periodic.
    As $\Omega$ is not nicely periodic, it must be the case that $\simeq_S$ consists of more than one equivalence class.
    After possibly swapping the axes, the earthquake decomposition of $S$ is therefore non-trivial.
    The case, when after this swapping we have $S_c\not=\emptyset$ is handled in the next paragraph.
    Thus $S_c=\emptyset$ implies that, after possibly swapping the vertical and horizontal axes, $|\quake{S}{\Omega}|>1$.
    Let $\Gamma_1=\langle a,b\rangle$ be such that $S=S_1$ is $\Gamma_1$-periodic.
    By Corollary~\ref{cor:FinNumberPlates}, the period lattice permutes finitely many noncolumn plates.
    Hence, after passing to a finite-index sublattice of $\Gamma_1$, every $P\in\quake{S}{\Omega}$ is $\Gamma_1$-periodic.
    Thus, by shifting the pieces independently by a vector $(0,\alpha_P)$ for some $\alpha_P\ge 0$ we can get that $\simeq_{S'}$ consists of a single equivalence class, where
    $$S'=\bigsqcup_{P\in \quake{S}{\Omega}} P+(0,\alpha_P).$$
    Note that $S'$ is $\Gamma_1$-periodic as so is $P$ for every $P\in \quake{S}{\Omega}$.
    This contradicts the assumption that $\Omega$ is not nicely periodic.

    Next we suppose that $m\ge 2$ or $S_c\not=\emptyset$.
    As in the proof of Theorem~\ref{thm:MostPeriodic}, we deduce that there are $\alpha,\beta>0$ and $\gamma_i\ge 0$ such that $\Gamma_c=\langle(0,\alpha),(\beta,0)\rangle$ and $\Gamma_i=\langle(0,\alpha),(\beta,\gamma_i)\rangle$.
    By periodicity of any of the tectonic plates, we conclude that $S$ has finite horizontal complexity.
    Indeed, it is easy to check that Lemma~\ref{lem:CharInfComplex}~(4) is not satisfied.
    By Theorem~\ref{thm:PeriodicityMain}, after possibly enlarging $\beta$ by a positive integer multiple and refining and relabelling the pieces, we may assume without loss of generality that $\{S_1,\dots,S_m\}=\quake{S}{\Omega}$, and that after possibly shifting each $S_i$ vertically, that $S_i$ is equal to a $\simeq_S$-equivalence class for every $1\le i\le m$.

    Assume without loss of generality that $\beta>2\diam(p_1(\Omega))$ and $0\in p_1(S_c)$ in case $S_c\not=\emptyset$.
    Enumerate as $\{x_1,\dots,x_\ell\}$ the set $p_1(S_c)\cap [0,\beta)$, and define $S^j_c=\{s\in S_c:p_1(s)\in x_j+\beta\Z\}$.
    As a final step, observe that if there is $1\le i\le m$ and $1\le j\le \ell$ such that $\overline{\Omega\oplus S_i}$ and $\overline{\Omega\oplus S^j_c}$ share boundary, then we can shift the columns in $S^j_c$ to produce a packing $S^j$ that is $\Gamma_i$-periodic, $\Omega\oplus S^j=\Omega\oplus S^j_c$, and $\simeq_{S_i\cup S^j}$ consists of a single equivalence class.
    Using this observation and recursively merging packings together, we produce $U\in \Tile(\Omega,\R^2)$ that admits a decomposition $U=U_1\sqcup\dots\sqcup U_m$ such that $S_i\subseteq U_i$, $U_i$ is $\Gamma_i$-periodic, $\Omega\oplus U_i$ is $\R(0,1)$-invariant and $\simeq_{U_i}$ consists of a single equivalence class for every $1\le i\le m$.
    The claim follows by vertically shifting $U_i$'s to guarantee that $\simeq_U$ consists of a single equivalence class, and then translating $U$ such that $(0,0)\in U$.
\end{proof}

\begin{tcolorbox}[
    title={Certificate for tileability: quasi-periodic case},
    fonttitle=\bfseries,
    coltitle=black,
    colbacktitle=black!6,
    colback=white,
    colframe=black!65,
    boxrule=0.5pt,
    sharp corners,
    breakable,
    left=3mm,
    right=3mm,
    top=2mm,
    bottom=2mm,
    before skip=1em,
    after skip=1em
]
\textbf{Search.}
Run the following search in both coordinate orientations.
For $k=1,2,\dots$, enumerate $1\le m\le k$ and vectors
$$(0,\alpha),(\beta,\gamma_1),\dots,(\beta,\gamma_m)\in k\star V(\Omega)$$
satisfying
\begin{equation}\label{eq:BoundedBox}
    \alpha,\beta>0,\qquad 0\le\gamma_i<\alpha,\qquad \overline{\Omega}\subseteq [-\beta,\beta]\times[-\alpha,\alpha],
\end{equation}
together with pairwise disjoint nonempty sets
$$Q_1,\dots,Q_m\subseteq k\star V(\Omega)\cap D,\qquad D=[0,\beta)\times[0,\alpha).$$

\medskip
\textbf{Finite test.}
For $1\le i\le m$, set
$$F_i=Q_i+\{r(\beta,\gamma_i)+t(0,\alpha):r,t\in\mathbb Z,\ |r|,|t|\le20\}$$
and
$$Y_i=p_1\bigl((\Omega+F_i)\cap\interior(D)\bigr).$$
Using the comparison oracle, check that:
\begin{enumerate}
    \item $\Omega+F_i$ is a packing for every $i$;
    \item for every $i$, up to a null set,
    $$(\Omega+F_i)\cap\interior(D)=Y_i\times(0,\alpha);$$
    \item up to a null set,
    $$(0,\beta)=Y_1\sqcup\dots\sqcup Y_m.$$
\end{enumerate}
These are finite checks involving only comparisons of
tile-edge coordinates.

\medskip
\textbf{Output.}
If all conditions hold, output \textsc{Yes} and halt.
Otherwise, continue the search.

\medskip
\textbf{Guarantee.}
If the test passes, set
$$S_i=Q_i+\langle(0,\alpha),(\beta,\gamma_i)\rangle,\qquad S=\bigcup_{i=1}^m S_i.$$
\begin{itemize}
    \item Every tile of $\Omega+S_i$ meeting $D$ has lattice
    indices $|r|\le2$ and $|t|\le4$, by \eqref{eq:BoundedBox}.
    Hence
    $$(\Omega+S_i)\cap\interior(D)=(\Omega+F_i)\cap\interior(D).$$
    By periodicity and the second test,
    $$\Omega+S_i=(Y_i+\beta\mathbb Z)\times\mathbb R \quad\text{up to a null set}.$$
    The third test therefore implies that $\Omega+S$
    covers $\mathbb R^2$ up to a null set.

    \item \item If $m_2((\Omega+s)\cap(\Omega+s'))>0$ for some $s,s'\in S$, then the displayed identity and the third test imply that $s,s'\in S_i$ for some $1\le i\le m$.
    After translating both by the same element of $\langle(0,\alpha),(\beta,\gamma_i)\rangle$, we may assume that
    $$s=q,\qquad s'=q'+r(\beta,\gamma_i)+t(0,\alpha),$$
    where $q,q'\in Q_i$ and $r,t\in\mathbb Z$.
    Thus,
    $$r(\beta,\gamma_i)+t(0,\alpha) \in q-q'+\Omega-\Omega.$$
    By \eqref{eq:BoundedBox}, we have
    $$|r|\le3,\qquad |r\gamma_i+t\alpha|\le3\alpha,$$
    hence $|t|\le6<20$.
    Therefore $s,s'\in F_i$, and the first test implies that $s=s'$.
\end{itemize}
By Proposition~\ref{pr:CertificateTile}, after enlarging the periods and $k$ if necessary, this search, run in
both coordinate orientations, halts whenever $\Omega$ is tileable and not nicely periodic.
\end{tcolorbox}

Running the non-tileability search and the two tileability
searches in parallel proves
Theorem~\ref{thm:DecidabilityComparison}.

\end{document}